\documentclass[letterpaper, 11pt, reqno]{amsart}
\title{\vspace{-10pt}KS-Groupoids and KS-Crossed Products\\ associated with Inverse Semigroup Actions}
\author{Hikaru Sekiyama}
\thanks{Department of Mathematics, Faculty of Science and Technology, Keio University, 3-14-1 Hiyoshi, Kohoku-ku, Yokohama, 223-8522, Japan.
Email: \texttt{hika-sekky777@keio.jp}}
\makeatletter
\def\@oddhead{\normalfont\thepage\hfil}
\def\@evenhead{\normalfont\hfil\thepage}
\makeatother

\usepackage[top=2.5cm,bottom=2.5cm,right=3cm,left=3cm]{geometry}

\usepackage{amsmath, amsthm, amssymb, mathrsfs}
\usepackage{framed}
\usepackage{tikz}
\usetikzlibrary{arrows}
\usepackage{xcolor}
\usepackage{hyperref}
\hypersetup{colorlinks=true, citecolor=blue, linkcolor=red, urlcolor=red, linktoc=page}
\usepackage{indentfirst} %文頭にインデントを入れる

\renewcommand{\eqref}[1]{\hyperref[#1]{\textup{(\ref*{#1})}}} %括弧もhyperlink

\numberwithin{equation}{section} %式のtagにセクション番号入れる

\makeatletter
\renewenvironment{proof}[1][\proofname]{\par
  \pushQED{\qed}%
  \normalfont
  \topsep=2pt \partopsep=0pt
  \trivlist
  \item[\hskip\labelsep\bfseries #1\@addpunct{:}]\ignorespaces
}{%
  \popQED\endtrivlist\@endpefalse
}
\makeatother
\renewcommand{\proofname}{Proof}

\makeatletter
\newcommand*{\defeq}{\mathrel{\rlap{%
                     \raisebox{0.3ex}{$\m@th\cdot$}}%
                     \raisebox{-0.3ex}{$\m@th\cdot$}}%
                     =}
\makeatother

\makeatletter
\def\l@section{\@tocline{1}{0pt}{1pc}{}{}}
\def\l@subsection{\@tocline{2}{0pt}{1pc}{4.6em}{}}
\renewcommand{\tocsubsection}[3]{%
  \indentlabel{%
    \@ifnotempty{#2}{%
      \hspace*{1.3em}%
      \makebox[2.3em][l]{\ignorespaces#1 #2.\hfill}%
    }%
  }#3%
}
\makeatother

\newtheoremstyle{mythm}{3pt}{3pt}{\normalfont}{}{}{}{ }{{\textbf{\thmname{#1} \thmnumber{#2}.} \thmnote{(#3)}}}
\theoremstyle{mythm}

\newtheorem{thm}{Theorem}[subsection]
\newtheorem{dfn}[thm]{Definition}
\newtheorem{lem}[thm]{Lemma}
\newtheorem{prop}[thm]{Proposition}
\newtheorem{ex}[thm]{Example}
\newtheorem{cor}[thm]{Corollary}
\newtheorem{rem}[thm]{Remark}

\newtheorem{obs}[thm]{Observation}

\newtheorem*{thm*}{Theorem}

\makeatletter
\@addtoreset{thm}{subsection}
\makeatother
\renewcommand{\thethm}{%
  \ifnum\value{subsection}=0
    \thesection.\arabic{thm}%
  \else
    \thesubsection.\arabic{thm}%
  \fi
}

\let\oldsubsection\subsection
\renewcommand{\subsection}{\vspace{-3pt}\oldsubsection}

\let\oldenumerate\enumerate
\renewcommand{\enumerate}{
   \oldenumerate
   \setlength{\itemsep}{2pt}
   \setlength{\parskip}{0pt}
   \setlength{\parsep}{0pt}
   \setlength{\leftskip}{-18pt}
   }

\DeclareMathOperator{\pt}{\mathrm{\textbf{pt}}}
\DeclareMathOperator{\ev}{\mathrm{ev}}

\DeclareMathOperator{\bis}{\mathrm{Bis}}
\DeclareMathOperator{\eval}{\mathrm{ev}}

\DeclareMathOperator{\Sp}{Sp}
\DeclareMathOperator{\ind}{\mathrm{Ind}}
\DeclareMathOperator{\supp}{\mathrm{supp}}
\newcommand{\kstimes}{\ltimes^\mathrm{KS}}
\newcommand{\ksr}{\ltimes^\mathrm{KS}_r}
\newcommand{\ks}{\mathrm{KS}}
\newcommand{\ltimesr}{\ltimes_r}
\DeclareMathOperator{\Span}{\mathrm{span}}
\newcommand{\alg}{\ltimes^{\mathrm{KS}}_{\mathrm{alg}}}

\DeclareMathOperator{\isa}{\mathbf{ISA}_\mathbf{pr}}

\def\d{\textbf{d}}

\newcommand{\ctext}[1]{\raise0.2ex\hbox{\textcircled{\scriptsize{#1}}}}

\begin{document}

\begin{abstract}
  For an action of an inverse semigroup $S$ on a locally compact Hausdorff space $X$, we construct an \'etale groupoid $S\ltimes X^\ks$, called the KS-groupoid associated with the action. 
  This construction may be viewed as a dynamical analogue of the construction of universal groupoids for inverse semigroups. 
  We provide a dynamical version of Paterson's theorem which states that the full and reduced crossed products (in the sense of Khoshkam and Skandalis) of $C_0(X)$ are canonically isomorphic to the full and reduced groupoid C*-algebras of $S\ltimes X^\ks$, respectively. 
  We further investigate some topological properties, functoriality and universality of KS-groupoids.
\end{abstract}

\maketitle

\setcounter{section}{-1}
\tableofcontents

\section{Introduction}

The relationship among inverse semigroups, \'etale groupoids, and C*-algebras has been studied extensively. 
This line of research has its origins in two constructions of C*-algebras from inverse semigroups and \'etale groupoids, namely, inverse semigroup C*-algebras and groupoid C*-algebras.
In his monograph \cite{renault}, Renault introduced a general construction of full and reduced groupoid C*-algebras $C^*(G), C^*_r(G)$ associated with locally compact groupoids $G$ equipped with Haar systems and showed that several important classes of C*-algebras, including AF-algebras and Cuntz algebras, can be realized as the C*-algebras of \'etale groupoids.
On the other hand, Duncan and Paterson initiated the systematic study of inverse semigroup C*-algebras in \cite{dp}.

One of the major developments in the study of the relationship among inverse semigroups, \'etale groupoids, and C*-algebras is Paterson's theorem (\cite{paterson,ks1}). 
For every inverse semigroup $S$, Paterson constructed an \'etale groupoid $S\ltimes\widehat{E(S)}$, called the \textit{universal groupoid} associated with $S$, and proved\footnote{In Paterson's original proof on the reduced case of the isomorphisms, there is a small gap arising from the fact that the universal groupoid need not be Hausdorff in general. This gap was fixed by Khoshkam and Skandalis in \cite{ks1}.} that the full (resp. reduced) inverse semigroup C*-algebra is canonically isomorphic to the full (resp.\ reduced) C*-algebra of this groupoid.
These isomorphisms allow us to investigate inverse semigroup C*-algebras through topological properties of their universal groupoids.

The theory extends naturally from inverse semigroups themselves to their actions.
The notion of an action of an inverse semigroup on a C*-algebra was introduced by Sieben in \cite{ns} as a natural generalization of partial group actions on C*-algebras. 
Several crossed-product constructions have subsequently been associated with such actions.
These include the full crossed products $S\ltimes A$ introduced by Sieben, the reduced crossed products $S\ltimesr A$ introduced by Exel in \cite{exel2011}, and the full and reduced crossed products $S\kstimes A$ and $S\ksr A$ in the sense of Khoshkam and Skandalis in \cite{ks2}.
Though Khoshkam and Skandalis refer to the latter constructions simply as full and reduced crossed products, in order to distinguish them from full and reduced crossed products due to Sieben and Exel, we call\footnote{Related terminology can be found in \cite{bn}.} them the full and reduced \textit{KS-crossed products}, respectively.
These constructions differ even for the identity action of $S$ on $\mathbb C$.
Indeed, full and reduced KS-crossed products are, by construction, nothing but the full and reduced inverse semigroup C*-algebras associated with $S$; on the other hand, full and reduced crossed products are full and reduced group C*-algebras associated with the maximal group image of $S$, respectively.

For an action of an inverse semigroup $S$ on a locally compact Hausdorff space $X$, the full and reduced crossed products of $C_0(X)$ are known to admit a natural groupoid model (see \cite{exel2008,bm,BHM2018,Neshveyev2026} for the full case and \cite{be} for the reduced case). 
Namely, every action of an inverse semigroup $S$ on a locally compact Hausdorff space $X$ gives rise to the \'etale groupoid $S\ltimes X$, called the \textit{transformation groupoid}, and there are canonical isomorphisms
\begin{equation}
  \label{eq:0-1}
  S\ltimes C_0(X)\cong C^*(S\ltimes X),\qquad S\ltimesr C_0(X)\cong C_r^*(S\ltimes X).
\end{equation}
Since every \'etale groupoid $G$ can be realized as the transformation groupoid associated with the canonical action of $\bis(G)$ on the unit space $G^{(0)}$ (\cite[Proposition 5.4]{exel2008}), full and reduced \'etale groupoid C*-algebras are canonically isomorphic to some full and reduced crossed products, respectively, by the isomorphisms above. 

By contrast, it seems to the author that KS-crossed products have not been explored enough in the literature. 
One notable exception is the work of B\'edos and Norling in \cite{bn}, where KS-crossed products are studied in the framework of Fell bundles.
In this paper, we focus on inverse semigroup actions on locally compact Hausdorff spaces and construct an explicit groupoid model for the full and reduced KS-crossed products.
This groupoid model allows us to study C*-algebraic properties of KS-crossed products in terms of the underlying actions.

More precisely, the following are the main contributions of this paper.
{
  \setlength{\leftmargini}{27pt}  
  \begin{itemize}
    \item[(A)] For a given action of an inverse semigroup $S$ on a locally compact Hausdorff space $X$, we introduce a locally compact Hausdorff space $X^\ks$ (Definition \ref{def:of-X^ks}) containing $X$ as a closed subspace. 
    We also construct an $S$-action on $X^\ks$ (Theorem \ref{thm:action-on-X^ks}) that extends the original action on $X$. 
    When the given action is the identity action on the one-point space $\pt$, the space $\pt^\ks$ and the $S$-action on $\pt^\ks$ can be canonically identified with the character space $\widehat{E(S)}$ and the $S$-action on $\widehat{E(S)}$, respectively.
    \item[(B)] It is shown that the corresponding transformation groupoid $S\ltimes X^\ks$ (which we call the \textit{KS-groupoid}) is a groupoid model for the full and reduced KS-crossed products (Subsection \ref{subsec:4.2}), that is, there are canonical isomorphisms
    \begin{equation}
      \label{eq:0-2}
      S\kstimes C_0(X)\cong C^*(S\ltimes X^\ks),\qquad S\ksr C_0(X)\cong C_r^*(S\ltimes X^\ks).
    \end{equation}
    We refer to this result as \textit{dynamical Paterson's theorem} because this is exactly Paterson's theorem when $X=\pt$ and the $S$-action is the identity action.
    \item[(C)] We characterize some topological properties of the KS-groupoid $S\ltimes X^\ks$, such as Hausdorffness (under an assumption), second countability, ampleness of the groupoid and ($\sigma$-)compactness of its unit space, in terms of the associated inverse semigroup action (Subsection \ref{subsec:3.1}).
    These enable us to characterize some elementary properties of KS-crossed products, such as separability and ($\sigma$-)unitality, in terms of the original actions.
    \item[(D)] We show in Subsection \ref{subsec:3.2} that the construction of KS-groupoids defines a functor from $\isa$ to $\mathbf{EG}$, where $\isa$ and $\mathbf{EG}$ denote the categories of inverse semigroup actions on locally compact Hausdorff spaces and \'etale groupoids, respectively, in the sense of Fujieda, Katsura and Uchimura in \cite{FKU}.
    \item[(E)] A universal property of KS-groupoids is established in Theorem \ref{thm:univ}.
    More precisely, for a given inverse semigroup $S$, actions on locally compact Hausdorff spaces $X,Y$ and an $S$-equivariant map $f\colon Y\to X$, we construct a canonical $S$-equivariant map $\rho$ from $Y$ to $X^\ks$ which induces a \textit{\d-bijective} groupoid homomorphism $\tilde{\rho}\colon S\ltimes Y\to S\ltimes X^\ks$. 
    For the definition of $\d$-bijective groupoid homomorphisms, see \cite[Definition 2.2]{es} or Definition \ref{def:d-bij}.
    Furthermore, it is also shown that $\tilde{\rho}$ is Borel if $f$ is Borel and $S^X=\{s\in S\mid X_{s^*s}\neq\emptyset\}$ is countable. 
    This generalizes the result in Section 3 of \cite{es} on universal groupoids to the setting of KS-groupoids.
  \end{itemize}
}

We next explain the main ideas behind these results and indicate some of the technical difficulties.

The starting point of the overall study is that, from \cite[Theorem 6.2]{ks2} (see also Subsection \ref{subsec:1.4}), for a given action of an inverse semigroup $S$ on a locally compact Hausdorff space $X$, we can construct a natural $S$-action on the KS-crossed product $E(S)\kstimes C_0(X)$, and there is a natural isomorphism
\begin{equation}
  \label{eq:0-3}
  S\kstimes C_0(X)\cong S\ltimes (E(S)\kstimes C_0(X)).
\end{equation}
Since $E(S)\kstimes C_0(X)$ is a commutative C*-algebra, one may compute the spectrum of $E(S)\kstimes C_0(X)$ and the $S$-action on it.
Although the spectrum admits a natural set-theoretic description as a subset of $X\times\widehat{E(S)}$, its topology can be strictly coarser than the induced topology.
Due to this fact, our approach for computing the spectrum of $E(S)\kstimes C_0(X)$ is to describe it as a concrete subspace $X^\ks$ of $\prod_{e\in E(S)}\widetilde{X_e}$ (see Definition \ref{def:of-X^ks} and Proposition \ref{prop:spectrum}), where $\widetilde{X_e}$ denotes the one-point compactification of the open subset $X_e$ of $X$ for each $e\in E(S)$.
The formula for the $S$-action on $X^\ks$ is naturally suggested by the $S$-action on $E(S)\kstimes C_0(X)$ and the homeomorphism from the spectrum to $X^\ks$.
In Subsection \ref{subsec:2.2}, we verify directly, without using C*-algebraic arguments, that this formula indeed defines an $S$-action.
Its compatibility with the action on $E(S)\kstimes C_0(X)$ is established later in Proposition \ref{prop:action-on-spectrum}.
Thus, the construction in (A) is carried out entirely from a topological point of view.

For (B), the isomorphism $S\kstimes C_0(X)\cong C^*(S\ltimes X^\ks)$ follows immediately from the argument above (see the beginning of Subsection \ref{subsec:4.2}). 
Therefore, the reduced case is the main remaining step for (B).
Its proof has three main ingredients.
{
  \setlength{\leftmargini}{33pt} 
  \begin{itemize}
    \item[(B-1)] Using a result of Khoshkam and Skandalis, we show that the reduced norm on $\mathcal{C}(S\ltimes X^\ks)$ is determined by the left regular representations at the points in $\{\tau_{e,\,x}\in X^\ks\mid e\in E(S),\;x\in X_e\}$ (Proposition \ref{prop:norm}), which form a dense subset of $(S\ltimes X^\ks)^{(0)}=X^\ks$.
    \item[(B-2)] By introducing \textit{induced representations} of $S\alg C_0(X)$ in Observation \ref{obs:ind-rep}, we show that the family of induced representations $\{\ind\rho_{e,\,x}\mid e\in E(S),\;x\in X_e\}$ determines the reduced norm on $S\alg C_0(X)$.
    \item[(B-3)] In Lemma \ref{lem:unitary-equiv}, we prove the unitary equivalence of the two representations $\ind\rho_{e,\,x}$ on $S\alg C_0(X)$ and the left regular representation of $\mathcal{C}(S\ltimes X^\ks)$ at $\tau_{e,\,x}\in X^\ks$, for each $e\in E(S)$ and $x\in X_e$, under the canonical algebraic map.
  \end{itemize}
}

For (C), (D) and (E), the explicit descriptions of $X^\ks$ and the $S$-action on $X^\ks$ are the main ingredients.
However, the characterization of Hausdorffness requires an additional argument.
For a given inverse semigroup $S$, Steinberg showed in \cite[Theorem 5.17]{steinberg} that the universal groupoid $S\ltimes\widehat{E(S)}$ is Hausdorff if and only if $S$ is a \textit{weak semilattice}.
On the other hand, Exel and Pardo in \cite[Theorem 3.15]{ep} gave a general characterization of Hausdorffness of the transformation groupoid $S\ltimes X$ in terms of the original action of an inverse semigroup $S$ on a locally compact Hausdorff space $X$.
In the beginning of Subsection \ref{subsec:3.1}, we obtain another characterization (Proposition \ref{prop:hausdorffness}) and show that the KS-groupoid $S\ltimes X^\ks$ is Hausdorff if and only if $S^X:=\{s\in S\mid X_{s^*s}\neq\emptyset\}$ is a weak semilattice (Theorem \ref{thm:hausdorffness-of-KS}), provided that $X_e$ is compact for each $e\in E(S)$.
This approach relates the two well-known results \cite[Theorem 5.17]{steinberg} and \cite[Theorem 3.15]{ep} mentioned above.
Examples \ref{ex1} and \ref{ex2} show that, in the general case, the condition $S^X$ being a weak semilattice is neither necessary nor sufficient for Hausdorffness of the KS-groupoid $S\ltimes X^\ks$.

This paper is organized as follows.

In Section \ref{sec:1}, we recall some basic facts about inverse semigroups and their actions, \'etale groupoids and their C*-algebras, and the constructions of the crossed products and KS-crossed products.

In Section \ref{sec:2}, we give a purely topological treatment of the space $X^\ks$ and the $S$-action on $X^\ks$ for a given action of an inverse semigroup $S$ on a locally compact Hausdorff space $X$.

In Section \ref{sec:3}, we study some topological properties of KS-groupoids (Subsection \ref{subsec:3.1}), the functoriality of the construction of KS-groupoids (Subsection \ref{subsec:3.2}) and a universal property of KS-groupoids (Subsection \ref{subsec:3.3}).

In Section \ref{sec:4}, we relate KS-groupoids to KS-crossed products. In Subsection \ref{subsec:4.1}, we prove that $X^\ks$ is canonically homeomorphic to the spectrum of the commutative C*-algebra $E(S)\kstimes C_0(X)$, and that the action of $S$ on $X^\ks$ constructed in Subsection \ref{subsec:2.2} corresponds, under this homeomorphism, to the canonical action of $S$ on $E(S)\kstimes C_0(X)$. 
We next prove dynamical Paterson's theorem in Subsection \ref{subsec:4.2}. 
As corollaries of the results in Subsection \ref{subsec:3.1} and dynamical Paterson's theorem, we obtain some elementary properties of full and reduced KS-crossed products, such as separability and ($\sigma$-)unitality, in Subsection \ref{subsec:4.3}.

For readers' convenience, we collect several auxiliary results and detailed proofs that are used throughout the paper in the appendices.

In Appendix \ref{sec:A}, we give an almost self-contained treatment of Exel's crossed product $S\ltimes_r A$. 
This also contains a detailed argument for the existence of appropriate cyclic vectors of GNS-representations on the $*$-algebra $S\alg A$ (Proposition \ref{prop:cyclic}) which is used in Appendix \ref{sec:C} to prove the reduced case of the isomorphisms \eqref{eq:0-1}.

In Appendix \ref{sec:B}, we give an alternative proof of \cite[Theorem 6.2]{ks2} (see the isomorphism \eqref{eq:0-3} above). 
Our approach is more algebraic and does not use the theory of covariant representations of KS-crossed products.

In Appendix \ref{sec:C}, we give a self-contained proof of the isomorphisms
\eqref{eq:0-1} above.
For the full case, we give a direct proof which avoids using covariant representations or disintegration theorems.
For another approach, see \cite[Proposition 4.23, Corollaries 4.26 and 4.27]{BKM2025} and \cite[Theorem 3.5]{Neshveyev2026}.
For the reduced case, the argument is essentially a reformulation of the proof of \cite[Theorem 4.6]{be} in the present setting.
But our approach is more algebraic: we work directly with the algebraic crossed product, using the existence of appropriate cyclic vectors for its GNS representations established in Proposition \ref{prop:cyclic}.

\noindent\textbf{AI Statement.}
The author used ChatGPT (GPT-5.6 Sol and GPT-6 Astra) to assist with writing this manuscript in English and correcting typographical errors.
All mathematical arguments in this paper were developed by the author under the supervision of Prof.\ Takeshi Katsura. 
The author takes full responsibility for the content of the paper.

\noindent\textbf{Acknowledgements.}
This paper is based on the author's master's thesis. 
The author would like to thank Prof.\ Takeshi Katsura and Fuyuta Komura for carefully reading earlier versions of this manuscript and providing many helpful comments.

\section{Preliminaries}
\label{sec:1}

\subsection{Inverse Semigroups and Their Actions}
\label{subsec:1.1}

An \textit{inverse semigroup} is a semigroup $S$ such that for any $s\in S$ there exists a unique element $s^*\in S$ (called the \textit{generalized inverse} for $s$) satisfying
\[
  ss^*s=s,\quad s^*ss^*=s^*.
\]
We denote by $E(S)$ the set of all idempotents in an inverse semigroup $S$. It is known that $E(S)$ is a commutative subsemigroup of $S$ (see \cite[Theorem 1.2.14]{lawson}). 
For $s,t\in S$, we write $s\le t$ if $ts^*s=s$. This defines a partial order on $S$ called the \textit{natural partial order} (see \cite[Proposition 2.1.2]{lawson}).

\begin{ex}
  Let $X$ be a set and $Y,Z\subset X$. We say that a map $f\colon Y\to Z$ is a \textit{partial bijection} on $X$ if $f\colon Y\to Z$ is bijective. We denote the set of all partial bijections on $X$ by $\mathcal{I}(X)$. For any two partial bijections $f_1\colon Y_1\to Z_1$ and $f_2\colon Y_2\to Z_2$ on $X$, define the partial composition $f_2\circ f_1$ by
  \[
    f_2\circ f_1\colon (f_1)^{-1}(Z_1\cap Y_2)\ni x\mapsto f_2(f_1(x))\in f_2(Z_1\cap Y_2).
  \]
  The set $\mathcal{I}(X)$ becomes an inverse semigroup with respect to the multiplication given by the partial composition, where the generalized inverse for $f\colon Y\to Z$ is the inverse $f^{-1}\colon Z\to Y$. We have $E(\mathcal{I}(X))=\{\mathrm{id}_Y\mid Y\subset X\}$. For $f_1\colon Y_1\to Z_1$, $f_2\colon Y_2\to Z_2$, we have $f_1\le f_2$ in $\mathcal{I}(X)$ if and only if $Y_1\subset Y_2$ and $f_1(x)=f_2(x)$ for all $x\in Y_1$.
\end{ex}

Let $S$ be an inverse semigroup. A \textit{set-theoretic action} $\theta$ of $S$ on a set $X$ is a semigroup homomorphism $\theta\colon S\ni s\mapsto \theta_s\in\mathcal{I}(X)$. 
Observe that, for each $e\in E(S)$, the partial bijection $\theta_e$ is an idempotent in $\mathcal{I}(X)$. 
Hence $\theta_e$ is the identity map on some subset of $X$. 
We denote by $X_e$ the corresponding subset of $X$ on which $\theta_e$ is the identity. 
Notice that the domain of $\theta_s$ is $X_{s^*s}$ and the range is $X_{ss^*}$ for each $s\in S$. 
If $s\le t$ in $S$, then we have $\theta_s\le\theta_t$ in $\mathcal{I}(X)$, that is, $\theta_t$ is an extension of $\theta_s$.

\begin{dfn}
  Let $S$ be an inverse semigroup. A \textit{(topological) action} $\theta$ of $S$ on a topological space $X$ is a set-theoretic action of $S$ on $X$ such that
  \begin{enumerate}
    \item $X_e$ is open for each $e\in E(S)$,
    \item $\theta_s\colon X_{s^*s}\to X_{ss^*}$ is continuous for each $s\in S$, and 
    \item the union $\bigcup_{e\in E(S)}X_e$ coincides with $X$.
  \end{enumerate}
\end{dfn}

\begin{rem}
  It follows from the definition that, for each $s\in S$, the map $\theta_s$ is a homeomorphism from $X_{s^*s}$ to $X_{ss^*}$ with the inverse $\theta_{s^*}$. 
\end{rem}

\begin{ex}
  \label{ex:spectral-action}
  Let $S$ be an inverse semigroup and let $\widehat{E(S)}$ denote the set of all nonzero semigroup homomorphisms $\chi\colon E(S)\to \{0,1\}$ equipped with the topology of pointwise convergence. Then $\widehat{E(S)}$ is locally compact Hausdorff. We can define an action $\theta^S$ of $S$ on $X=\widehat{E(S)}$ in the following way. For each $e\in E(S)$, set $X_e=\{\chi\in X\mid\chi(e)=1\}$. Clearly, $X_e$ is a compact open subset of $X$ for each $e\in E(S)$. For each $s\in S$, define $\theta^S_s\colon X_{s^*s}\to X_{ss^*}$ by
  \[
    \theta^S_s(\chi)(e)=\chi(s^*es)
  \]
  for all $\chi\in X_{s^*s}$ and $e\in E(S)$. These assignments define an $S$-action on $\widehat{E(S)}$.
\end{ex}

\begin{dfn}
  Let $\theta,\sigma$ be actions of $S$ on topological spaces $X,Y$, respectively. A map $f\colon X\to Y$ is said to be $S$-equivariant if $f$ satisfies $f(X_e)\subset Y_e$ for each $e\in E(S)$, and the diagram
  \begin{center}
    \begin{tikzpicture}[auto]
      \node (1) at (0,1.5) {$X_{s^*s}$};
      \node (2) at (2,1.5) {$X_{ss^*}$};
      \node (3) at (0,0) {$Y_{s^*s}$};
      \node (4) at (2,0) {$Y_{ss^*}$};
      \draw[->] (1) to node {$\scriptstyle\theta_s$} (2);
      \draw[->] (1) to node[swap] {$\scriptstyle f|_{X_{s^*s}}$} (3);
      \draw[->] (2) to node {$\scriptstyle f|_{X_{ss^*}}$} (4);
      \draw[->] (3) to node[swap] {$\scriptstyle\sigma_s$} (4);
    \end{tikzpicture}
  \end{center}
  commutes for each $s\in S$.
\end{dfn}

\begin{dfn}[{\cite[Definition 9.1]{exel2008}}]
  Let $S$ be an inverse semigroup and $A$ be a C*-algebra. A \textit{(C*-algebraic) action} $\alpha$ of $S$ on $A$ is a set-theoretic action of $S$ on $A$ satisfying
  \begin{enumerate}
    \item $A_e$ is a (closed two-sided) ideal of $A$ for each $e\in E(S)$,
    \item $\alpha_s\colon A_{s^*s}\to A_{ss^*}$ is a $*$-homomorphism for each $s\in S$, and
    \item the sum $\sum_{e\in E(S)}A_e$ is dense in $A$.
  \end{enumerate}
  We say that $(A,\alpha)$, or simply $A$, is an \textit{$S$-C*-algebra} if $A$ is equipped with an $S$-action $\alpha$.
\end{dfn}

\begin{ex}
  \label{ex:induced-action-on-C0}
  Let $\theta$ be an action of $S$ on a locally compact Hausdorff space $X$. Then $\theta$ induces an action on the C*-algebra $C_0(X)$ by 
  \begin{enumerate}
    \item $C_0(X)_e:=C_0(X_e)\triangleleft C_0(X)$ for each $e\in E(S)$, and
    \item $\alpha_s\colon C_0(X_{s^*s})\ni f\mapsto f\circ\theta_{s^*}\in C_0(X_{ss^*})$ for each $s\in S$.
  \end{enumerate}
  Conversely, every $S$-action on $C_0(X)$ is of this form. 
\end{ex}

\subsection{\'{E}tale Groupoids}
\label{subsec:1.2}

We recall some fundamental notions about \'etale groupoids and the constructions of their C*-algebras.

A \textit{groupoid} is a set $G$ together with a distinguished subset $G^{(0)}\subset G$ (called the \textit{unit space} of $G$), domain and range maps $d,r\colon G\to G^{(0)}$, a multiplication map
\[
  G^{(2)}:=\{(\gamma_2,\gamma_1)\in G\times G\mid d(\gamma_2)=r(\gamma_1)\}\ni(\alpha,\beta)\mapsto \alpha\beta\in G
\]
and an inverse map
\[
  G\ni \gamma\mapsto \gamma^{-1}\in G
\]
satisfying
\begin{enumerate}
  \item $d(x)=x=r(x)$ for all $x\in G^{(0)}$,
  \item $r(\gamma)\gamma=\gamma=\gamma d(\gamma)$ for all $\gamma\in G$,
  \item $r(\gamma_2\gamma_1)=r(\gamma_2)$ and $d(\gamma_2\gamma_1)=d(\gamma_1)$ for all $(\gamma_2,\gamma_1)\in G^{(2)}$,
  \item $(\gamma_3\gamma_2)\gamma_1=\gamma_3(\gamma_2\gamma_1)$ for all $(\gamma_3,\gamma_2),(\gamma_2,\gamma_1)\in G^{(2)}$, and 
  \item $\gamma^{-1}\gamma=d(\gamma),\;\gamma\gamma^{-1}=r(\gamma)$ for all $\gamma\in G$.
\end{enumerate}

Let $G,H$ be groupoids. A \textit{groupoid homomorphism} $f\colon G\to H$ is a map such that if $(\gamma_2,\gamma_1)\in G^{(2)}$, then $(f(\gamma_2),f(\gamma_1))\in H^{(2)}$ and $f(\gamma_2\gamma_1)=f(\gamma_2)f(\gamma_1)$.

A \textit{topological groupoid} is a groupoid $G$ equipped with a topology such that the inverse map $G\ni \gamma\mapsto \gamma^{-1}\in G$ and the multiplication map
\[
  G^{(2)}\ni(\alpha,\beta)\mapsto \alpha\beta\in G
\]
are continuous.

\begin{dfn}
  A topological groupoid $G$ is said to be \textit{\'{e}tale} if the domain map $d\colon G\to G^{(0)}$ is a local homeomorphism, that is, for all $\gamma\in G$, there exists an open neighborhood $U\subset G$ of $\gamma$ such that $d(U)$ is open in $G^{(0)}$ and $d|_U\colon U\to d(U)$ is a homeomorphism.
\end{dfn}

In this paper, we always assume that every \'{e}tale groupoid has a locally compact Hausdorff unit space. 
However, we do not assume that \'etale groupoids are Hausdorff.

For any $x\in G^{(0)}$, let $G_x=\{\gamma\in G\mid d(\gamma)=x\}$. Note that, if $G$ is \'etale, then $G_x$ is a discrete subset of $G$. 

\begin{ex}
  For an action $\theta$ of an inverse semigroup $S$ on a locally compact Hausdorff space $X$, we associate an \'{e}tale groupoid $S\ltimes X$ called the \textit{transformation groupoid} as follows (see Section 4 of \cite{exel2008} for details).
  Let 
  \[
    S*X=\{(s,x)\in S\times X\mid x\in X_{s^*s}\}.
  \]
  We define the equivalence relation $\sim$ on $S*X$ by $(s,x)\sim(t,y)$ if $x=y$ and $se=te$, $x\in X_e$ for some $e\in E(S)$. Let $S\ltimes X=S*X\,/\,{\sim}$ and let $[s,x]$ denote the equivalence class of $(s,x)\in S*X$.

  The unit space $(S\ltimes X)^{(0)}$ of $S\ltimes X$ is $\{[e,x]\mid e\in E(S),\;x\in X_e\}$, which we identify with $X$ via the bijection $(S\ltimes X)^{(0)}\ni [e,x]\mapsto x\in X$. The domain and range maps $d,r\colon S\ltimes X\to X$ are given by
  \[
    d([s,x])=x,\quad r([s,x])=\theta_s(x).
  \]
  The multiplication is defined by $[t,y][s,x]=[ts,x]$ whenever $[s,x],[t,y]\in S\ltimes X$ satisfy $d([t,y])=r([s,x])$, that is, $y=\theta_s(x)$. The inverse is defined by $[s,x]^{-1}=[s^*,\theta_s(x)]$. These operations are well-defined and they make $S\ltimes X$ a groupoid. Moreover, the subsets of the form
  \[
  [s,U]:=\{[s,x]\in S\ltimes X\mid x\in U\},\quad s\in S,\quad U\in\mathbb{O}(X_{s^*s})
  \]
  form an open basis for $S\ltimes X$, and this topology makes $S\ltimes X$ a topological groupoid. Here, $\mathbb{O}(X_{s^*s})$ denotes the set of all open subsets in $X_{s^*s}$. The bijection above is a homeomorphism from $(S\ltimes X)^{(0)}$ onto $X$. Since $d|_{[s,U]}\colon [s,U]\to U$ is easily seen to be a homeomorphism for any $s\in S$ and $U\in\mathbb{O}(X_{s^*s})$, it follows that $S\ltimes X$ is an \'{e}tale groupoid.

  The \textit{universal groupoid} $S\ltimes\widehat{E(S)}$ associated with an inverse semigroup $S$ is the transformation groupoid corresponding to the $S$-action on $\widehat{E(S)}$ described in Example \ref{ex:spectral-action}.
  
  The next lemma will be needed in Subsection \ref{subsec:3.3}.
\end{ex}

\begin{lem}
  \label{lem:hom}
  Let $\theta,\sigma$ be actions of an inverse semigroup $S$ on locally compact Hausdorff spaces $X,Y$, respectively. If $f\colon X\to Y$ is an $S$-equivariant map, then the map $\tilde{f}\colon S\ltimes X\to S\ltimes Y$ defined by
  \[
    \tilde{f}([s,x])=[s,f(x)]
  \]
  is a well-defined groupoid homomorphism. Moreover, $\tilde{f}$ is continuous (resp.\ Borel) if $f$ is continuous (resp.\ $S^X=\{s\in S\mid X_{s^*s}\neq\emptyset\}$ is countable and $f$ is Borel).
\end{lem}
\begin{proof}
  It is routine to verify that $\tilde{f}$ is a well-defined groupoid homomorphism.

  The last statement follows from the observation that
  \[
  (\tilde{f})^{-1}(U)=\bigcup_{t\in S^X}(\tilde{f}|_{[t,X_{t^*t}]})^{-1}(U)
  \]
  for any $U\in\mathbb{O}(S\ltimes Y)$, and with the identifications $[t,X_{
  t^*t}]=X_{t^*t}$ and $[t,Y_{t^*t}]=Y_{t^*t}$ via domain maps we have $\tilde{f}|_{[t,X_{t^*t}]}=f|_{X_{t^*t}}$, which is continuous (resp.\ Borel) provided that $f$ is continuous (resp.\ $f$ is Borel).
\end{proof}

Throughout this subsection, let $G$ denote an \'etale groupoid. 

A \textit{bisection} $U$ of $G$ is a subset such that the restrictions $d|_U,r|_U\colon U\to G^{(0)}$ are injective. 
If $U$ is an open bisection, it follows that $d|_U,r|_U$ are homeomorphisms onto their images since $d,r\colon G\to G^{(0)}$ are open maps. 
In particular, $U$ is locally compact Hausdorff. 
We denote by $\bis(G)$ the set of all open bisections of $G$. 
It is known that $\bis(G)$ forms an open basis for $G$ (see Proposition 3.5 of \cite{exel2008}). 
It is also known that $\bis(G)$ is an inverse semigroup with respect to the multiplication given by 
\[
  U V=\{\gamma\eta\in G\mid \gamma\in U,\;\eta\in V,\;(\gamma,\eta)\in G^{(2)}\}
\]
for all $U,V\in \bis(G)$, where the generalized inverse of $U\in\bis(G)$ is given by $U^{-1}=\{\gamma^{-1}\mid \gamma\in U\}$ (see Section 5 of \cite{exel2008}). 

For any set $X$, let $\mathbb{C}^X$ denote the linear space of all $\mathbb{C}$-valued functions on $X$.

If $U$ is an open bisection of $G$, we may view $C_c(U)\subset\mathbb{C}^G$ by regarding each function $f$ in $C_c(U)$ as a function on $G$ by extending it by zero outside $U$. 
We define $\mathcal{C}(G)$ to be the linear span of $\bigcup_UC_c(U)$ in $\mathbb{C}^G$, where the union is taken over all open bisections $U$ of $G$.

A standard partition of unity argument implies the following.

\begin{lem}[{\cite[Proposition 3.10]{exel2008}}]
  \label{lem:partition}
  Let $\mathcal{O}$ be a covering of $G$ consisting of open bisections of $G$. Then
  \[
    \mathcal{C}(G)=\Span\bigcup_{U\in\mathcal{O}}C_c(U).
  \]
\end{lem}

It is shown in \cite[Proposition 3.11]{exel2008} that $\mathcal{C}(G)$ is a $*$-algebra with respect to the multiplication and involution given by
\[
  f*g(\gamma)=\sum_{\beta\in G_{d(\gamma)}}f(\gamma\beta^{-1})g(\beta),\quad f^*(\gamma)=\overline{f(\gamma^{-1})}
\]
for all $f,g\in\mathcal{C}(G)$. Note also from \cite[Proposition 3.11]{exel2008} that, if $f\in C_c(U)$ and $g\in C_c(V)$ for some $U,V\in \bis(G)$, then $f*g\in C_c(UV)$ and $f^*\in C_c(U^{-1})$.

It is known that, for any $*$-homomorphism $\varphi\colon\mathcal{C}(G)\to D$ into a C*-algebra $D$, we have $\|\varphi(f)\|\le \|f\|_\infty$ for any $U\in\bis(G)$ and $f\in C_c(U)$ (see, for example, \cite[Proposition 3.14]{exel2008}). 
Hence we can consider the universal enveloping C*-algebra $C^*(G)$ of the $*$-algebra $\mathcal{C}(G)$, which we call the \textit{full groupoid C*-algebra} of $G$.

We define the \textit{left regular representation} $\lambda_x\colon \mathcal{C}(G)\to\mathbb{B}(l^2(G_x))$ at $x\in G^{(0)}$ by
\[
  \lambda_x(f)\delta_\gamma=\sum_{\beta\in G_{r(\gamma)}}f(\beta)\delta_{\beta\gamma},
\]
where $f\in \mathcal{C}(G)$ and $\gamma\in G_x$. It is known that the map $\lambda_x\colon \mathcal{C}(G)\to\mathbb{B}(l^2(G_x))$ defined above is a well-defined $*$-homomorphism and the $*$-homomorphism 
\[
  \lambda\colon \mathcal{C}(G)\ni f\mapsto (\lambda_x(f))_{x\in G^{(0)}}\in \prod_{x\in G^{(0)}}\mathbb{B}(l^2(G_x))
\]
is well-defined and injective (see, for example, \cite{paterson, as}).

\begin{dfn}
  Define the C*-norm $\|\cdot\|_r$ on $\mathcal{C}(G)$ by 
  \[
    \|f\|_r=\|\lambda(f)\|=\sup_{x\in G^{(0)}}\|\lambda_x(f)\|.
  \]
  The completion of $\mathcal{C}(G)$ with respect to the C*-norm $\|\cdot\|_r$ is denoted by $C^*_r(G)$. We call this C*-algebra the \textit{reduced groupoid C*-algebra} of $G$.
\end{dfn}

The following will be needed in Subsection \ref{subsec:4.2}.

\begin{thm}
  \label{thm:KS}
  Let $G$ be an \'{e}tale groupoid and let $Z\subset G^{(0)}$ be a subset which is dense in $G^{(0)}$ with respect to the initial topology on $G^{(0)}$ induced from the family of functions $\{f|_{G^{(0)}}\mid f\in C_c(U),\;U\in\bis(G)\}$ on $G^{(0)}$. Then 
  \[
    \|f\|_r=\sup_{z\in Z}\|\lambda_z(f)\|
  \]
  for all $f\in\mathcal{C}(G)$.
\end{thm} 
\begin{proof}
  See \cite[Corollary 2.11]{ks1} and the beginning of the proof of \cite[Lemma 3.4]{ks1}.
\end{proof}

\subsection{Crossed Products and KS-Crossed Products}
\label{subsec:1.3}

There are four different constructions of C*-algebras arising from inverse semigroup actions on C*-algebras: full crossed products introduced by Sieben in \cite{ns}, reduced crossed products introduced by Exel in \cite{exel2011}, and the full and reduced KS-crossed products introduced by Khoshkam and Skandalis in \cite{ks2}. In this subsection, we review their constructions and set up the notation used throughout the paper.

\noindent\textbf{KS-Crossed Products.}
For a given $S$-action on $A$, let $S\alg A$ denote the vector space $\bigoplus_{s\in S}A_{s^*s}$. We denote by $\delta_s\,x$ the element in $S\alg A$ whose components are all zero except for its $s$-th component which is $x\in A_{s^*s}$. Then any element $T$ in $S\alg A$ is written as
\[
  T=\sum_{i=1}^n\delta_{s_i}\,x_i.
\]
A routine verification shows that $S\alg A$ becomes a $*$-algebra with respect to the multiplication and involution given by
\[
  (\delta_s\,x)\cdot(\delta_t\,y)=\delta_{st}\,\alpha_{t^*}(x\alpha_t(y)),\quad (\delta_s\,x)^*=\delta_{s^*}\,\alpha_s(x^*)
\]
for any $s,t\in S$, $x\in A_{s^*s}$ and $y\in A_{t^*t}$ (see \cite[Proposition 4.1]{ns}).

It is sometimes useful to consider the unique $*$-homomorphism $q_s\colon A\to M(A_{ss^*})$ such that the diagram 
\begin{center}
  \begin{tikzpicture}[auto]
    \node (3) at (0,0) {$A$};
    \node (4) at (2.3,0) {$M(A_{ss^*})$};
    \node (1) at (0,1.2) {$A_{s^*s}$};
    \node (2) at (2,1.2) {$A_{ss^*}$};
    \node (void) at (2,0.15) {};
    \draw[->] (3) to node[swap] {$\scriptstyle q_s$} (4);
    \draw[right hook->] (1) to node {} (3);
    \draw[right hook->] (2) to node {} (void);
    \draw[->] (1) to node {$\scriptstyle\alpha_s$} (2);
  \end{tikzpicture}
\end{center}
commutes. In fact, we have $q_{t^*}(x)y=\alpha_{t^*}(x\alpha_t(y))$ for any $t\in S$, $x\in A$ and $y\in A_{t^*t}$. Note also that $q_s(\alpha_t(x))y=q_{st}(x)y$ for any $s,t\in S$, $x\in A_{t^*t}$ and $y\in A_{(st)(st)^*}$.

\begin{lem}
  \label{lem:contractive}
  Let $B$ be a C*-algebra and $\pi\colon S\alg A\to B$ be a $*$-homomorphism. Then 
  \[
    \|\pi(T)\|\le \|T\|_1
  \]
  for all $T\in S\alg A$, where $\|\cdot\|_1$ denotes the $L^1$-norm on $S\alg A$.
\end{lem}
\begin{proof}
  Observe first that for each $s\in S$ and $x\in A_{s^*s}$, we have
  \[
    \|\pi(\delta_s\,x)\|^2=\|\pi(\delta_{s^*s}\,x^*x)\|\le \|x\|^2,
  \]
  where we have used the fact that $A_{s^*s}\ni y\mapsto \pi(\delta_{s^*s}\,y)\in B$ is a $*$-homomorphism.

  Write $T=\sum_{i=1}^n\delta_{s_i}\,x_i$ with distinct $s_i$'s. It follows from the observation above that
  \[
    \|\pi(T)\|\le \sum_{i=1}^n\|\pi(\delta_{s_i}\,x_i)\|\le \sum_{i=1}^n\|x_i\|=\|T\|_1.\qedhere
  \]
\end{proof}

From the lemma above, we can consider the universal enveloping C*-algebra $S\kstimes A$ of the $*$-algebra $S\alg A$, which we call the \textit{full KS-crossed product} of $A$ by $\alpha$.

In \cite{ks2}, Khoshkam and Skandalis construct a faithful representation of $S\alg A$ on a Hilbert module, which we call the \textit{left regular representation} $\ell$ of $S\alg A$, and define the reduced version of their crossed products. To review this construction, we first represent $S\alg A$ on a Hilbert $A_e$-module $l^2(S_e,A_e)=\overline{\mathrm{span}}\,\{\delta_t\,y\mid t\in S_e,\;y\in A_e\}$, that is, we construct a $*$-homomorphism from $S\alg A$ to $\mathbb{B}(l^2(S_e,A_e))$, where $\mathbb{B}(l^2(S_e,A_e))$ is the C*-algebra consisting of all adjointable maps on $l^2(S_e,A_e)$.

For each $e\in E(S)$, let $S_e=\{s\in S\mid s^*s=e\}$.

\begin{lem}[cf.\ {\cite[Proposition 5.10]{ks2}}]
  For each $e\in E(S)$, there exists a unique $*$-homomorphism $\ell_e\colon S\alg A\to\mathbb{B}(l^2(S_e,A_e))$ such that
  \begin{align*}
    \ell_e(\delta_s\,x)\,(\delta_t\,y)=
    \begin{cases}
      \delta_{st}\,\alpha_{t^*}(x\alpha_t(y)) &\mathrm{if}\; s^*st=t\\
      0 & \mathrm{otherwise}
    \end{cases}
  \end{align*}
  for all $s\in S$, $x\in A_{s^*s}$, $t\in S_e$ and $y\in A_e$, where $\mathbb{B}(l^2(S_e,A_e))$ denotes the C*-algebra consisting of all adjointable maps on $l^2(S_e,A_e)$.
\end{lem}
\begin{proof}
  For any $s\in S$ and $x\in A_{s^*s}$, define a linear map $\ell_e(\delta_s\,x)$ on $\mathrm{span}\,\{\delta_t\,y\mid t\in S_e,\;y\in A_e\}$ by
  \begin{align*}
    \ell_e(\delta_s\,x)\,(\delta_t\,y)=
    \begin{cases}
      \delta_{st}\,\alpha_{t^*}(x\alpha_t(y)) &\mathrm{if}\; s^*st=t\\
      0 & \mathrm{otherwise}
    \end{cases}
  \end{align*}
  for all $t\in S_e$ and $y\in A_e$.
  We claim that
  \[
    \|\ell_e(\delta_s\,x)\xi\|\le\|x\|\|\xi\|
  \]
  for all $s\in S$, $x\in A_{s^*s}$, and $\xi\in\mathrm{span}\,\{\delta_t\,y\mid t\in S_e,\;y\in A_e\}$. To see this, write $\xi=\sum_{t\in F}\delta_t\,y_t$ for some finite subset $F\subset S_e$ and $y_t\in A_e$ for each $t\in F$, and set $F_0=\{t\in F\mid s^*st=t\}$. Compute that
  \begin{align*}
    \langle\ell_e(\delta_s\,x)\xi\mid\ell_e(\delta_s\,x)\xi\rangle
    &=\sum_{t_1,t_2\in F_0}\bigl\langle\delta_{st_1}\,q_{t_1^*}(x)y_{t_1}\bigm|\delta_{st_2}\,q_{t_2^*}(x)y_{t_2}\bigr\rangle\\
    &=\sum_{t\in F_0}y_t^*q_{t^*}(x^*x)y_t\\
    &\le \|x\|^2\sum_{t\in F_0}y_t^*y_t\\
    &\le \|x\|^2\langle\xi|\xi\rangle,
  \end{align*}
  where we have used the observation in the second equality that $st_1=st_2$ if and only if $t_1=t_2$ for each $t_1,t_2\in F_0$. 

  This claim implies that $\ell_e(\delta_s\,x)$ extends uniquely to a bounded map on $l^2(S_e,A_e)$ with $\|\ell_e(\delta_s\,x)\|\le \|x\|$ for each $s\in S$ and $x\in A_{s^*s}$.

  We next show that $\ell_e(\delta_s\,x)$ is adjointable with adjoint $\ell_e((\delta_s\,x)^*)$. To this end, it suffices to show that
  \[
    \bigl\langle\ell_e((\delta_s\,x)^*)\,(\delta_{t_1}\,y_1)\bigm|\delta_{t_2}\,y_2\bigr\rangle=\bigl\langle \delta_{t_1}\,y_1\bigm|\ell_e(\delta_s\,x)\,(\delta_{t_2}\,y_2)\bigr\rangle
  \]
  for all $t_1,t_2\in S_e$ and $y_1,y_2\in A_e$. First observe that the two conditions
  \begin{enumerate}
    \item $ss^*t_1=t_1$ and $s^*t_1=t_2$,
    \item $s^*st_2=t_2$ and $t_1=st_2$
  \end{enumerate}
  are equivalent. If these equivalent conditions hold, we compute that 
  \begin{align*}
    \bigl\langle\ell_e((\delta_s\,x)^*)\,(\delta_{t_1}\,y_1)\bigm|\delta_{t_2}\,y_2\bigr\rangle
    &=\bigl\langle\ell_e(\delta_{s^*}\,\alpha_s(x^*))\,(\delta_{t_1}\,y_1)\bigm|\delta_{t_2}\,y_2\bigr\rangle\\
    &=\bigl\langle \delta_{s^*t_1}\, q_{t_1^*}(\alpha_s(x^*))y_1\bigm|\delta_{t_2}\,y_2\bigr\rangle\\
    &=\bigl\langle \delta_{s^*t_1}\, q_{t_1^*s}(x^*)y_1\bigm|\delta_{t_2}\,y_2\bigr\rangle\\
    &=y_1^*q_{t_2^*}(x)y_2,
  \end{align*}
  and
  \begin{align*}
    \bigl\langle \delta_{t_1}\,y_1\bigm|\ell_e(\delta_s\,x)\,(\delta_{t_2}\,y_2)\bigr\rangle
    &=\bigl\langle\delta_{t_1}\,y_1\bigm|\delta_{st_2}\,q_{t_2^*}(x)y_2\bigr\rangle\\
    &=y_1^*q_{t_2^*}(x)y_2.
  \end{align*}

  If the equivalent conditions do not hold, we have
  \[
    \bigl\langle\ell_e((\delta_s\,x)^*)\,(\delta_{t_1}\,y_1)\bigm|\delta_{t_2}\,y_2\bigr\rangle=0=\bigl\langle \delta_{t_1}\,y_1\bigm|\ell_e(\delta_s\,x)\,(\delta_{t_2}\,y_2)\bigr\rangle.
  \]
  This proves $\ell_e(\delta_s\,x)\in\mathbb{B}(l^2(S_e,A_e))$ and $\ell_e$ is $*$-preserving.

  It only remains to show that $\ell_e$ is multiplicative. To this end, we will show that
  \[
    \ell_e(\delta_{s_2}\,x_2)\ell_e(\delta_{s_1}\,x_1)\,(\delta_t\,y)=\ell_e((\delta_{s_2}\,x_2)(\delta_{s_1}\,x_1))\,(\delta_t\,y)
  \]
  for all $s_1,s_2\in S$, $x_1\in A_{s_1^*s_1}$, $x_2\in A_{s_2^*s_2}$, $t\in S_e$ and $y\in A_e$. First notice that the two conditions
  \begin{enumerate}
    \item $s_1^*s_1t=t$ and $s_2^*s_2(s_1t)=s_1t$,
    \item $(s_2s_1)^*(s_2s_1)t=t$
  \end{enumerate}
  are equivalent. If these equivalent conditions hold, we compute that 
  \begin{align*}
    \ell_e(\delta_{s_2}\,x_2)\ell_e(\delta_{s_1}\,x_1)\,(\delta_t\,y)
    &=\ell_e(\delta_{s_2}\,x_2)\,(\delta_{s_1t}\,q_{t^*}(x_1)y)\\
    &=\delta_{s_2s_1t}\,q_{t^*s_1^*}(x_2)q_{t^*}(x_1)y\\
    \ell_e((\delta_{s_2}\,x_2)(\delta_{s_1}\,x_1))\,(\delta_t\,y)
    &=\ell_e(\delta_{s_2s_1}\,q_{s_1^*}(x_2)x_1)\,(\delta_t\,y)\\
    &=\delta_{s_2s_1t}\,q_{t^*s_1^*}(x_2\alpha_{s_1}(x_1))y\\
    &=\delta_{s_2s_1t}\,q_{t^*s_1^*}(x_2)q_{t^*}(x_1)y.
  \end{align*}
  If the equivalent conditions do not hold, then we have
  \[
    \ell_e(\delta_{s_2}\,x_2)\ell_e(\delta_{s_1}\,x_1)\,(\delta_t\,y)=0=\ell_e((\delta_{s_2}\,x_2)(\delta_{s_1}\,x_1))\,(\delta_t\,y).
  \]
  Hence $\ell_e$ is a $*$-homomorphism.
\end{proof}

\begin{prop}[cf. {\cite[Theorem 4.3]{bn}}]
  The $*$-representation
  \[
    \ell\colon S\alg A\ni T\mapsto (\ell_e(T))_{e\in E(S)}\in\prod_{e\in E(S)}\mathbb{B}(l^2(S_e,A_e))\subset \mathbb{B}\Bigl(\bigoplus_{e\in E(S)}l^2(S_e,A_e)\Bigr)
  \]
  of $S\alg A$ on the Hilbert $\bigoplus_{e\in E(S)} A_e$-module $\bigoplus_{e\in E(S)} l^2(S_e,A_e)$ is well-defined and faithful.
\end{prop}

\begin{proof}
  The well-definedness follows from Lemma \ref{lem:contractive}.

  Let $T$ be a non-zero element in $S\alg A$. Write $T=\sum_{s\in F}\delta_s\,x_s$ for some non-empty finite subset $F\subset S$ and non-zero $x_s\in A_{s^*s}$ for each $s\in F$. Fix $t\in F$ such that $t^*t$ is a maximal element in $\{s^*s\mid s\in F\}$. Using the observation that, for each $s\in F$, $s^*st^*t=t^*t$ if and only if $s\in S_{t^*t}$, we compute that
  \[
    \ell_{t^*t}(T)\,(\delta_{t^*t}\,x_t^*)=\sum_{s\in F\cap S_{t^*t}}\delta_{st^*t}\,x_sx_t^*=\sum_{s\in F\cap S_{t^*t}}\delta_s\,x_sx_t^*\neq 0.
  \]
  In particular, $\ell(T)\neq 0$. Thus $\ell$ is faithful.
\end{proof}

\begin{dfn}
  For an $S$-C*-algebra $(A,\alpha)$, define the C*-norm $\|\cdot\|_{\ks,r}$ on $S\alg A$ by
  \[
    \|\cdot\|_{\ks,r}:=\|\ell(\cdot)\|=\sup\{\|\ell_e(\cdot)\|\mid e\in E(S)\}.
  \]
  Let $S\ksr A$ denote the completion of $S\alg A$ with respect to $\|\cdot\|_{\ks,r}$. We call this C*-algebra the \textit{reduced KS-crossed product} of $A$ by $\alpha$.
\end{dfn} 

\begin{rem}
  Since the left regular representation $\ell$ above is faithful, the canonical maps from $S\alg A$ to $S\kstimes A$ and $S\ksr A$ are injective. 
\end{rem}

\begin{rem}
  \label{rem:C*S}
  If $A$ is the $S$-C*-algebra $\mathbb{C}$ equipped with the identity action of $S$, then $S\kstimes A$ (resp.\ $S\ksr A$) coincides with the full (resp.\ reduced) inverse semigroup C*-algebra $C^*(S)$ (resp.\ $C^*_r(S)$).
\end{rem}

In the following, we introduce the induced representation $\ind\rho_e$ of $S\alg A$ from each representation $\rho_e$ of $A_e$ on some Hilbert $D_e$-module $E_e$ for each $e\in E(S)$. This will be needed in Subsection \ref{subsec:4.2}.

\begin{obs}
  \label{obs:ind-rep}
  For each $e\in E(S)$, let $D_e$ be a C*-algebra, $E_e$ be a Hilbert $D_e$-module and $\rho_e\colon A_e\to\mathbb{B}(E_e)$ be a representation. Then there exists a $*$-homomorphism
  \[
    \mathbb{B}(l^2(S_e,A_e))\ni T\mapsto T\otimes_{\rho_e}1\in\mathbb{B}(l^2(S_e,A_e)\otimes_{\rho_e}E_e),
  \]
  where $\otimes_{\rho_e}$ denotes the interior tensor product (see Chapter 4 of \cite{lance}).
  It is routine to verify that this map is injective if $\rho_e$ is faithful. We further observe that the map
  \[
    l^2(S_e,A_e)\otimes_{\rho_e}E_e\ni(\delta_s\,x)\otimes_{\rho_e}\xi\mapsto \delta_s\,\rho_e(x)\xi\in l^2(S_e,E_e)
  \]
  preserves the inner products. Moreover, if $\rho_e\colon A_e\to\mathbb{B}(E_e)$ is a non-degenerate representation, then this map is surjective and hence $l^2(S_e,A_e)\otimes_{\rho_e}E_e$ and $l^2(S_e,E_e)$ are unitarily equivalent. For a non-degenerate representation $\rho_e\colon A_e\to\mathbb{B}(E_e)$, we denote the composition
  \[
    S\alg A\,\overset{\ell_e}{\to}\,\mathbb{B}(l^2(S_e,A_e))\,\to\,\mathbb{B}(l^2(S_e,A_e)\otimes_{\rho_e}E_e)\,\cong\,\mathbb{B}(l^2(S_e,E_e))
  \]
  by $\ind \rho_e$.
  
  A direct computation gives
  \begin{align*}
    \ind \rho_e(\delta_s\,x)\,(\delta_t\,\rho_e(y)\xi)
    =\begin{cases}
      \delta_{st}\, \rho_e(\alpha_{t^*}(x\alpha_t(y)))\xi & \mathrm{if}\;s^*st=t\\
      0 & \mathrm{otherwise}.
    \end{cases}
  \end{align*}

  Notice also that if moreover $\rho_e$ is faithful, then 
  \begin{align}
    \label{eq:rep}
    \|\ind \rho_e(T)\|=\|\ell_e(T)\| 
  \end{align}
  for all $T\in S\alg A$. In particular, if each $\rho_e$ is faithful and non-degenerate, then we can faithfully represent $S\ksr A$ on the Hilbert $\bigoplus_{e\in E(S)}D_e$-module $\bigoplus_{e\in E(S)}l^2(S_e,E_e)$.
\end{obs}

If $A$ is an $S$-C*-algebra, then $A$ is also an $E(S)$-C*-algebra in a canonical way. It is known that the canonical surjection from $E(S)\kstimes A$ onto $E(S)\ksr A$ is an isomorphism (\cite[Proposition 5.2]{bn}). We present an alternative proof using the following lemma, which will also be needed later.

\begin{lem}
  \label{lem:closed}
  Let $E_0\subset E(S)$ be a finite subsemigroup. Then the subset
  \[
    E_0\alg A:=\Span\{\delta_e\,x\mid e\in E_0,\;x\in A_e\}
  \]
  is closed in $E(S)\kstimes A$. 
\end{lem}
\begin{proof}
  We prove this by induction on $n=|E_0|$. 

  If $n=1$, that is, $E_0=\{e\}$ for some $e\in E(S)$, then the claim is true since the map $A_{e}\ni x\mapsto \delta_e\,x\in E(S)\kstimes A$ is a $*$-homomorphism between C*-algebras and its range is precisely $E_0\alg A$. 

  Suppose that the claim is true for $n=k\ge 1$. Consider the case where $n=k+1$. Then we can find $e_0\in E_0$ such that $e_0\not<e$ for each $e\in E_0$. Observe that $E_0':=E_0\setminus\{e_0\}$ is a subsemigroup of $E(S)$ with $|E_0'|= k$. Hence it follows from the induction hypothesis that 
  \[
    E_0'\alg A=\Span\{\delta_e\,x\mid e\in E_0',\;x\in A_e\}
  \]
  is closed in $E(S)\kstimes A$, or equivalently, complete with respect to the norm induced from $E(S)\kstimes A$. Moreover, $E_0'\alg A$ is a $*$-ideal of the $*$-algebra $E_0\alg A$, and the quotient $E_0\alg A\,/\,E_0'\alg A$ is canonically isomorphic to $A_{e_0}$. Since the quotient norm on $E_0\alg A\,/\,E_0'\alg A=A_{e_0}$ is a C*-norm, it is also complete. It follows that $E_0\alg A$ is complete and hence it is closed in $E(S)\kstimes A$, as desired.
\end{proof}

\begin{cor}
  The C*-norm on $E(S)\alg A$ is unique, and the canonical surjection gives
  \[
    E(S)\kstimes A\cong E(S)\ksr A.
  \]
\end{cor}
\begin{proof}
  From the lemma above, we can write $E(S)\alg A=\bigcup_\lambda B_\lambda$, where $\{B_\lambda\}$ is a directed family of C*-subalgebras of $E(S)\kstimes A$. This completes the proof.
\end{proof}

\noindent\textbf{Crossed Products.}
In the following, we review the constructions of full and reduced crossed products due to Sieben \cite{ns} and Exel \cite{exel2011}, respectively. 
The main difference from KS-crossed products is that the natural partial order on $S$ is taken into account for the construction of the crossed products, by considering the $*$-ideal $\mathcal{N}_A$ of $S\alg A$ as follows.

\begin{lem}
  The subspace
  \[
    \mathcal{N}_A:=\Span\{\delta_t\,x-\delta_s\,x\mid t\ge s,\;x\in A_{s^*s}\}
  \]
  is a $*$-ideal of $S\alg A$. 
  Moreover, $\mathcal{N}_A$ is a $*$-ideal generated by elements of the form 
  \[
    \delta_e\,x-\delta_f\,x,
  \]
  where $e\ge f$ in $E(S)$ and $x\in A_f$.
\end{lem}
\begin{proof}
  It is routine to verify this.
\end{proof}

\begin{dfn}
  Let $B$ be a C*-algebra and $\Phi\colon S\alg A\to B$ be a $*$-homomorphism. We say that $\Phi$ is \textit{admissible} if $\Phi(\mathcal{N}_A)=0$, that is, 
  \[
  \Phi(\delta_s\,x)=\Phi(\delta_t\,x)
  \]
  for all $s,t\in S$ with $s\ge t$, and $x\in A_{t^*t}$.
\end{dfn}

\begin{rem}
  From the previous lemma, the definition of admissible $*$-homomorphisms coincides with the one in \cite{exel2008}.
\end{rem}

\begin{dfn}
  For an $S$-C*-algebra $(A,\alpha)$, define a C*-seminorm $\|\cdot\|$ of $S\alg A$ by
  \[
    \|\cdot\|=\sup_\pi\|\pi(\cdot)\|,
  \]
  where the supremum is taken over all admissible $*$-homomorphisms from $S\alg A$ into arbitrary C*-algebras. Let $S\ltimes A$ denote the Hausdorff completion of $S\alg A$ with respect to $\|\cdot\|$. We call this C*-algebra the \textit{full crossed product} of $A$ by $\alpha$. 
\end{dfn}

\begin{rem}
  From the universality of the C*-algebras $S\kstimes A$ and $S\ltimes A$, we have the natural isomorphism
  \[
    S\ltimes A\cong (S\kstimes A)\,/\,\overline{\mathcal{N}_A}.
  \]
\end{rem}

\begin{rem}
  From the definition of admissible $*$-homomorphisms, we have the canonical one-to-one correspondence between admissible $*$-homomorphisms from $S\alg A$ and $*$-homomorphisms from $(S\alg A)\,/\,\mathcal{N}_A$. Hence we can regard $S\ltimes A$ as a universal enveloping C*-algebra of $(S\alg A)\,/\,\mathcal{N}_A$.
\end{rem}

Throughout the paper, let $[\delta_s\,x]$ denote the image of $\delta_s\,x\in S\alg A$ under the canonical map from $S\alg A$ to $S\ltimes A$ for each $s\in S$ and $x\in A_{s^*s}$.

In \cite{exel2011}, Exel introduced a construction of an admissible representation $\pi_{\overline{\varphi}}$ of $S\alg A$ from each pure state $\varphi$ on $A$. In Appendix \ref{sec:A}, we give an almost self-contained description of his construction in terms of our setting, with some detailed arguments concerning GNS-representations of the $*$-algebra $S\alg A$.

\begin{dfn}
  For an $S$-C*-algebra $(A,\alpha)$, define a C*-seminorm $\|\cdot\|_r$ on $S\alg A$ by
  \[
    \|\cdot\|_r=\sup_\varphi\|\pi_{\overline{\varphi}}(\cdot)\|,
  \]
  where the supremum is taken over all pure states $\varphi$ on $A$. Let $S\ltimesr A$ denote the Hausdorff completion of $S\alg A$ with respect to $\|\cdot\|_r$. We call this C*-algebra the \textit{reduced crossed product} of $A$ by $\alpha$.
\end{dfn}

We also denote by $[\delta_s\,x]$ the image of $\delta_s\,x\in S\alg A$ under the canonical map from $S\alg A$ to $S\ltimesr A$ for each $s\in S$ and $x\in A_{s^*s}$.

\subsection{Connection between KS-Crossed Products and Crossed Products}
\label{subsec:1.4}

In Theorem 6.2 of \cite{ks2}, Khoshkam and Skandalis relate their full crossed products and those by Sieben. In this subsection, we review their construction and fix some notation.

Let $S$ be an inverse semigroup and  $(A,\alpha)$ be an $S$-C*-algebra. Then the KS-crossed product $A^\ks:=E(S)\kstimes A$ can be seen as an $S$-C*-algebra in the following way.

\begin{prop}
  \label{prop:action-on-A^ks}
  For each $e_0\in E(S)$, set
  \[
    A^\ks_{e_0}=\overline{\Span}\{\delta_{e}\,x\mid e\le e_0,\;x\in A_{e}\}\subset A^\ks.
  \]
  Then $A^\ks_{e_0}$ is an ideal of $A^\ks$.

  For each $s\in S$, there exists a unique isomorphism $\alpha^\ks_s:A^\ks_{s^*s}\to A^\ks_{ss^*}$ such that
  \[
    \alpha^\ks_s(\delta_e\, x)=\delta_{ses^*}\,\alpha_{se}(x)=\delta_{ses^*}\,\alpha_s(x)
  \]
  for each $s\in S$, $e\le s^*s$ and $x\in A_e$.

  These assignments define an inverse semigroup action $\alpha^\ks$ on $A^\ks=E(S)\kstimes A$.
\end{prop}
\begin{proof}
  It is routine to verify that $A^\ks_{e_0}$ is an ideal of $E(S)\kstimes A$  for each $e_0\in E(S)$.

  Let $B_{e_0}=\Span\{\delta_e\,x\mid e\le e_0,\;x\in A_{e}\}\subset A^\ks$. Note that $B_{e_0}$ is a $*$-subalgebra of $A^\ks_{e_0}$. Define a linear map $\beta_s\colon B_{s^*s}\to B_{ss^*}$ by
  \[
    \beta_s(\delta_e\, x)=\delta_{ses^*}\,\alpha_{se}(x)=\delta_{ses^*}\,\alpha_s(x)
  \]
  for each $s\in S$, $e\le s^*s$ and $x\in A_e$. Since $ses^*\le ss^*$ and $\alpha_{se}(x)\in A_{ses^*}$, $\beta_s$ is well-defined.

  To prove the existence of the map $\alpha_s^\ks$, using Lemma \ref{lem:closed}, it suffices to show that $\beta_s$ is a $*$-homomorphism for each $s\in S$.

  Clearly, $\beta_s$ is $*$-preserving for each $s\in S$. To see that it is multiplicative, compute that 
  \begin{align*}
    \beta_s(\delta_e\,x)\beta_s(\delta_f\,y)
    &=(\delta_{ses^*}\,\alpha_{s}(x))(\delta_{sfs^*}\,\alpha_{s}(y))\\
    &=\delta_{ses^*sfs^*}\,\alpha_{s}(x)\alpha_{s}(y)\\
    &=\delta_{sefs^*}\,\alpha_{s}(xy)\\
    &=\beta_s((\delta_e\,x)(\delta_f\,y))
  \end{align*}
  for each $e,f\le s^*s$, $x\in A_e$ and $y\in A_f$.

  Thus, the map $\beta_s$ uniquely extends to a $*$-homomorphism $\alpha_s^\ks\colon A_{s^*s}^\ks\to A_{ss^*}^\ks$ for each $s\in S$.

  Note that $\alpha^\ks_{s^*}\circ\alpha^\ks_s=\mathrm{id}_{A^\ks_{s^*s}}$ for each $s\in S$. Indeed, for any $e\le s^*s$ and $x\in A_e$, we calculate that
  \begin{align*}
    \alpha^\ks_{s^*}(\alpha^\ks_s(\delta_e\,x))
    =\alpha^\ks_{s^*}(\delta_{ses^*}\,\alpha_s(x))
    =\delta_{s^*ses^*s}\,\alpha_{s^*s}(x)
    =\delta_e\,x.
  \end{align*}
  Thus $\alpha^\ks_s:A^\ks_{s^*s}\to A^\ks_{ss^*}$ is an isomorphism for each $s\in S$.

  To see that these define an inverse semigroup action on $A^\ks$, first observe that $\sum_{e\in E(S)}A^\ks_e$ is dense in $A^\ks$. Hence it remains to show that $\alpha_t^\ks\circ\alpha_s^\ks=\alpha_{ts}^\ks$ as partial maps for each $s,t\in S$.

  We claim that $\alpha^\ks_{s^*}(A^\ks_{ss^*}\cap A^\ks_{t^*t})=A^\ks_{(ts)^*(ts)}$. We first note that
  \[
    A^\ks_{ss^*}\cap A^\ks_{t^*t}=A^\ks_{ss^*t^*t}
  \]
  since $A^\ks_{ss^*}\cap A^\ks_{t^*t}=A^\ks_{ss^*}\cdot A^\ks_{t^*t}$.
  Notice further that, for all $e\in E(S)$ with $e\le ss^*,t^*t$ and $x\in A_e$, we have
  \[
    \alpha^\ks_{s^*}(\delta_e\,x)=\delta_{s^*es}\,\alpha_{s^*}(x),
  \]
  and $s^*es\le (ts)^*(ts)$. Hence $\alpha^\ks_{s^*}(A^\ks_{ss^*}\cap A^\ks_{t^*t})\subset A^\ks_{(ts)^*(ts)}$. Conversely, for any $e\le (ts)^*(ts)$ and $x\in A_e$, it follows that $ses^*\le ss^*, t^*t$ and hence
  \[
    \alpha^\ks_s(\delta_e\,x)=\delta_{ses^*}\,\alpha_s(x)\in A^\ks_{ss^*}\cap A^\ks_{t^*t}.
  \]
  
  It is now routine to show that $\alpha_t^\ks\circ\alpha_s^\ks$ is equal to $\alpha^\ks_{ts}$ as partial maps.
\end{proof}

\begin{thm}[{\cite[Theorem 6.2]{ks2}}]
  \label{thm:isom-of-KS}
  Let $S$ be an inverse semigroup and let $(A,\alpha)$ be an $S$-C*-algebra. Then there exists a unique isomorphism $\Phi\colon S\kstimes A\to S\ltimes A^\ks$ such that
  \[
    \Phi(\delta_s\,x)=[\delta_s\,(\delta_{s^*s}\,x)]
  \]
  for each $s\in S$ and $x\in A_{s^*s}$.
\end{thm}

In \cite{ks2}, this theorem is obtained using the theory of covariant representations for KS-crossed products. In Appendix \ref{sec:B}, we will give a direct algebraic proof without using the theory of covariant representations.

\subsection{Crossed Products of \texorpdfstring{$C_0(X)$}{} as Groupoid C*-algebras}
\label{subsec:1.5}

We have seen that, for a given inverse semigroup action $\theta$ of $S$ on a locally compact Hausdorff space $X$, one can construct a transformation groupoid $S\ltimes X$, full and reduced crossed products $S\ltimes C_0(X)$, $S\ltimesr C_0(X)$, respectively. It is known that these notions are closely related, namely, we have the following isomorphisms.

\begin{thm}
  \label{thm:isom}
  Suppose that an inverse semigroup $S$ acts on a locally compact Hausdorff space $X$ and consider the induced action of $S$ on the C*-algebra $C_0(X)$. Then there exist isomorphisms 
  \[
    S\ltimes C_0(X)\cong C^*(S\ltimes X),\quad S\ltimesr C_0(X)\cong C^*_r(S\ltimes X),
  \]
  which map $[\delta_s\,f]$ to $f\circ d|_{[s,X_{s^*s}]}\in C_c([s,X_{s^*s}])\subset \mathcal{C}(S\ltimes X)$ for all $s\in S$ and $f\in C_c(X_{s^*s})$.
\end{thm}

The reduced case of these isomorphisms is a special case of \cite[Theorem 4.6]{be} by Buss and Exel.
The full case was obtained by Exel in \cite[Theorem 9.8]{exel2008} with the assumptions that $S$ is countable and $X$ is second-countable.
As noted in \cite[Theorem 3.5 and Remark 3.6]{Neshveyev2026}, the full case of Theorem \ref{thm:isom} also follows from, or is closely related to, results in \cite{bm,BHM2018,BKM2025}.
Since Theorem \ref{thm:isom} plays a significant role in this paper, we give a direct proof in Appendix \ref{sec:C}.

\section{Construction of KS-Groupoids}
\label{sec:2}

Let $\theta$ be an action of an inverse semigroup $S$ on a locally compact Hausdorff space $X$.

In this section, we construct a locally compact Hausdorff space $X^\ks$ containing $X$ as a subspace, together with an action of $S$ on $X^\ks$ which extends the $S$-action on $X$. 
When $X$ is the one-point space $\mathbf{pt}$ equipped with the identity action by $S$, the space $X^\ks$ and the $S$-action on $X^\ks$ are nothing but the space $\widehat{E(S)}$ and the $S$-action on $\widehat{E(S)}$, respectively. 
We will prove in Subsection \ref{subsec:4.2} that the corresponding transformation groupoid $S\ltimes X^\ks$, which we call the \textit{KS-groupoid}, is a natural groupoid model for the KS-crossed product $S\kstimes C_0(X)$.

Although some of the results in this section follow immediately from C*-algebraic methods, we provide purely topological proofs in order to give a topological account of the construction.

\subsection{Construction of the Space \texorpdfstring{$X^\ks$}{}}
\label{subsec:2.1}

\begin{dfn}
  \label{def:of-X^ks}
  Let $\widetilde{X^\ks}$ be the set consisting of $\tau\in\prod_{e\in E(S)}\widetilde{X_e}$ satisfying

    \begin{enumerate}
        \item[(1)] if $e_1\le e_2$ and $\tau(e_1)\neq\infty_{e_1}$, then $\tau(e_2)=\tau(e_1)$, and
        \item[(2)] if $\tau(e_1)\neq\infty_{e_1}$ and $\tau(e_2)\neq\infty_{e_2}$, then $\tau(e_1e_2)\neq\infty_{e_1e_2}$,
    \end{enumerate}
    where $\infty_e$ denotes the point at infinity in the one-point compactification\footnote{In this paper, we use the convention that, even when $X_e$ is compact, its one-point compactification is obtained by adding one point $\infty_e$ to $X_e$.} $\widetilde{X_e}$ of $X_e$.
    
    Set $X^\ks=\widetilde{X^\ks}\setminus\{\infty\}$, where $\infty:E(S)\ni e\mapsto \infty_e\in \widetilde{X_e}$.
\end{dfn}

\begin{rem}
  \label{rem:}
  The following hold for $\tau\in\widetilde{X^\ks}$ and $e_1,e_2\in E(S)$.
  \begin{enumerate}
  \item $\tau(e_1)=\tau(e_2)$ if $\tau(e_1)\neq\infty_{e_1}$ and $\tau(e_2)\neq\infty_{e_2}$. 
  \item $\tau(e_1e_2)\neq\infty_{e_1e_2}$ if and only if $\tau(e_1)\neq\infty_{e_1}$ and $\tau(e_2)\neq\infty_{e_2}$.
  \end{enumerate}
\end{rem}

\begin{ex}
  If $X$ is the one-point space $\mathbf{pt}$ equipped with the $S$-action given by $X_e=X$ for each $e\in E(S)$, then $X^\ks=\widehat{E(S)}$ by identifying $\widetilde{X_e}$ with $\{0,1\}$ via the identification of the element $\infty_e\in\widetilde{X_e}$ with $0$ and the unique element in $\mathbf{pt}$ with $1$.
\end{ex}

\begin{lem}
  \label{lem:X^ks-is-lch}
  $\widetilde{X^\ks}$ is compact and Hausdorff with respect to the induced topology inherited from the product topology on $\prod_{e\in E(S)}\widetilde{X_e}$. In particular, $X^\ks$ is locally compact Hausdorff.
\end{lem}
\begin{proof}
  Since $\prod_{e\in E(S)}\widetilde{X_e}$ is compact and Hausdorff, it suffices to show that $\widetilde{X^\ks}$ is closed in $\prod_{e\in E(S)}\widetilde{X_e}$. To see this, let $(\tau_\lambda)_\lambda\subset\widetilde{X^\ks}$ be a net which converges to $\tau$ in $\prod_{e\in E(S)}\widetilde{X_e}$, and we show that $\tau$ satisfies the conditions (1), (2) of Definition \ref{def:of-X^ks}.

  For (1), suppose $e_1\le e_2$ in $E(S)$ and $\tau(e_1)\neq\infty_{e_1}$. Since $(\tau_\lambda(e_1))_\lambda$ converges to $\tau(e_1)$ in $\widetilde{X_{e_1}}$ and $\tau(e_1)\in X_{e_1}$, we may assume that $(\tau_\lambda(e_1))_\lambda\subset X_{e_1}$ by taking a subnet. In particular, $\tau_\lambda(e_1)\neq\infty_{e_1}$ and hence $\tau_\lambda(e_2)=\tau_\lambda(e_1)$ for each $\lambda$. From $\tau(e_2)=\lim_\lambda\tau_\lambda(e_2)$, we obtain $\tau(e_2)=\tau(e_1)$.

  For (2), suppose $\tau(e_1)\neq\infty_{e_1}$ and $\tau(e_2)\neq\infty_{e_2}$. Since $(\tau_\lambda(e_i))_\lambda$ converges to $\tau(e_i)$ in $\widetilde{X_{e_i}}$ for each $i=1,2$, we may assume $\tau_\lambda(e_i)\neq\infty_{e_i}$ for each $\lambda$, by taking a subnet. Then $\tau_\lambda(e_1e_2)\neq\infty_{e_1e_2}$ for each $\lambda$. From the remark above, we have $\tau_\lambda(e_1e_2)=\tau_\lambda(e_1)$. Hence $\tau(e_1e_2)=\lim_\lambda\tau_\lambda(e_1e_2)=\lim_\lambda\tau_\lambda(e_1)=\tau(e_1)\neq\infty_{e_1e_2}$ in $\widetilde{X_{e_1e_2}}$, as desired.
\end{proof}

The following propositions describe the relationship between $X$ and $X^\ks$.

\begin{prop}
  For $x\in X$, define $\tau_x\in\prod_{e\in E(S)}\widetilde{X_e}$ by
  \begin{align*}
    \tau_x(e)=
    \begin{cases}
      x & \mathrm{if}\; x\in X_e\\
      \infty_e & \mathrm{otherwise}.
    \end{cases}
  \end{align*}
  Then $\tau_x\in X^\ks$ and the map $\iota_X\colon X\ni x\mapsto\tau_x\in X^\ks$ is continuous.
\end{prop}
\begin{proof}
  It is easy to check that $\tau_x\in X^\ks$.

  To see $\iota_X$ is continuous, suppose a net $(x_\lambda)_\lambda$ converges to $x$ in $X$. We show that $(\tau_{x_\lambda}(e))_\lambda$ converges to $\tau_x(e)$ in $\widetilde{X_e}$ for each $e\in E(S)$.

  Suppose first that $x\in X_e$, that is, $\tau_x(e)=x$. Fix an open subset $U\subset \widetilde{X_e}$ which contains $x$. Since $U\cap X_e$ is an open subset in $X$ which contains $x$, there exists $\lambda_0$ such that $x_\lambda\in U\cap X_e$ for each $\lambda$ with $\lambda\ge \lambda_0$. Then $\tau_{x_\lambda}(e)=x_\lambda\in U$ for each $\lambda$ with $\lambda\ge\lambda_0$. Thus $\lim_\lambda\tau_{x_\lambda}(e)=\tau_x(e)$ in $X_e$ for each $e\in E(S)$.

  Consider next the case $x\notin X_e$, that is, $\tau_x(e)=\infty_e$ in $\widetilde{X_e}$. Fix an open subset $U\subset\widetilde{X_e}$ containing $\infty_e\in\widetilde{X_e}$. Then $K:=\widetilde{X_e}\setminus U$ is compact in $X_e$, and hence it is closed in $X$. Since $x\in X\setminus K$, there exists $\lambda_0$ such that $x_\lambda\in X\setminus K$ for each $\lambda$ with $\lambda\ge\lambda_0$. For such $\lambda$, we have $x_\lambda\notin X_e$ or $x_\lambda\in U$. It follows $\tau_{x_\lambda}(e)\in U$ for each $\lambda$ with $\lambda\ge\lambda_0$.
\end{proof}

\begin{prop}
  \label{prop:projection}
  Define $p_X:X^\ks\to X$ by $p_X(\tau)=\tau(e)$ for $e\in E(S)$ with $\tau(e)\neq\infty_e$. Then $p_X:X^\ks\to X$ is well-defined and continuous.
\end{prop}
\begin{proof}
  Well-definedness follows from Remark \ref{rem:}.

  Suppose a net $(\tau_\lambda)_\lambda$ converges to $\tau$ in $X^\ks$. Fix $e\in E(S)$ with $\tau(e)\neq \infty_e$. Since $(\tau_\lambda(e))_\lambda$ converges to $\tau(e)=p_X(\tau)$ in $\widetilde{X_e}$, it follows that $(p_X(\tau_\lambda))_\lambda$ converges to $p_X(\tau)$ in $X$.
\end{proof}

\begin{rem}
  \label{rem:embedding}
  Clearly, $p_X\circ\iota_X=\mathrm{id}_X$. It follows that $p_X:X^\ks\to X$ is surjective, $\iota_X:X\to X^\ks$ is a homeomorphism onto its image and the image $\iota_X(X)$ is closed in $X^\ks$.
\end{rem}

\subsection{Construction of an \texorpdfstring{$S$}{}-Action on \texorpdfstring{$X^\ks$}{}}
\label{subsec:2.2}

The $S$-action on $X$ can be extended to an $S$-action on $X^\ks$ in the following way.

\begin{thm}
  \label{thm:action-on-X^ks}
  For each $e\in E(S)$, the subset
  \[
    X^\ks_e=\{\tau\in X^\ks\mid \tau(e)\neq \infty_e\}
  \]
  is open in $X^\ks$, and we have $X^\ks=\bigcup_{e\in E(S)}X^\ks_e$.

  For each $s\in S$, the map $\theta^\ks_s\colon X^\ks_{s^*s}\to X^\ks_{ss^*}$ given by
  \begin{align*}
    \theta^\ks_s(\tau)(e)
      =\begin{cases}
        \theta_{es}(\tau(s^*es)) & \mathrm{if}\;\tau(s^*es)\neq\infty_{s^*es}\\
        \infty_e & \mathrm{otherwise}
      \end{cases}
  \end{align*}
  is a homeomorphism. Moreover, these assignments define an $S$-action on $X^\ks$.
\end{thm}

We note that the openness of each $X^\ks_e$ in $X^\ks$ and the equality $X^\ks=\bigcup_{e\in E(S)}X^\ks_e$ are immediate. 
To proceed with the proof, we need some terminology and lemmas.

\begin{dfn}
  Fix $e\in E(S)$. Let $\mathcal{X}_e$ be the set consisting of all $\omega\in\prod_{f\le e}\widetilde{X_{f}}$ satisfying 
  \begin{enumerate}
    \item $\omega(e)\neq\infty_e$,
    \item if $e_1\le e_2\le e$ and $\omega(e_1)\neq\infty_{e_1}$, then $\omega(e_2)=\omega(e_1)$, and
    \item if $e_1,e_2\le e$, $\omega(e_1)\neq\infty_{e_1}$ and $\omega(e_2)\neq\infty_{e_2}$, then $\omega(e_1e_2)\neq\infty_{e_1e_2}$.
  \end{enumerate}
\end{dfn}

\begin{rem}
  For each $e\in E(S)$, an argument similar to the proof of Lemma \ref{lem:X^ks-is-lch} shows that $\mathcal{X}_e$ is a locally compact Hausdorff space with the induced topology inherited from the product topology on $\prod_{f\le e}\widetilde{X_{f}}$.
\end{rem}

Clearly, there exists a map $\pi^e\colon X^\ks_e\to\mathcal{X}_e$ given by $\pi^e(\tau)(f)=\tau(f)$ for each $f\le e$ in $E(S)$. In the following lemma, we show that this is a homeomorphism.

\begin{lem}
  The following hold for each $e\in E(S)$.
  \begin{enumerate}
    \item For each $f\in E(S)$ the map $\varphi^e_f\colon \mathcal{X}_e\to\widetilde{X_f}$ defined by
    \begin{align*}
      \varphi^e_f(\omega)=
      \begin{cases}
        \omega(ef) & \mathrm{if}\;\omega(ef)\neq\infty_{ef}\\
        \infty_{f} & \mathrm{otherwise}
      \end{cases}
    \end{align*}
    is continuous.
    \item Define a map $\varphi^e\colon \mathcal{X}_e\to X^\ks$ by
    \[
      \varphi^e(\omega)(f)=\varphi^e_f(\omega),\quad f\in E(S).
    \]
    Then $\varphi^e$ is a well-defined, continuous map.
    \item The canonical map $\pi^e\colon X^\ks_e\to \mathcal{X}_e$ is a homeomorphism with the inverse $\varphi^e$.
  \end{enumerate}
\end{lem}
\begin{proof}
  For (i), suppose a net $(\omega_\lambda)_\lambda$ converges to $\omega$ in $\mathcal{X}_e$. It suffices to show that the net $(\varphi^e_f(\omega_\lambda))_\lambda$ converges to $\varphi^e_f(\omega)$ in $\widetilde{X_f}$ for each $f\in E(S)$.

  We only prove the case $\omega(ef)=\infty_{ef}$ since the case where $\omega(ef)\neq\infty_{ef}$ is routine to show. 

  Fix an open subset $U$ in $\widetilde{X_f}$ with $\infty_f\in U$. Then $K:=\widetilde{X_f}\setminus U\subset X_f$ is compact in $X_f$ and hence it is closed in $X$. Suppose first that $\omega(e)\in K\subset X_f$. Then $\omega(e)\in X_e\cap X_f=X_{ef}$. Since $\omega(ef)=\infty_{ef}$, we can find open subsets $V_1,V_2$ of $\widetilde{X_{ef}}$ such that
  \[
    \omega(e)\in V_1,\quad \omega(ef)=\infty_{ef}\in V_2,\;\mathrm{and}\;V_1\cap V_2=\emptyset.
  \]
  Then $V_1$ is open in $X$ and hence there exists $\lambda_1$ such that $\omega_\lambda(e)\in V_1$ for each $\lambda\ge\lambda_1$. On the other hand, we can find $\lambda_2$ such that $\omega_\lambda(ef)\in V_2$ for each $\lambda$ with $\lambda\ge\lambda_2$. Fix $\lambda_0$ with $\lambda_0\ge\lambda_1,\lambda_2$. From $V_1\cap V_2=\emptyset$, we have $\omega_\lambda(ef)=\infty_{ef}$ for each $\lambda\ge\lambda_0$. In particular, $\varphi^e_f(\omega_\lambda)\in U$ for each $\lambda\ge\lambda_0$.

  Suppose next that $\omega(e)\notin K$. Since $K$ is closed in $X$, there exists $\lambda_0$ such that $\omega_\lambda(e)\notin K$ for each $\lambda$ with $\lambda\ge\lambda_0$. Then we clearly have $\varphi^e_f(\omega_\lambda)\in U$.

  (ii): Since $\varphi^e(\omega)(e)=\omega(e)\neq\infty_e$, we have $\varphi^e(\omega)\neq\infty$. To check the condition (1) of Definition \ref{def:of-X^ks}, suppose $f_1\le f_2$ and $\varphi^e(\omega)(f_1)\neq\infty_{f_1}$. Then $\omega(ef_1)\neq\infty_{ef_1}$. It follows from $ef_1\le ef_2\le e$ that $\omega(ef_2)=\omega(ef_1)$. Thus
  \[
    \varphi^e(\omega)(f_2)=\omega(ef_2)=\omega(ef_1)=\varphi^e(\omega)(f_1).
  \]
  For (2) of Definition \ref{def:of-X^ks}, suppose $\varphi^e(\omega)(f_1)\neq\infty_{f_1}$ and $\varphi^e(\omega)(f_2)\neq\infty_{f_2}$. Then $\omega(ef_1)\neq\infty_{ef_1}$ and $\omega(ef_2)\neq\infty_{ef_2}$. Thus $\omega(ef_1f_2)\neq\infty_{ef_1f_2}$ and hence $\varphi^e(\omega)(f_1f_2)\neq\infty_{f_1f_2}$. This proves $\varphi^e(\omega)\in X^\ks_e$. By (i), the map $\varphi^e\colon \mathcal{X}_e\to X^\ks_e$ is continuous. 

  (iii): Note that $\varphi^e_f(\omega)=\omega(f)$ for all $f\le e$ and $\omega\in\mathcal{X}_e$. This implies that $\pi^e\circ\varphi^e=\mathrm{id}_{\mathcal{X}_e}$.
  
  Using Remark \ref{rem:}, we compute for $\tau\in X_e^\ks$ that
  \begin{align*}
    (\varphi^e\circ\pi^e(\tau))(f)
    &=\begin{cases}
      \pi^e(\tau)(ef) & \mathrm{if}\;\pi^e(\tau)(ef)\neq\infty_{ef}\\
      \infty_f & \mathrm{otherwise}
    \end{cases}\\
    &=\begin{cases}
      \tau(ef) & \mathrm{if}\;\tau(ef)\neq\infty_{ef}\\
      \infty_f & \mathrm{otherwise}
    \end{cases}\\
    &=\tau(f)
  \end{align*}
  Hence $\varphi^e\circ\pi^e=\mathrm{id}_{X^\ks_e}$. 
\end{proof}

\begin{lem}
  The following hold.
  \begin{enumerate}
    \item $\theta_s^\ks\colon X_{s^*s}^\ks\to X^\ks_{ss^*}$ is well-defined.
    \item For $s\in S$, define $\vartheta_s\colon\mathcal{X}_{s^*s}\to\mathcal{X}_{ss^*}$ by
    \begin{align*}
      \vartheta_s(\omega)(e)
      =\widetilde{\theta_{es}}(\omega(s^*es)),\quad e\le ss^*
    \end{align*}
    where $\widetilde{\theta_{es}}\colon\widetilde{X_{s^*es}}\to\widetilde{X_e}$ is the canonical extension of the homeomorphism $\theta_{es}\colon X_{s^*es}\to X_{e}$. Then $\vartheta_s\colon\mathcal{X}_{s^*s}\to\mathcal{X}_{ss^*}$ is a well-defined homeomorphism.
    \item The diagram
    \begin{center}
      \begin{tikzpicture}[auto]
        \node (1) at (0,1.5) {$X_{s^*s}^\ks$};
        \node (2) at (2,1.5) {$X_{ss^*}^\ks$};
        \node (3) at (0,0) {$\mathcal{X}_{s^*s}$};
        \node (4) at (2,0) {$\mathcal{X}_{ss^*}$};
        \draw[->] (1) to node {$\scriptstyle\theta_s^\ks$} (2);
        \draw[->] (1) to node[swap] {$\scriptstyle\pi^{s^*s}$} (3);
        \draw[->] (2) to node {$\scriptstyle\pi^{ss^*}$} (4);
        \draw[->] (3) to node[swap] {$\scriptstyle\vartheta_s$} (4);
      \end{tikzpicture}
    \end{center}
    commutes. In particular, $\theta_s^\ks\colon X_{s^*s}^\ks\to X_{ss^*}^\ks$ is a homeomorphism for each $s\in S$.
  \end{enumerate}
\end{lem}
\begin{proof}
  (i): We first check that $\theta_s^\ks(\tau)$ satisfies the conditions (1) and (2) of Definition \ref{def:of-X^ks}. 
  
  For (1), suppose $e_1\le e_2$ in $E(S)$ and $\theta_s^\ks(\tau)(e_1)\neq\infty_{e_1}$. By the definition of $\theta_s^\ks(\tau)(e_1)$, we have $\tau(s^*e_1s)\neq\infty_{s^*e_1s}$. It follows from $s^*e_1s\le s^*e_2s$ that $\tau(s^*e_2s)\neq\infty_{s^*e_2s}$ and $\tau(s^*e_2s)=\tau(s^*e_1s)$. Thus 
  \[
    \theta_s^\ks(\tau)(e_2)=\theta_{e_2s}(\tau(s^*e_2s))=\theta_{e_1s}(\tau(s^*e_1s))=\theta_s^\ks(\tau)(e_1).
  \]

  For (2), suppose $\theta_s^\ks(\tau)(e_1)\neq\infty_{e_1}$ and $\theta_s^\ks(\tau)(e_2)\neq\infty_{e_2}$. Then $\tau(s^*e_1s)\neq \infty_{s^*e_1s}$ and $\tau(s^*e_2s)\neq\infty_{s^*e_2s}$. Hence $\tau(s^*e_1e_2s)=\tau((s^*e_1s)(s^*e_2s))\neq\infty_{s^*e_1e_2s}$, from which it follows $\theta_s^\ks(\tau)(e_1e_2)\neq\infty_{e_1e_2}$. This proves $\theta_s^\ks(\tau)\in\widetilde{X^\ks}$.

  Since $\theta_s^\ks(\tau)(ss^*)=\theta_{s}(\tau(s^*s))\neq\infty_{ss^*}$, we have $\theta_s^\ks(\tau)\in X^\ks_{ss^*}$, as desired.

  (ii): It is routine to check that $\vartheta_s(\omega)\in\mathcal{X}_{ss^*}$ for all $\omega\in\mathcal{X}_{s^*s}$. The continuity of $\vartheta_s$ is clear by the definition of $\vartheta_s$. It is also easy to see $\vartheta_{s^*}\circ\vartheta_s=\mathrm{id}_{\mathcal{X}_{s^*s}}$ for each $s\in S$.

  (iii): For $\tau\in X^\ks_{s^*s}$ and $e\le ss^*$, compute that
  \begin{align*}
    (\vartheta_s\circ\pi^{s^*s}(\tau))\,(e)
    &=\widetilde{\theta_{es}}(\pi^{s^*s}(\tau)(s^*es))\\
    &=\begin{cases}
      \theta_{es}(\tau(s^*es)) & \mathrm{if}\;\tau(s^*es)\neq\infty_{s^*es}\\
      \infty_e & \mathrm{otherwise}
    \end{cases}\\
    (\pi^{ss^*}\circ\theta_s^\ks(\tau))\,(e)
    &=\theta_s^\ks(\tau)(e)\\
    &=\begin{cases}
      \theta_{es}(\tau(s^*es)) & \mathrm{if}\;\tau(s^*es)\neq\infty_{s^*es}\\
      \infty_e & \mathrm{otherwise}.
    \end{cases}
  \end{align*}
  Hence $\vartheta_s\circ\pi^{s^*s}=\pi^{ss^*}\circ\theta_s^\ks$, as desired.
\end{proof}

\hspace{-11pt}\textbf{Proof of Theorem \ref{thm:action-on-X^ks}:}
  It remains to show that $\theta_t^\ks\circ\theta_s^\ks=\theta_{ts}^\ks$ as partial maps for each $s,t\in S$. We first show their domains, $\theta_{s^*}^\ks(X^\ks_{ss^*}\cap X^\ks_{t^*t})$ and $X^\ks_{(ts)^*(ts)}$ coincide.

  Fix $\tau\in X^\ks_{ss^*}\cap X^\ks_{t^*t}$. Observe that
  \[
    \tau(s((ts)^*(ts))s^*)=\tau(ss^*\cdot t^*t)\neq\infty_{ss^*t^*t}.
  \]
  Hence $\theta^\ks_{s^*}(\tau)\in X^\ks_{(ts)^*(ts)}$.

  Conversely, take $\tau\in X^\ks_{(ts)^*(ts)}\subset X^\ks_{s^*s}$. Then $\theta_s^\ks(\tau)\in X^\ks_{ss^*}$ and  
  \[
    \theta^\ks_s(\tau)(t^*t)=\theta_{t^*ts}(\tau((ts)^*(ts)))\neq\infty_{t^*t}.
  \]
  Thus $\theta^\ks_s(\tau)\in X^\ks_{ss^*}\cap X^\ks_{t^*t}$. Since $\theta^\ks_{s^*}$ is the inverse of $\theta_s^\ks$, it follows that $\tau\in\theta^\ks_{s^*}(X^\ks_{ss^*}\cap X^\ks_{t^*t})$. This proves $\theta^\ks_{s^*}(X^\ks_{ss^*}\cap X^\ks_{t^*t})=X^\ks_{(ts)^*(ts)}$.

  For $\tau\in X^\ks_{(ts)^*(ts)}=\theta_{s^*}^\ks(X^\ks_{ss^*}\cap X^\ks_{t^*t})$, we compute that

  \begin{align*}
    (\theta_t^\ks\circ\theta_s^\ks(\tau))\,(e)
    &=\begin{cases}
      \theta_{et}(\theta_s^\ks(\tau)(t^*et)) & \mathrm{if}\;\theta_s^\ks(\tau)(t^*et)\neq\infty_{t^*et}\\
      \infty_e & \mathrm{otherwise}
    \end{cases}\\
    &=\begin{cases}
      \theta_{et}\bigl(\theta_{t^*ets}(\tau(s^*t^*ets))\bigr) & \mathrm{if}\;\tau(s^*t^*ets)\neq\infty_{s^*t^*ets}\\
      \infty_e & \mathrm{otherwise}
    \end{cases}\\
    &=\begin{cases}
      \theta_{ets}(\tau((ts)^*e(ts))) & \mathrm{if}\;\tau((ts)^*e(ts))\neq\infty_{(ts)^*e(ts)}\\
      \infty_e & \mathrm{otherwise}\\
    \end{cases}\\
    &=\theta_{ts}^\ks(\tau)\,(e)
  \end{align*}
  for each $e\in E(S)$. This completes the proof.\qed

  \begin{dfn}
    Let $\theta$ be an action of an inverse semigroup $S$ on a locally compact Hausdorff space $X$. We call the transformation groupoid $S\ltimes X^\ks$ the \textit{KS-groupoid}.
  \end{dfn}

\begin{prop} The following hold.
  \begin{enumerate}
    \item The continuous map $\iota_X\colon X\to X^\ks$ is $S$-equivariant and satisfies $\iota_X(X_e)=X_e^\ks\cap\iota_X(X)$ for each $e\in E(S)$.
    \item The continuous map $p_X\colon X^\ks\to X$ is $S$-equivariant.
  \end{enumerate}
\end{prop}
\begin{proof}
  (i): Note that $\tau_x(e)=x$ if and only if $x\in X_e$. Hence $\iota_X(X_e)=X^\ks_e\cap\iota_X(X)$.
  
  To show that the diagram 
  \begin{center}
    \begin{tikzpicture}[auto]
      \node (1) at (0,1.5) {$X_{s^*s}$};
      \node (2) at (2,1.5) {$X_{ss^*}$};
      \node (3) at (0,0) {$X_{s^*s}^\ks$};
      \node (4) at (2,0) {$X_{ss^*}^\ks$};
      \draw[->] (1) to node {$\scriptstyle\theta_s$} (2);
      \draw[->] (1) to node[swap] {$\scriptstyle\iota_X|_{X_{s^*s}}$} (3);
      \draw[->] (2) to node {$\scriptstyle\iota_X|_{X_{ss^*}}$} (4);
      \draw[->] (3) to node[swap] {$\scriptstyle\theta_s^\ks$} (4);
    \end{tikzpicture}
  \end{center}
  commutes for each $s\in S$, we compute that
  \begin{align*}
    (\theta_s^\ks\circ\iota_X(x))(e)
    &=\theta_s^\ks(\tau_x)(e)
    =\begin{cases}
      \theta_{es}(x) & \mathrm{if}\;x\in X_{s^*es}\\
      \infty_e & \mathrm{otherwise}
    \end{cases}\\
    (\iota_X\circ\theta_s(x))(e)
    &=\tau_{\theta_s(x)}(e)
    =\begin{cases}
      \theta_s(x) & \mathrm{if}\;\theta_s(x)\in X_e\\
      \infty_e & \mathrm{otherwise}
    \end{cases}
  \end{align*}
  for any $x\in X_{s^*s}$ and $e\in E(S)$. Note that $x\in X_{s^*es}$ if and only if $\theta_s(x)\in X_e$, and in this case, $\theta_{es}(x)=\theta_s(x)$.

  (ii): For any $\tau\in X^\ks_e$, we have $p_X(\tau)=\tau(e)\in X_e$. Hence $p_X(X_e^\ks)\subset X_e$.

  To show that the diagram 
  \begin{center}
    \begin{tikzpicture}[auto]
      \node (1) at (0,1.5) {$X_{s^*s}^\ks$};
      \node (2) at (2,1.5) {$X_{ss^*}^\ks$};
      \node (3) at (0,0) {$X_{s^*s}$};
      \node (4) at (2,0) {$X_{ss^*}$};
      \draw[->] (1) to node {$\scriptstyle\theta^\ks_s$} (2);
      \draw[->] (1) to node[swap] {$\scriptstyle p_X|_{X_{s^*s}^\ks}$} (3);
      \draw[->] (2) to node {$\scriptstyle p_X|_{X_{ss^*}^\ks}$} (4);
      \draw[->] (3) to node[swap] {$\scriptstyle\theta_s$} (4);
    \end{tikzpicture}
  \end{center}
  commutes for each $s\in S$, compute for $\tau\in X_{s^*s}^\ks$ that
  \begin{align*}
    \theta_s\circ p_X(\tau)
    &=\theta_s(\tau(s^*s))\\
    p_X\circ\theta_s^\ks(\tau)
    &=\theta_s^\ks(\tau)(ss^*)=\theta_s(\tau(s^*s)).\qedhere
  \end{align*}
\end{proof}

\begin{rem}
  \label{rem:compact}
  Clearly, $X_e^\ks=\emptyset$ if and only if $X_e=\emptyset$. We further note that $X_e^\ks$ is compact in $X^\ks$ if and only if $X_e$ is compact in $X$. Indeed, $p_X|_{X_e^\ks}\colon X_e^\ks\to X_e$ is a continuous surjection and $X^\ks_e=q_e^{-1}(X_e)$, where $q_e\colon\widetilde{X^\ks}\to\widetilde{X_e}$ denotes the evaluation at $e\in E(S)$.
\end{rem}

\section{Some Properties of KS-Groupoids}
\label{sec:3}

\subsection{Topological Properties of KS-groupoids}
\label{subsec:3.1}

In this subsection, we study some topological properties of KS-groupoids, such as Hausdorffness, ampleness, second countability and ($\sigma$-)compactness of the unit space.

\noindent\textbf{Hausdorffness.}
We first investigate when the KS-groupoid $S\ltimes X^\ks$ becomes Hausdorff, provided that each open subset $X_e$ is compact in $X$. For the universal groupoid $S\ltimes\widehat{E(S)}$ associated with an inverse semigroup $S$, which is a KS-groupoid associated with the identity action of $S$ on the one-point space $\mathbf{pt}$, Hausdorffness is completely characterized by Steinberg in \cite[Theorem 5.17]{steinberg}. His strategy can be applied to the KS-groupoid $S\ltimes X^\ks$. However, we choose a different approach which we believe is more natural. Our argument is based on the following slightly improved version of the characterization of Hausdorffness for transformation groupoids $S\ltimes X$ due to Exel and Pardo in \cite[Theorem 3.15]{ep}.

\begin{dfn}
  For an inverse semigroup $S$ and $t\in S$, set
  \[
    t^\downarrow=\{s\in S\mid s\le t\}.
  \]
\end{dfn}

\begin{prop}
  \label{prop:hausdorffness}
  For any action of an inverse semigroup $S$ on a locally compact Hausdorff space $X$, the following are equivalent.
  \begin{enumerate}
    \item The transformation groupoid $S\ltimes X$ is Hausdorff.
    \item For any $s,t\in S$, the subset $\bigcup_{u\in s^\downarrow\cap t^\downarrow}X_{u^*u}$ is closed in $X_{s^*s}\cap X_{t^*t}$.
    \item For any $s\in S$, the subset $\bigcup_{e\in E(S)\cap s^\downarrow}X_e$ is closed in $X_{s^*s}$.
  \end{enumerate}
  If each $X_e$ is compact in $X$, these conditions are equivalent to 
  \begin{enumerate}
    \item[(ii)'] For any $s,t\in S$, there exists a finite subset $F$ of $s^\downarrow\cap t^\downarrow$ such that
    \[
      \bigcup_{v\in F}X_{v^*v}=\bigcup_{u\in s^\downarrow\cap t^\downarrow}X_{u^*u}.
    \]
  \end{enumerate}
\end{prop}
\begin{proof}
  (i)$\Rightarrow$(ii): Suppose that a net $(x_\lambda)_\lambda\subset\bigcup_{u\in s^\downarrow\cap t^\downarrow}X_{u^*u}$ converges to $x$ in $X_{s^*s}\cap X_{t^*t}$. Then we have $[s,x_\lambda]=[t,x_\lambda]$ for each $\lambda$. Moreover, the nets $([s,x_\lambda])_\lambda$ and $([t,x_\lambda])_\lambda$ converge to $[s,x]$ and $[t,x]$ in $S\ltimes X$, respectively. From the assumption that $S\ltimes X$ is Hausdorff, it follows $[s,x]=[t,x]$. Hence there exists $u\in s^\downarrow\cap t^\downarrow$ such that $x\in X_{u^*u}$, which proves (ii).

  \noindent
  (ii)$\Rightarrow$(iii): Choose $t=s^*s$.

  \noindent
  (iii)$\Rightarrow$(i): See Theorem 3.15 of \cite{ep}.

  Suppose that $X_e$ is compact in $X$ for each $e\in E(S)$. It is clear that (ii)' implies (ii). To see the converse, first note that $X_{s^*s}\cap X_{t^*t}$ is compact in $X$. By (ii), it follows that $\bigcup_{u\in s^\downarrow\cap t^\downarrow}X_{u^*u}$ is also compact in $X$. Hence we can find a finite subset $F\subset s^\downarrow\cap t^\downarrow$ with the desired property.
\end{proof}

For a subset $P\subset S$, we say $P$ is a \textit{weak semilattice} if for any $s,t\in P$ there exists a finite subset $F\subset P\cap s^\downarrow\cap t^\downarrow$ such that $P\cap s^\downarrow\cap t^\downarrow\subset \bigcup_{u\in F}u^\downarrow$.

Given an inverse semigroup action $S$ on $X$, let $S^X:=\{s\in S\mid X_{s^*s}\neq\emptyset\}$.

For each $e\in E(S)$ and $x\in X_e$, define $\tau_{e,\,x}\in X^\ks$ by
\begin{align*}
  \tau_{e,\,x}(f)=
  \begin{cases}
    x & \mathrm{if}\;e\le f\\
    \infty_f & \mathrm{otherwise}.
  \end{cases}
\end{align*}
Clearly, $\tau_{e,\,x}\in X^\ks_e$.

\begin{thm}
  \label{thm:hausdorffness-of-KS}
  For any action of an inverse semigroup $S$ on a locally compact Hausdorff space $X$ such that each $X_e$ is compact, the KS-groupoid $S\ltimes X^\ks$ is Hausdorff if and only if $S^X=\{s\in S\mid X_{s^*s}\neq\emptyset\}$ is a weak semilattice.
\end{thm}
\begin{proof}
  Observe first that since $X_e$ is compact in $X$, we have $X^\ks_e$ is compact in $X^\ks$ by Remark \ref{rem:compact}.

  Suppose that $S^X$ is a weak semilattice. To see that $S\ltimes X^\ks$ is Hausdorff, it suffices to check the condition (ii)' of the previous proposition. This is straightforward.

  Conversely, suppose $S\ltimes X^\ks$ is Hausdorff. Fix $s,t\in S^X$. Then, by (ii)' of the previous proposition, there exists a finite subset $F\subset S^X\cap s^\downarrow\cap t^\downarrow$ such that
  \[
    \bigcup_{v\in F}X^\ks_{v^*v}=\bigcup_{u\in s^\downarrow\cap t^\downarrow}X^\ks_{u^*u}.
  \]
  We claim that $S^X\cap s^\downarrow\cap t^\downarrow\subset F^\downarrow$. Fix $u\in S^X\cap s^\downarrow\cap t^\downarrow$. From $X_{u^*u}\neq\emptyset$, we can find $x\in X_{u^*u}$. Then we have $\tau_{u^*u,\,x}\in X^\ks_{u^*u}$. Using the equality above, there exists $v\in F$ such that $\tau_{u^*u,\,x}\in X^\ks_{v^*v}$. Then $u^*u\le v^*v$ by the definition of $\tau_{u^*u,\,x}$. Compute that
  \[
    vu^*u=sv^*vu^*u=su^*u=u,
  \]
  where we used $u,v\le s$. Hence $u\in F^\downarrow$ for each $u\in S^X\cap s^\downarrow\cap t^\downarrow$, which proves the theorem.
\end{proof}

As shown in the following examples, if $X_e$ is not compact in $X$ for some $e\in E(S)$, the condition that $S^X$ is a weak semilattice is neither necessary nor sufficient for Hausdorffness of the KS-groupoid $S\ltimes X^\ks$.

\begin{ex}
  \label{ex1}
  Let $S=\{0\}\cup\{e_n\mid n\in\mathbb{Z}\}\cup\{f\}\cup \{s\}$ be the inverse semigroup with zero $0\in S$ whose multiplication is given by 
  \begin{enumerate}
    \item $e_ne_n=e_n$ for each $n\in\mathbb{Z}$ and $ff=f$,
    \item $e_ne_m=0$ for each $n,m\in\mathbb{Z}$ with $n\neq m$,
    \item $e_nf=e_n$ for each $n\in\mathbb{Z}$, 
    \item $ss=f$, $sf=s=f s$, and
    \item $se_n=e_n=e_ns$ for each $n\in\mathbb{Z}$.
  \end{enumerate}
  Note that $E(S)=\{0\}\cup\{e_n\mid n\in\mathbb{Z}\}\cup\{f\}$ and the generalized inverse of $s\in S$ is $s$ itself. Let $X$ be the locally compact Hausdorff space $\mathbb{R}\setminus\mathbb{Z}$ and define the $S$-action on $X$ by $X_{0}=\emptyset$, $X_{e_n}=(n,n+1)$ for each $n\in\mathbb{Z}$, $X_{f}=X$, and $\theta_s=\mathrm{id}_X$. Note that $X_e$ is closed in $X$ but it is not compact for each $e\in E(S)\setminus\{0\}$. The inverse semigroup $S$ and the space $X^\ks$ can be visualized as follows:

  
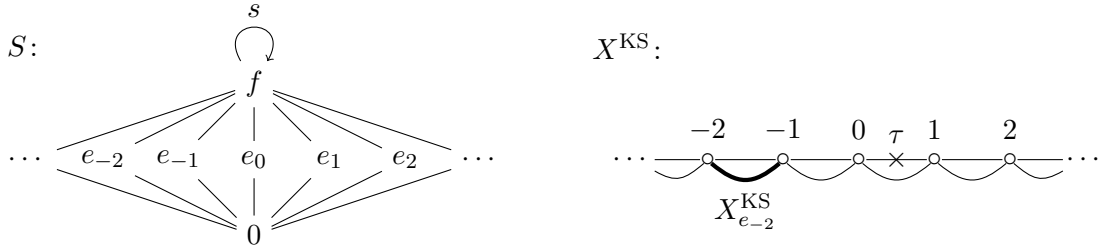
\begin{center}
    \begin{tikzpicture}[auto]
      \node at (-4,2.5) {$S\colon$};
      \node (0) at (-1,0) {$0$};
      \node(-3) at (-4,1) {$\cdots$};
      \node(-2) at (-3,1) {$e_{-2}$};
      \node (-1) at (-2,1) {$e_{-1}$};
      \node (e_0) at (-1,1) {$e_0$};
      \node (1) at (0,1) {$e_1$};
      \node (2) at (1,1) {$e_2$};
      \node (3) at (2,1) {$\cdots$};
      \node (infty) at (-1,2) {$f$};
      \draw[-] (0) to node {} (-2); \draw (0) to (-3);
      \draw[-] (0) to node {} (-1);
      \draw[-] (0) to node {} (e_0);
      \draw[-] (0) to node {} (1);
      \draw[-] (0) to node {} (2); \draw (0) to (3);
      \draw[-] (-2) to node {} (infty); \draw[-] (-3) to (infty);
      \draw[-] (-1) to node {} (infty);
      \draw[-] (e_0) to node {} (infty);
      \draw[-] (1) to node {} (infty);
      \draw[-] (2) to node {} (infty); \draw[-] (3) to (infty);
      \draw[->] (infty) to[out=120, in=60, looseness=5] node {$s$} (infty);
      \node at (4,2.5) {$X^\ks\colon$};
      \node(-22) at (4,1) {$\cdots$}; 
      \node(-11) at (5,1) {$\circ$}; \node at (5,1.4) {$-2$};
      \node at (5.5,0.3) {$X^\ks_{e_{-2}}$};
      \node(00) at (6,1) {$\circ$}; \node at (6,1.4) {$-1$};
      \node(11) at (7,1) {$\circ$}; \node at (7,1.4) {$0$}; 
      \node at (7.5,1) {$\times$}; \node at (7.5,1.3) {$\tau$};
      \node(22) at (8,1) {$\circ$}; \node at (8,1.4) {$1$};
      \node(33) at (9,1) {$\circ$}; \node at (9,1.4) {$2$};
      \node(44) at (10,1) {$\cdots$};
      \draw[-] (4.3,0.85) to[out=-30, in=-140, looseness=1.3] (4.95,0.95);
      \draw[-] (4.3,1) to (4.9,1); 
      \draw[ultra thick] (5.05,0.95) to[out=-40, in=-140, looseness=1.3] (5.95,0.95);
      \draw[-] (5.1,1) to (5.9,1);
      %\draw[semithick] (5.43,0.65) to (5.43,0.35); \draw[semithick] (5.58,0.65) to (5.58,0.35);
      \draw[-] (6.05,0.95) to[out=-40, in=-140, looseness=1.3] (6.95,0.95);
      \draw[-] (6.1,1) to (6.9,1);
      \draw[-] (7.05,0.95) to[out=-40, in=-140, looseness=1.3] (7.95,0.95);
      \draw[-] (7.1,1) to (7.9,1);
      \draw[-] (8.05,0.95) to[out=-40, in=-140, looseness=1.3] (8.95,0.95);
      \draw[-] (8.1,1) to (8.9,1);
      \draw[-] (9.05,0.95) to[out=-40, in=-150, looseness=1.3] (9.7,0.85);
      \draw[-] (9.1,1) to (9.7,1);
    \end{tikzpicture}
  \end{center}
  In this figure, $\tau$ is the element in $X^\ks$ such that $\tau(f)=1/2$ and $\tau(e_n)=\infty_{e_n}$ for each $n\in\mathbb{Z}$. Clearly, $S^X$ is not a weak semilattice since $S^X\cap s^\downarrow\cap f^\downarrow=\{e_n\mid n\in\mathbb{Z}\}$. On the other hand, the subset $\bigcup_{n\in\mathbb{Z}}X^\ks_{e_n}$ is closed in $X^\ks$ and this implies the condition (ii) of Proposition \ref{prop:hausdorffness} for $S\ltimes X^\ks$. Hence $S\ltimes X^\ks$ is Hausdorff.
\end{ex}

\begin{ex}
  \label{ex2}
  Let $S=\{0,e,s\}$ be the inverse semigroup with zero $0\in S$ whose multiplication is given by $ee=e$, $es=se=s$ and $ss=e$. Note that $E(S)=\{0,e\}$ and the generalized inverse of $s\in S$ is $s$ itself. Let $X=\mathbb{R}$ and define the $S$-action on $X$ by $X_0=(-1,1)$, $X_e=X$ and $\theta_s=\mathrm{id}_X$. The inverse semigroup $S$ and the space $X^\ks$ can be visualized as follows:
  
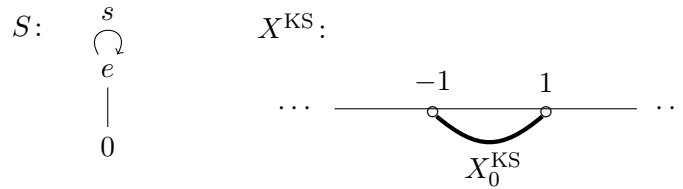
\begin{center}
    \begin{tikzpicture}[auto]
      \node at (-1,1.6) {$S\colon$};
      \node (0) at (0,0) {$0$};
      \node (1) at (0,1) {$e$};
      \draw (0) to (1);
      \draw[->] (1) to[out=120, in=60, looseness=5] node {$s$} (1);
      \draw (3,0.5) to (7,0.5);
      \node at (2.5,1.6) {$X^\ks\colon$};
      \node at (2.5,0.5) {$\cdots$}; \node at (7.5,0.5) {$\cdots$};
      \node at (4.3,0.45) {$\circ$}; \node at (4.3,0.85) {$-1$};
      \node at (5.8,0.45) {$\circ$}; \node at (5.8,0.85) {$1$};
      \draw[ultra thick] (4.35,0.4) to[out=-40, in=-140, looseness=1.3] (5.75,0.4);
      \node at (5.1,-0.3) {$X^\ks_0$};
    \end{tikzpicture}
  \end{center}
  Clearly, $S^X=S$ is a weak semilattice. On the other hand, the subset $X^\ks_0$ is not closed in $X^\ks$ and this implies that the condition (ii) of Proposition \ref{prop:hausdorffness} does not hold for $S\ltimes X^\ks$. Hence $S\ltimes X^\ks$ is not Hausdorff.
\end{ex}

\noindent\textbf{Ampleness.}
For any inverse semigroup $S$, the universal groupoid $S\ltimes\widehat{E(S)}$ is \textit{ample}, in the sense that its unit space $\widehat{E(S)}$ is totally disconnected. On the other hand, KS-groupoids $S\ltimes X^\ks$ need not be ample (for example, see Examples \ref{ex1} and \ref{ex2} above). In the following, we give a natural characterization of ampleness of KS-groupoids.

\begin{prop}
  Let $\theta$ be an action of an inverse semigroup $S$ on a locally compact Hausdorff space $X$. Then the KS-groupoid $S\ltimes X^\ks$ is ample if and only if $X$ is totally disconnected.
\end{prop}
\begin{proof}
  If $X^\ks$ is totally disconnected, then $X$ is also totally disconnected since $X$ is a subspace of $X^\ks$ by Remark \ref{rem:embedding}. 

  Suppose that $X$ is totally disconnected. 
  To see $X^\ks$ is also totally disconnected, observe that for any non-empty connected subset $C$ of $X^\ks$, we can find $x\in X$ with $C\subset p_X^{-1}(\{x\})$ since $p_X(C)$ is a non-empty connected subset in the totally disconnected space $X$.  
  Hence it suffices to show that $p_X^{-1}(\{x\})$ is totally disconnected. 
  This follows from the fact that $p_X^{-1}(\{x\})$ is canonically homeomorphic to a subspace of $\{0,1\}^{E(S)}$, which is clearly totally disconnected.
\end{proof}

\noindent\textbf{Second Countability.}
We next characterize the second countability condition on the KS-groupoid $S\ltimes X^\ks$ in terms of the underlying inverse semigroup action.

\begin{prop}
  \label{prop:2nd-ctbl}
  Let $\theta$ be an action of an inverse semigroup $S$ on a locally compact Hausdorff space $X$. Then the KS-groupoid $S\ltimes X^\ks$ is second countable if and only if $X$ is second countable and $S^X=\{s\in S\mid X_{s^*s}\neq\emptyset\}$ is countable.
\end{prop}
\begin{proof}
  Suppose first that $X$ is second countable and $S^X$ is countable. Then $\prod_{e\in E(S)}\widetilde{X_e}$ is second countable and hence $X^\ks$ is also second countable. 
  Since $S\ltimes X^\ks=\bigcup_{s\in S^X}[s,X_{s^*s}^\ks]$ and $[s,X_{s^*s}^\ks]$ is homeomorphic to $X^\ks_{s^*s}$, it follows that $S\ltimes X^\ks$ is second countable.

  Conversely, suppose $S\ltimes X^\ks$ is second countable. Then the unit space $X^\ks=(S\ltimes X^\ks)^{(0)}$ is second countable. 
  Since $X$ is homeomorphic to some subspace of $X^\ks$ by Remark \ref{rem:embedding}, second countability passes to $X$. 

  For each $x\in X$, let $E(S)^x=\{e\in E(S)\mid x\in X_e\}$. We claim that $E(S)^x$ is countable for all $x\in X$. 
  To see this, observe that $p_X^{-1}(\{x\})$ is second countable and $p_X^{-1}(\{x\})\cap X_e^\ks=q_e^{-1}(\{x\})$ is compact open subset in $p_X^{-1}(\{x\})$ for each $e\in E(S)^x$, where $q_e\colon\widetilde{X^\ks}\to\widetilde{X_e}$ denotes the evaluation at $e\in E(S)^x$. 
  Note further that $p_X^{-1}(\{x\})\cap X_e^\ks\neq p_X^{-1}(\{x\})\cap X_f^\ks$ for each $e\neq f$ in $E(S)^x$. 
  Since every second countable space has at most countably many compact open subsets, it follows that $E(S)^x$ is countable.

  Take any countable, dense subset $F$ of $X$. Then we notice that 
  \[
    E(S)^X:=\{e\in E(S)\mid X_e\neq\emptyset\}=\bigcup_{x\in F}E(S)^x.
  \]
  From $S^X=\bigcup_{e\in E(S)^X}S_e$, it remains to show $S_e$ is countable for each $e\in E(S)^X$. To this end, observe that $S_e\ni s\mapsto [s,\tau_{e,\,x}]\in S\ltimes X^\ks$ is injective for any fixed $x\in X_e$ and its range $\{[s,\tau_{e,\,x}]\mid s\in S_e\}=(S\ltimes X^\ks)_{\tau_{e,\,x}}$ (see Lemma \ref{lem:domain}) is a discrete subset in $S\ltimes X^\ks$. Since every discrete subset in a second countable space is countable, $S_e$ is countable for each $e\in E(S)^X$. This completes the proof.
\end{proof}

\noindent\textbf{($\sigma$-)Compactness of the Unit Space}

\begin{lem}
  Let $\theta$ be an action of an inverse semigroup $S$ on a locally compact Hausdorff space $X$. Suppose that $\{(f_i,K_i)\}_{i\in I}$ is a family of pairs $(f_i,K_i)$ such that $f_i\in E(S)$ and $K_i\subset X_{f_i}$ is compact for each $i\in I$. Then the following are equivalent.
  \begin{enumerate}
    \item $X^\ks=\bigcup_{i\in I}q_{f_i}^{-1}(K_i)$, where $q_e\colon\widetilde{X^\ks}\to\widetilde{X_e}$ denotes the evaluation at $e\in E(S)$.
    \item $X_e\subset\bigcup_{i\in I,\,f_i\ge e} K_i$ for each $e\in E(S)$.
  \end{enumerate}
\end{lem}
\begin{proof}
  (i)$\Rightarrow$(ii): Let $x\in X_e$. By (i), $\tau_{e,\,x}$ belongs to $q_{f_i}^{-1}(K_i)$ for some $i\in I$. Then $f_i\ge e$ and $x\in K_i$. Thus $X_e\subset\bigcup_{i\in I,\,f_i\ge e}K_i$.

  \noindent
  (ii)$\Rightarrow$(i): Let $\tau\in X^\ks$ and fix $e\in E(S)$ with $x:=\tau(e)\neq\infty_e$. Then $x\in X_e$ and hence, by (ii), we have $x\in K_i$ for some $i\in I$ with $f_i\ge e$. Then $\tau(f_i)=x\in K_i$, that is, $\tau\in q_{f_i}^{-1}(K_i)$.
\end{proof}

\begin{thm}
  \label{thm:sigma-cpt}
  Let $\theta$ be an action of an inverse semigroup $S$ on a locally compact Hausdorff space $X$. Then the following are equivalent.
  \begin{enumerate}
    \item There exists a finite (resp.\ countable) family of pairs $\{(f_i,K_i)\}_{i\in I}$ such that $f_i\in E(S)$, $K_i\subset X_{f_i}$ is compact for each $i\in I$ and $X_e\subset\bigcup_{i\in I,\, f_i\ge e} K_i$ for each $e\in E(S)$.
    \item $X^\ks$ is compact (resp.\ $\sigma$-compact).
  \end{enumerate}
\end{thm}
\begin{proof}
  Observe that $X^\ks$ is compact (resp.\ $\sigma$-compact) if and only if $\{\infty\}$ is open (resp.\ a $G_\delta$-subset) in $\widetilde{X^\ks}$, which is equivalent to saying that there exists a finite (resp.\ countable) family of pairs $\{(f_i,K_i)\}_{i\in I}$ such that $f_i\in E(S)$, $K_i\subset X_{f_i}$ is compact for each $i\in I$ and $\{\infty\}=\bigcap_{i\in I}q_{f_i}^{-1}(\widetilde{X_{f_i}}\setminus K_i)$, where $q_e\colon\widetilde{X^\ks}\to\widetilde{X_e}$ denotes the evaluation at $e\in E(S)$. By taking the complement, the equality is equivalent to $X^\ks=\bigcup_{i\in I}q_{f_i}^{-1}(K_i)$. From the lemma above, we are done.
\end{proof}

\subsection{Functoriality of the Construction of KS-Groupoids}
\label{subsec:3.2}

In this subsection, we check that the constructions of the space $X^\ks$ and the induced $S$-action on $X^\ks$ define an endofunctor on the category $\isa$ consisting of all inverse semigroup actions on locally compact Hausdorff spaces as objects and all proper action morphisms (in the sense of \cite{FKU}) as morphisms. From the functoriality of the construction of transformation groupoids (see Subsection 5.1 of \cite{FKU}), it follows that the construction of KS-groupoids defines a functor from $\isa$ to the category $\mathbf{EG}$, consisting of all \'etale groupoids as objects and all couple morphisms (in the sense of \cite{FKU}) as morphisms.

We first recall some terminology from \cite{FKU} which will be needed in this subsection.

For topological spaces $X,Y$, a \textit{partial map} from $Y$ to $X$ is a continuous map from an open subset $D_f$ of $Y$ to $X$. 
We denote this by $f\colon Y\supset D_f\to X$. 
For any partial maps $f\colon Z\supset D_f\to Y$ and $g\colon Y\supset D_g\to X$, we can define a composition $g\circ f\colon Z\supset f^{-1}(D_g)\to X$, which is a partial map from $Z$ to $X$.

\begin{dfn}[{\cite[Definition 3.1]{FKU}}]
  \label{def:action-morphism}
  Let $(S,X,\theta), (T,Y,\sigma)$ be actions of inverse semigroups. An \textit{action morphism} $(\varphi,f)$ from $(S,X,\theta)$ to $(T,Y,\sigma)$ consists of a semigroup homomorphism $\varphi\colon S\to T$ and a partial map $f\colon Y\supset D_f\to X$ satisfying
  \begin{enumerate}
    \item $f^{-1}(X_e)=Y_{\varphi(e)}$ for each $e\in E(S)$, and
    \item $\theta_s(f(y))=f(\sigma_{\varphi(s)}(y))$ for all $s\in S$, $y\in f^{-1}(X_{s^*s})=Y_{\varphi(s^*s)}$.
  \end{enumerate}
  An action morphism $(\varphi,f)$ is said to be \textit{proper} if the partial map $f\colon Y\supset D_f\to X$ is proper as a map from $D_f$ to $X$.
\end{dfn}

\begin{rem}
  If $(\varphi,f)$ is an action morphism from $(S,X,\theta)$ to $(T,Y,\sigma)$, then it follows from (i) above that $D_f=\bigcup_{e\in E(S)}Y_{\varphi(e)}$. 
\end{rem}

Let $(\varphi,f)$ be an action morphism from $(S,X,\theta)$ to $(T,Y,\sigma)$, where $X,Y$ are locally compact Hausdorff spaces. Define a map $\widetilde{f^\ks_\varphi}\colon \widetilde{Y^\ks}\to\widetilde{X^\ks}$ by
\begin{align*}
  \widetilde{f^\ks_\varphi}(\omega)\,(e)=
    \begin{cases}
      f(\omega(\varphi(e))) & \mathrm{if}\quad\omega(\varphi(e))\neq\infty_{\varphi(e)}\\
      \infty_e & \mathrm{otherwise}
    \end{cases}
\end{align*}
for $\omega\in\widetilde{Y^\ks}$ and $e\in E(S)$.
It is routine to verify that $\widetilde{f^\ks_\varphi}(\omega)$ belongs to $\widetilde{X^\ks}$.

\begin{lem}
  \label{lem:continuity}
  If $(\varphi,f)$ is a proper action morphism from $(S,X,\theta)$ to $(T,Y,\sigma)$, then the map $\widetilde{f^\ks_\varphi}\colon \widetilde{Y^\ks}\to\widetilde{X^\ks}$ is continuous.
\end{lem}
\begin{proof}
  Suppose that a net $(\omega_\lambda)_\lambda\subset\widetilde{Y^\ks}$ converges to $\omega$ in $\widetilde{Y^\ks}$. 
  It suffices to show that the net $(\widetilde{f^\ks_\varphi}(\omega_\lambda)\,(e))_\lambda$ converges to $\widetilde{f^\ks_\varphi}(\omega)\,(e)$ in $\widetilde{X_e}$ for each $e\in E(S)$. 
  If $\widetilde{f^\ks_\varphi}(\omega)\,(e)\neq\infty_e$, this claim is easy to check. 
  Suppose $\widetilde{f^\ks_\varphi}(\omega)\,(e)=\infty_e$, or equivalently, $\omega(\varphi(e))=\infty_{\varphi(e)}$. 
  Fix an open subset $U$ in $\widetilde{X_e}$ with $\infty_e\in U$. 
  Let $K=\widetilde{X_e}\setminus U$. 
  Then $K$ is a compact subset of $X_e$. 
  Since $f$ is proper, it follows that $f^{-1}(K)$ is a compact subset of $f^{-1}(X_e)=Y_{\varphi(e)}$. 
  Thus $V:=\widetilde{Y_{\varphi(e)}}\setminus f^{-1}(K)$ is an open subset in $\widetilde{Y_{\varphi(e)}}$ containing $\infty_{\varphi(e)}$. 
  Since the net $(\omega_\lambda(\varphi(e)))_\lambda$ converges to $\omega(\varphi(e))=\infty_{\varphi(e)}$, we can find $\lambda_0$ such that $\omega_\lambda(\varphi(e))\in V$ for all $\lambda\ge\lambda_0$. 
  For such $\lambda$, we have $\widetilde{f^\ks_\varphi}(\omega_\lambda)\,(e)\in\widetilde{X_e}\setminus K=U$. 
  This proves the lemma.
\end{proof}

For any $\omega\in Y^\ks$, observe that $\widetilde{f^\ks_\varphi}(\omega)\in X^\ks$ if and only if $\widetilde{f^\ks_\varphi}(\omega)\,(e)\neq\infty_e$ for some $e\in E(S)$, which is equivalent to saying $\omega(\varphi(e))\neq\infty_{\varphi(e)}$ for some $e\in E(S)$. Hence we have 
\[
  \bigl(\widetilde{f^\ks_\varphi}\bigr)^{-1}(X^\ks)=\bigcup_{e\in E(S)}Y^\ks_{\varphi(e)}.
\]
We denote this open subset in $Y^\ks$ by $D_{f^\ks_\varphi}$. Let $f^\ks_\varphi=\widetilde{f^\ks_\varphi}|_{D_{f^\ks_\varphi}}\colon D_{f^\ks_\varphi}\to X^\ks$. 

If $(\varphi,f)$ is a proper action morphism, then $f^\ks_\varphi$ is a proper partial map from $Y^\ks$ to $X^\ks$ by the lemma above. 
On the other hand, even if an action morphism $(\varphi,f)$ is not proper, this map $f^\ks_\varphi$ is continuous. 
We do not need this fact, but we include the proof because it may be interesting in its own right.

\begin{prop}
  If $(\varphi,f)$ is an action morphism from $(S,X,\theta)$ to $(T,Y,\sigma)$, then the map $f^\ks_\varphi\colon D_{f^\ks_\varphi}\to X^\ks$ is continuous.
\end{prop}
\begin{proof}
  It suffices to show that the map 
  \[
    f^\ks_{\varphi,\,e}\colon D_{f^\ks_\varphi}\ni\omega\mapsto f^\ks_\varphi(\omega)\,(e)\in\widetilde{X_e}
  \]
  is continuous for each $e\in E(S)$. 

  To this end, fix an open subset $U$ in $\widetilde{X_e}$. 
  If $\infty_e\notin U$, then $(f^\ks_{\varphi,\,e})^{-1}(U)=Y^\ks_{\varphi(e)}\cap p_Y^{-1}(f^{-1}(U))$ is open in $D_{f^\ks_\varphi}$ since $p_Y\colon Y^\ks\to Y$ is continuous (Proposition \ref{prop:projection}). 
  Suppose $\infty_e\in U$ and let $K=\widetilde{X_e}\setminus U$. 
  Then $K$ is a compact subset of $X_e$, and hence it is closed in $X$. 
  To see $(f^\ks_{\varphi,\,e})^{-1}(U)$ is open in $D_{f^\ks_\varphi}$, we claim that $C_K:=(f^\ks_{\varphi,\,e})^{-1}(K)$ is closed in $D_{f^\ks_\varphi}$.

  Let $(\omega_\lambda)_\lambda$ be a net in $C_K$ which converges to $\omega$ in $D_{f^\ks_\varphi}$. 
  Then $\omega_\lambda(\varphi(e))\neq\infty_{\varphi(e)}$ and $f(\omega_\lambda(\varphi(e)))\in K$ for each $\lambda$.
  Find $e_0\in E(S)$ with $\omega(\varphi(e_0))\neq\infty_{\varphi(e_0)}$.
  Since the net $(\omega_\lambda(\varphi(e_0)))_\lambda$ converges to $\omega(\varphi(e_0))\neq\infty_{\varphi(e_0)}$ in $\widetilde{Y_{\varphi(e_0)}}$, we may assume that $\omega_\lambda(\varphi(e_0))\neq\infty_{\varphi(e_0)}$ by taking a subnet. 
  Note that, from Remark \ref{rem:}, $\omega_\lambda(\varphi(e_0))=\omega_\lambda(\varphi(e))$ for each $\lambda$.
  From the continuity of the map $f$, the net $(f(\omega_\lambda(\varphi(e_0))))_\lambda$ converges to $f(\omega(\varphi(e_0)))$ in $X$.
  Since each $f(\omega_\lambda(\varphi(e_0)))=f(\omega_\lambda(\varphi(e)))$ belongs to $K$ and $K$ is closed in $X$, we see that $f(\omega(\varphi(e_0)))\in X_e$.
  From $f^{-1}(X_e)=Y_{\varphi(e)}$, it follows that $\omega(\varphi(e_0))\in Y_{\varphi(e)}$. 
  Since $\omega_\lambda(\varphi(e))=\omega_\lambda(\varphi(e_0))$ for each $\lambda$ and $\widetilde{Y_{\varphi(e)}}$ is Hausdorff, their limits coincide.
  Hence $\omega(\varphi(e))=\omega(\varphi(e_0))$. 
  In particular, $\omega(\varphi(e))\neq\infty_{\varphi(e)}$.
  Now we have
  \[
    f^\ks_{\varphi,\,e}(\omega)=f^\ks_\varphi(\omega)\,(e)=f(\omega(\varphi(e)))=f(\omega(\varphi(e_0)))\in K,
  \]
  or equivalently, $\omega\in C_K$. This completes the proof.
\end{proof}

Now we prove that the constructions of the space $X^\ks$ and the induced $S$-action on $X^\ks$ define an endofunctor on $\isa$. 

\begin{lem}
  For any proper action morphism $(\varphi,f)$ from $(S,X,\theta)$ to $(T,Y,\sigma)$, the pair $(\varphi,f_\varphi^\ks)$ is a proper action morphism from $(S,X^\ks,\theta^\ks)$ to $(T,Y^\ks,\sigma^\ks)$.
\end{lem}
\begin{proof}
  We have shown that $f_\varphi^\ks\colon Y^\ks\supset D_{f^\ks_\varphi}\to X^\ks$ is a proper continuous map. 
  Hence it remains to check the conditions (i) and (ii) of Definition \ref{def:action-morphism}.

  For (i), let $e\in E(S)$ and $\omega\in Y^\ks$. 
  Observe that $f^\ks_\varphi(\omega)\in X^\ks_e$ if and only if $f^\ks_\varphi(\omega)\,(e)\neq\infty_e$, which is equivalent to $\omega(\varphi(e))\neq\infty_{\varphi(e)}$. 
  Thus $(f^\ks_\varphi)^{-1}(X^\ks_e)=Y^\ks_{\varphi(e)}$ for each $e\in E(S)$.

  For (ii), let $s\in S$, $\omega\in Y^\ks_{\varphi(s^*s)}$ and $e\in E(S)$. Compute that
  \begin{align*}
    \theta^\ks_s(f^\ks_\varphi(\omega))\, (e)
    &=\begin{cases}
      \theta_{es}(f^\ks_\varphi(\omega)\,(s^*es)) & \mathrm{if}\; f^\ks_\varphi(\omega)\,(s^*es)\neq\infty_{s^*es}\\
      \infty_e & \mathrm{otherwise}
    \end{cases}\\
    &=\begin{cases}
      \theta_{es}(f(\omega(\varphi(s^*es)))) & \mathrm{if}\; \omega(\varphi(s^*es))\neq\infty_{\varphi(s^*es)}\\
      \infty_e & \mathrm{otherwise}
    \end{cases}
  \end{align*}
  and 
  \begin{align*}
    f^\ks_\varphi(\sigma^\ks_{\varphi(s)}(\omega))\,(e)
    &=\begin{cases}
      f(\sigma^\ks_{\varphi(s)}(\omega)\,(\varphi(e))) & \mathrm{if}\;\sigma^\ks_{\varphi(s)}(\omega)\,(\varphi(e))\neq\infty_{\varphi(e)}\\
      \infty_e & \mathrm{otherwise}\\
    \end{cases}\\
    &=\begin{cases}
      f(\sigma_{\varphi(es)}(\omega(\varphi(s^*es)))) & \mathrm{if}\; \omega(\varphi(s^*es))\neq\infty_{\varphi(s^*es)}\\
      \infty_e & \mathrm{otherwise}.
    \end{cases}
  \end{align*}
  Since $\omega(\varphi(s^*es))\in Y_{\varphi(s^*es)}$ and $(\varphi,f)$ is an action morphism, we have $\theta_{es} (f(\omega(\varphi(s^*es))))=f(\sigma_{\varphi(es)}(\omega(\varphi(s^*es))))$.
  This proves $\theta^\ks_s(f^\ks_\varphi(\omega))\, (e)=f^\ks_\varphi(\sigma^\ks_{\varphi(s)}(\omega))\,(e)$.
\end{proof}

It is shown in Section 3 of \cite{FKU} that all inverse semigroup actions on locally compact Hausdorff spaces and all proper action morphisms form a category. 
We denote this category by $\isa$. 
Note that the composition of two action morphisms $(\varphi,f)\colon (S_1,X_1,\theta_1)\to (S_2,X_2,\theta_2)$ and $(\psi,g)\colon (S_2,X_2,\theta_2)\to (S_3,X_3,\theta_3)$ in $\isa$ is defined by 
\[
  (\psi,g)\circ(\varphi,f)=(\psi\circ\varphi,f\circ g).
\]

\begin{thm}
  The constructions $(S,X,\theta)\mapsto (S,X^\ks,\theta^\ks)$ and $(\varphi,f)\mapsto (\varphi,f^\ks_\varphi)$ form a functor from $\isa$ to $\isa$.
\end{thm}
\begin{proof}
  The only non-trivial part is to prove that this correspondence preserves any composition of two action morphisms.

  Let $(\varphi,f)\colon (S_1,X,\theta_1)\to (S_2,Y,\theta_2)$ and $(\psi,g)\colon (S_2,Y,\theta_2)\to (S_3,Z,\theta_3)$ be morphisms in the category $\isa$. 
  To see $(\psi,g_\psi^\ks)\circ(\varphi,f_\varphi^\ks)=(\psi\circ\varphi,(f\circ g)_{\psi\circ\varphi}^\ks)$, it suffices to show that $f_\varphi^\ks\circ g_\psi^\ks=(f\circ g)_{\psi\circ\varphi}^\ks$ as partial maps.

  First note that the domain of $(f\circ g)_{\psi\circ\varphi}^\ks$ is $\bigcup_{e\in E(S_1)}Z^\ks_{\psi\circ\varphi(e)}$. 
  Observe that the domain of $f_\varphi^\ks\circ g_\psi^\ks$ is
  \[
    (g_\psi^\ks)^{-1}(D_{f_\varphi^\ks})=(g_\psi^\ks)^{-1}\Bigl(\bigcup_{e\in E(S_1)}Y^\ks_{\varphi(e)}\Bigr)=\bigcup_{e\in E(S_1)}Z^\ks_{\psi\circ\varphi(e)}.
  \]
  Hence the domain of $f_\varphi^\ks\circ g_\psi^\ks$ and the domain of $(f\circ g)_{\psi\circ\varphi}^\ks$ coincide.

  It is now routine to verify that $f_\varphi^\ks\circ g_\psi^\ks(\omega)=(f\circ g)_{\psi\circ\varphi}^\ks(\omega)$ for any $\omega\in D_{(f\circ g)^\ks_{\psi\circ\varphi}}$.
\end{proof}

\subsection{Universality of KS-Groupoids}
\label{subsec:3.3}

The universality of the universal groupoid $S\ltimes\widehat{E(S)}$ has been studied widely in the literature: see, for example, \cite{paterson,steinberg,es}. In this subsection, we extend these arguments to the KS-groupoids.

Let $\sigma,\theta$ be actions of $S$ on locally compact Hausdorff spaces $Y,X$, respectively, and let $f\colon Y\to X$ be an $S$-equivariant map. Define a map $f^\ks\colon Y^\ks\to X^\ks$ by
\begin{align*}
  f^\ks(\tau)(e)=
  \begin{cases}
    f(\tau(e)) & \mathrm{if}\;\tau(e)\neq\infty_e\\
    \infty_e & \mathrm{otherwise.}
  \end{cases}
\end{align*}
It is easy to see $f^\ks(\tau)\in X^\ks$ for all $\tau\in Y^\ks$.

\begin{lem}
  \label{lem:on-f^ks}
  The following hold for $f\colon Y\to X$.
  \begin{enumerate}
    \item $f^\ks$ is $S$-equivariant and $(f^\ks)^{-1}(X^\ks_e)=Y_e^\ks$ for each $e\in E(S)$.
    \item $f^\ks$ is Borel if $E(S)^X:=\{e\in E(S)\mid X_e\neq\emptyset\}$ is countable and $f$ is Borel.
  \end{enumerate}
\end{lem}
\begin{proof}
  (i): It is easy to see $(f^\ks)^{-1}(X^\ks_e)=Y_e^\ks$ for each $e\in E(S)$. For any $\tau\in Y^\ks_{s^*s}$ and $e\in E(S)$, we compute that
  \begin{align*}
    \theta_s^\ks(f^\ks(\tau))\,(e)
    &=\begin{cases}
      \theta_{es}(f^\ks(\tau)(s^*es)) & \mathrm{if}\;f^\ks(\tau)(s^*es)\neq\infty_{s^*es}\\
      \infty_e & \mathrm{otherwise}
    \end{cases}\\
    &=\begin{cases}
      \theta_{es}(f(\tau(s^*es))) & \mathrm{if}\;\tau(s^*es)\neq\infty_{s^*es}\\
      \infty_e &\mathrm{otherwise}
    \end{cases}
  \end{align*}
  and 
  \begin{align*}
    f^\ks(\sigma_s^\ks(\tau))\,(e)
    &=\begin{cases}
      f(\sigma_s^\ks(\tau)(e)) & \mathrm{if}\;\sigma_s^\ks(\tau)(e)\neq\infty_{e}\\
      \infty_e &\mathrm{otherwise}
    \end{cases}\\
    &=\begin{cases}
      f(\sigma_{es}(\tau(s^*es))) & \mathrm{if}\;\tau(s^*es)\neq\infty_{s^*es}\\
      \infty_e &\mathrm{otherwise}
    \end{cases}\\
    &=\begin{cases}
      \theta_{es}(f(\tau(s^*es))) &\mathrm{if}\;\tau(s^*es)\neq\infty_{s^*es}\\
      \infty_e & \mathrm{otherwise}.
    \end{cases}
  \end{align*}

  (ii): For any $V\in\mathbb{O}(X)$, $e\in E(S)^X$ and a finite subset $G\subset E(S)^X$, set 
  \[
    B_X(V;e,G)=p_X^{-1}(V)\cap X^\ks_e\cap \bigcap_{g\in G}(X^\ks\setminus X^\ks_g).
  \]
  Clearly, this subset is Borel in $X^\ks$. 
  We first show that every open subset of $X^\ks$ is a countable union of subsets of this form.
  For any $e\in E(S)^X$ and an open subset $U$ of $\widetilde{X_e}$, we have $q_e^{-1}(U)=p_X^{-1}(U)\cap X^\ks_e$ if $\infty_e\notin U$; and 
  \begin{align*}
    q_e^{-1}(U)
    &=\bigl(p_X^{-1}(U\setminus\{\infty_e\})\cap X_e^\ks\bigr)\cup (X^\ks\setminus X_e^\ks)\\
    &=\bigl(p_X^{-1}(U\setminus\{\infty_e\})\cap X_e^\ks\bigr)\cup\bigcup_{f\in E(S)^X} B_X(X;f,\{e\})
  \end{align*}
  if $\infty_e\in U$, where $q_e\colon X^\ks\to\widetilde{X_e}$ denotes the evaluation at $e\in E(S)^X$.
  Since there are only countably many pairs $(e,G)$ such that $e\in E(S)^X$ and $G\subset E(S)^X$ is a finite subset, and 
  \[
    \bigcup_i B_X(V_i;e,G)=B_X\Bigl(\bigcup_i V_i;e,G\Bigr),\quad B_X(V;e,G)\cap B_X(W;e,G)=B_X(V\cap W;e,G)
  \]
  for any family $\{V_i\}_i\subset \mathbb{O}(X)$ and $V,W\in\mathbb{O}(X)$, it follows that every open subset of $X^\ks$ is a countable union of subsets of this form $B_X(V;e,G)$.

  Let $W$ be an open subset of $X^\ks$ and write $W=\bigcup_{n=1}^\infty B_X(V_n;e_n,G_n)$. Using (i) and the equality $p_X\circ f^\ks=f\circ p_Y$, we have
  \[
    (f^\ks)^{-1}(W)=\bigcup_{n=1}^\infty\;\Bigl(p_Y^{-1}(f^{-1}(V_n))\cap Y^\ks_{e_n}\cap \bigcap_{g\in G_n}(Y^\ks\setminus Y^\ks_g)\Bigr).
  \]
  Hence $(f^\ks)^{-1}(W)$ is Borel since $f$ is assumed to be Borel.
  This proves (ii).
\end{proof}

In the following, we denote the composition $f^\ks\circ\iota_Y\colon Y\to X^\ks$ by $\rho$.

\begin{lem}
  \label{lem:rho}
  For all $y\in Y$ and $e\in E(S)$,
  \begin{align*}
    \rho(y)(e)=(f^\ks\circ\iota_Y(y))(e)=
    \begin{cases}
      f(y) & \mathrm{if}\;y\in Y_e\\
      \infty_e & \mathrm{otherwise.}
    \end{cases}
  \end{align*}
\end{lem}
\begin{proof}
  It is routine to verify this.
\end{proof}

\begin{prop}
  \label{prop:rho}
  For the map $\rho=f^\ks\circ\iota_Y$, the following hold.
  \begin{enumerate}
    \item $y\in Y_e$ if and only if $\rho(y)\in X_e^\ks$.
    \item $\rho$ is $S$-equivariant.
    \item $\rho$ is a Borel map if $E(S)^X:=\{e\in E(S)\mid X_e\neq\emptyset\}$ is countable and $f$ is Borel.
    \item $\rho$ is continuous if $f$ is continuous and $Y_e$ is closed for each $e\in E(S)$.
  \end{enumerate}
\end{prop}
\begin{proof}
  (i),(iv) follow immediately from the lemma above.

  (ii),(iii) follow immediately from Lemma \ref{lem:on-f^ks} and the fact that the map $\iota_Y\colon Y\to Y^\ks$ is continuous and $S$-equivariant.
\end{proof}

Since $\rho\colon Y\to X^\ks$ is $S$-equivariant, it follows from Lemma \ref{lem:hom} that the map
\[
  \widetilde{\rho}\colon S\ltimes Y\ni [s,y]\mapsto [s,\rho(y)]\in S\ltimes X^\ks
\]
is a well-defined groupoid homomorphism.

\begin{dfn}
  \label{def:d-bij}
  A groupoid homomorphism $g\colon G\to H$ is said to be \textit{\d-bijective} if the map $g|_{G_x}\colon G_x\to H_{g(x)}$ is bijective for all $x\in G^{(0)}$.
\end{dfn}

\begin{thm}
  \label{thm:univ}
  Let $\theta,\sigma$ be actions of an inverse semigroup $S$ on locally compact Hausdorff spaces $X,Y$, respectively, and let $f\colon Y\to X$ be an $S$-equivariant map. Then for the $S$-equivariant map $\rho=f^\ks\circ\iota_Y\colon Y\to X^\ks$, the groupoid homomorphism $\widetilde{\rho}$ is \d-bijective. Moreover, the following hold.
  \begin{enumerate}
    \item $\tilde{\rho}$ is Borel if $S^X$ is countable and $f$ is Borel.
    \item $\tilde{\rho}$ is continuous if $f$ is continuous and $Y_e$ is closed for each $e\in E(S)$.
  \end{enumerate}
\end{thm}
\begin{proof}
  To see that $\tilde{\rho}$ is \d-bijective, take $y\in Y=(S\ltimes Y)^{(0)}$. 
  Surjectivity of the map $\widetilde{\rho}|_{(S\ltimes Y)_y}\colon (S\ltimes Y)_y\to (S\ltimes X^\ks)_{\rho(y)}$ is clear from (i) of Proposition \ref{prop:rho}. 
  For injectivity, suppose $\widetilde{\rho}([s,y])=\widetilde{\rho}([t,y])$, that is, $[s,\rho(y)]=[t,\rho(y)]$. 
  Then there exists $e\in E(S)$ such that $se=te$ and $\rho(y)\in X^\ks_e$, which is equivalent to saying $se=te$ and $y\in Y_e$ for some $e\in E(S)$ by Proposition \ref{prop:rho} (i). 
  Then $[s,y]=[t,y]$, which proves the injectivity. 
  The last two statements follow from Lemma \ref{lem:hom}.
\end{proof}

\begin{rem}
  If $g\colon G\to H$ is a \d-bijective Borel groupoid homomorphism between second countable \'etale groupoids $G,H$, it is known that amenability of $H$ passes to $G$ (\cite[Proposition 2.4]{es}). In particular, if $S\ltimes X^\ks$ is amenable, then $S\ltimes Y$ is amenable in the setting of Theorem \ref{thm:univ} (i) with the assumption that $X$ and $Y$ are second countable.
\end{rem}

\section{KS-Groupoids and KS-Crossed Products}
\label{sec:4}

\subsection{The Space \texorpdfstring{$X^\ks$}{} as the Spectrum of \texorpdfstring{$E(S)\kstimes C_0(X)$}{}}
\label{subsec:4.1}

Let $\theta$ be an action of an inverse semigroup $S$ on a locally compact Hausdorff space $X$. In this subsection, we show that the space $X^\ks$ is naturally homeomorphic to the spectrum of the commutative C*-algebra $E(S)\kstimes C_0(X)$, and the $S$-action on $X^\ks$ defined in Theorem \ref{thm:action-on-X^ks} is compatible with the $S$-action on $E(S)\kstimes C_0(X)$ as in Proposition \ref{prop:action-on-A^ks}.

\begin{prop}
  \label{prop:spectrum}
  The following hold.
  \begin{enumerate}
    \item For a character $\chi$ on $E(S)\kstimes C_0(X)$, there exists a unique element $\tau_\chi$ in $\prod_{e\in E(S)}\widetilde{X_e}$ satisfying
    \[
      g(\tau_\chi(e))=\chi(\delta_e\,g)
    \]
    for all $e\in E(S)$ and $g\in C_0(X_e)$. Moreover, $\tau_\chi\in X^\ks$.
    \item For $\tau\in X^\ks$, there exists a unique character $\chi_\tau:E(S)\kstimes C_0(X)\to\mathbb{C}$ such that 
    \[
      \chi_\tau(\delta_e\,g)=g(\tau(e))
    \]
    for all $e\in E(S)$ and $g\in C_0(X_e)$.
    \item These correspondences 
    \[
      \Sp(E(S)\kstimes C_0(X))\ni\chi\mapsto \tau_\chi\in X^\ks,\quad X^\ks\ni\tau\mapsto\chi_\tau\in\Sp(E(S)\kstimes C_0(X))
      \]
    are inverses of each other, and are homeomorphisms between $\Sp(E(S)\kstimes C_0(X))$ and $X^\ks$.
  \end{enumerate}
\end{prop}
\begin{proof}
  (i): Clearly, the map $C_0(X_e)\ni g\mapsto \chi(\delta_e\,g)\in\mathbb{C}$ is a $*$-homomorphism. Hence there exists $\tau_\chi(e)\in\widetilde{X_e}$ such that $\chi(\delta_e\,g)=g(\tau_\chi(e))$ for all $g\in C_0(X_e)$. This defines $\tau_\chi\in\prod_{e\in E(S)}\widetilde{X_e}$ with the desired property.

  We next show that $\tau_\chi\in X^\ks$. Clearly, $\tau_\chi\neq\infty$ since $\chi\neq 0$. 

  It suffices to check the conditions (1) and (2) of Definition \ref{def:of-X^ks}. For (1), suppose $e_1\le e_2$ and $\tau_\chi(e_1)\neq\infty_{e_1}$. Then there exists $g\in C_0(X_{e_1})$ with $\chi(\delta_{e_1}\,g)=g(\tau_\chi(e_1))=1$. For $f\in C_0(X_{e_2})$, compute that 
  \[
    f(\tau_\chi(e_2))=\chi(\delta_{e_2}\,f)=\chi(\delta_{e_2}\,f)\chi(\delta_{e_1}\,g)=\chi(\delta_{e_1}\,fg)=(fg)(\tau_\chi(e_1))=f(\tau_\chi(e_1)).
  \]
  Since $f\in C_0(X_{e_2})$ is arbitrary, we obtain $\tau_\chi(e_2)=\tau_\chi(e_1)$.

  For (2), suppose $\tau_\chi(e_1)\neq\infty_{e_1}$ and $\tau_\chi(e_2)\neq\infty_{e_2}$. Then there exist $f\in C_0(X_{e_1})$ and $g\in C_0(X_{e_2})$ satisfying $f(\tau_\chi(e_1))=1$ and $g(\tau_\chi(e_2))=1$. Then we compute that
  \[
    (fg)(\tau_\chi(e_1e_2))=\chi(\delta_{e_1e_2}\,fg)=\chi(\delta_{e_1}\,f)\chi(\delta_{e_2}\,g)=f(\tau_\chi(e_1))g(\tau_\chi(e_2))=1,
  \]
  from which it follows $\tau_\chi(e_1e_2)\neq\infty_{e_1e_2}$.

  \noindent
  (ii): It suffices to show that the functional $\chi_\tau\colon E(S)\alg C_0(X)\to\mathbb{C}$ given by $\chi_\tau(\delta_e\,g)=g(\tau(e))$ is a nonzero $*$-homomorphism. Clearly, $\chi_\tau\neq 0$ since $\tau\neq\infty$. It is also easy to see $\chi_\tau$ is $*$-preserving. To show that $\chi_\tau$ is multiplicative, recall from Remark \ref{rem:} that the condition $\tau(e_1e_2)\neq\infty_{e_1e_2}$ is equivalent to $\tau(e_1)\neq\infty_{e_1}$ and $\tau(e_2)\neq\infty_{e_2}$. If these equivalent conditions hold, we have $\tau(e_1)=\tau(e_1e_2)=\tau(e_2)$ and hence
  \[
    \chi_\tau(\delta_{e_1}\,f)\chi_\tau(\delta_{e_2}\,g)=f(\tau(e_1))g(\tau(e_2))=(fg)(\tau(e_1e_2))=\chi_\tau(\delta_{e_1e_2}\,fg)
  \]
  for all $f\in C_0(X_{e_1})$ and $g\in C_0(X_{e_2})$. If the equivalent conditions do not hold, then 
  \[
    \chi_\tau(\delta_{e_1}\,f)\chi_\tau(\delta_{e_2}\,g)=0=\chi_\tau(\delta_{e_1e_2}\,fg).
  \]
  Hence $\chi_\tau$ uniquely extends to a character on $E(S)\kstimes C_0(X)$ by the universality of $E(S)\kstimes C_0(X)$.

  \noindent
  (iii): It is routine to check the correspondences $\chi\mapsto\tau_\chi$ and $\tau\mapsto \chi_\tau$ are continuous. 

  Compute that
  \begin{align*}
    \chi_{\tau_\chi}(\delta_e\, g)&=g(\tau_\chi(e))=\chi(\delta_e\, g),\\
    g(\tau_{\chi_\tau}(e))&=\chi_\tau(\delta_e\,g)=g(\tau(e))
  \end{align*}
  for all $e\in E(S)$ and $g\in C_0(X_e)$. Thus $\chi_{\tau_\chi}=\chi$ and $\tau_{\chi_\tau}=\tau$. These identities show that the correspondences $\chi\mapsto\tau_\chi$ and $\tau\mapsto \chi_\tau$ are inverses of each other, and are homeomorphisms between $\Sp(E(S)\kstimes C_0(X))$ and $X^\ks$.
\end{proof}

Recall from Theorem \ref{thm:action-on-X^ks} that the assignments
\begin{enumerate}
  \item $X^\ks_e=\{\tau\in X^\ks\mid \tau(e)\neq\infty_e\}$ for each $e\in E(S)$, and
  \item $\theta_s^\ks\colon X^\ks_{s^*s}\to X^\ks_{ss^*}$ given by
  \begin{equation*}
    \theta_s^\ks(\tau)\,(e)=
    \begin{cases}
      \theta_{es}(\tau(s^*es)) & \mathrm{if}\;\tau(s^*es)\neq\infty_{s^*es}\\
      \infty_e & \mathrm{otherwise}
    \end{cases}
  \end{equation*}
  for each $s\in S$
\end{enumerate}
define an $S$-action on $X^\ks$.
We next check that this action of $S$ on $X^\ks$ is compatible with the $S$-action on $E(S)\kstimes C_0(X)=C_0(X^\ks)$.

\begin{prop}
  \label{prop:action-on-spectrum}
  With the identification 
  \[
    \Sp(C_0(X)^\ks_e)=\{\chi\in\Sp(C_0(X)^\ks)\mid \chi|_{C_0(X)^\ks_e}\neq 0\},
  \]
  the homeomorphism between $\Sp(C_0(X)^\ks)$ and $X^\ks$ as in Proposition \ref{prop:spectrum} restricts to a homeomorphism between $\Sp(C_0(X)_e^\ks)$ and $X_e^\ks$.

  Define an isomorphism $\Phi_e:C_0(X)^\ks_e\to C_0(X_e^\ks)$ by
  \[
    \Phi_e(\delta_{e_0}\,g)(\tau)=\chi_\tau(\delta_{e_0}\,g)=g(\tau(e_0)),\quad e_0\le e,\quad g\in C_0(X_{e_0}).
  \]
  Then the diagram
  \begin{center}
    \begin{tikzpicture}[auto]
      \node(1) at (0,1.5) {$C_0(X)^\ks_{s^*s}$};
      \node (2) at (2.5,1.5) {$C_0(X)^\ks_{ss^*}$};
      \node (3) at (0,0) {$C_0(X^\ks_{s^*s})$};
      \node (4) at (2.5,0) {$C_0(X^\ks_{ss^*})$};
      \draw[->] (1) to node {$\scriptstyle\alpha_s^\ks$} (2);
      \draw[->] (1) to node[swap] {$\scriptstyle\Phi_{s^*s}$} (3);
      \draw[->] (3) to node[swap] {$\scriptstyle -\circ\theta_{s^*}^\ks$} (4);
      \draw[->] (2) to node {$\scriptstyle\Phi_{ss^*}$} (4);
    \end{tikzpicture}
  \end{center}
  commutes for each $s\in S$.
\end{prop}
\begin{proof}
  For the first statement, it suffices to show that
  \begin{enumerate}
     \item $\tau_\chi\in X_e^\ks$ for all $\chi\in\Sp(C_0(X)^\ks_e)$, and
     \item $\chi_\tau\in\Sp(C_0(X)^\ks_e)$ for all $\tau\in X_e^\ks$.
  \end{enumerate}

  For (i), take $\chi\in \Sp(C_0(X)^\ks_e)$. Then $\chi(\delta_{e_0}\, f)\neq 0$ for some $e_0\le e$ and $f\in C_0(X_{e_0})$. This implies $\tau_\chi(e_0)\neq\infty_{e_0}$, and it follows from $e_0\le e$ that $\tau_\chi(e)=\tau_\chi(e_0)$. 

  For (ii), fix $\tau\in X_e^\ks$. Then $\tau(e)\neq\infty_e$ and hence $\chi_\tau(\delta_e\,g)=g(\tau(e))\neq 0$ for some $g\in C_0(X_e)$. Since $\delta_e\, g\in C_0(X)^\ks_e$, we obtain $\chi_\tau\in \Sp(C_0(X)^\ks_e)$, as desired.

  To check the diagram commutes, fix $e_0\le s^*s$ and $g\in C_0(X_{e_0})$. Compute that
  \begin{align*}
    \bigl((-\circ\theta^\ks_{s^*})\circ\Phi_{s^*s}(\delta_{e_0}\,g)\bigr)\,(\tau)
    &=\Phi_{s^*s}(\delta_{e_0}\,g)\,(\theta^\ks_{s^*}(\tau))\\
    &=g(\theta_{s^*}^\ks(\tau)(e_0))\\
    &=\begin{cases}
      g(\theta_{e_0s^*}(\tau(se_0s^*))) & \mathrm{if}\;\tau(se_0s^*)\neq\infty_{se_0s^*}\\
      0 & \mathrm{otherwise}
    \end{cases}\\
    &=\alpha_{se_0}(g)(\tau(se_0s^*))\\
    &=\Phi_{ss^*}(\delta_{se_0s^*}\,\alpha_{se_0}(g))\,(\tau)\\
    &=\Phi_{ss^*}\circ\alpha_s^\ks(\delta_{e_0}\,g)\,(\tau)
  \end{align*}
  for all $\tau\in X^\ks_{ss^*}$. 
\end{proof}

\subsection{Dynamical Paterson's Theorem}
\label{subsec:4.2}

In this subsection, we will prove that
\[
  S\kstimes C_0(X)\cong C^*(S\ltimes X^\ks),\quad\mathrm{and}\quad S\ksr C_0(X)\cong C^*_r(S\ltimes X^\ks),
\]
for any action of an inverse semigroup $S$ on a locally compact Hausdorff space $X$. 
Recall from Remark \ref{rem:C*S} that $C^*(S)= S\kstimes C(\pt)$, $C_r^*(S)= S\ksr C(\pt)$ and $\pt^\ks=\widehat{E(S)}$, where the one-point space $\pt$ is equipped with the $S$-action given by $\pt_e=\pt$.
Hence these isomorphisms can be seen as a generalization of Paterson's theorem, which states that
\[
  C^*(S)\cong C^*(S\ltimes\widehat{E(S)}),\quad C^*_r(S)\cong C^*_r(S\ltimes\widehat{E(S)})
\]
for any inverse semigroup $S$ (see \cite[Theorem 4.4.1]{paterson} for the full case, and see \cite[Theorem 4.4.2]{paterson}, \cite[Theorem 3.5]{ks1} for the reduced case).

Notice that the full case is now clear since
\begin{align}
  S\kstimes C_0(X)\cong S\ltimes C_0(X)^\ks\cong S\ltimes C_0(X^\ks)\cong C^*(S\ltimes X^\ks)
\end{align}
where we use Theorem \ref{thm:isom-of-KS}, Proposition \ref{prop:spectrum}, Proposition \ref{prop:action-on-spectrum} and Theorem \ref{thm:isom}. 
Observe that the composition of these isomorphisms from $S\kstimes C_0(X)$ to $C^*(S\ltimes X^\ks)$ is induced from the $*$-homomorphism $\varphi\colon S\alg C_c(X)\to\mathcal{C}(S\ltimes X^\ks)$ which maps $\delta_s\,f\in S\alg C_c(X)$ to the element $\varphi(\delta_s\,f)\in C_c([s,X^\ks_{s^*s}])\subset\mathcal{C}(S\ltimes X^\ks)$ given by 
\[
  \varphi(\delta_s\,f)([s,\tau])=f(\tau(s^*s))
\]
for all $\tau\in X^\ks_{s^*s}$. In this subsection, we show that $\varphi$ induces the isomorphism $S\ksr C_0(X)\cong C^*_r(S\ltimes X^\ks)$. Since the image of $S\alg C_c(X)$ under $\varphi$ is dense in $C^*_r(S\ltimes X^\ks)$, the essential part of the proof is the equality $\|\varphi(T)\|_r=\|T\|_{\ks,r}$ for all $T\in S\alg C_c(X)$. For this, recall from the definitions of each norm that
\begin{equation}
  \label{eq:norm}
    \begin{aligned}
      \|T\|_{\ks,r}&=\sup\{\|\ell_e(T)\|\mid e\in E(S)\},\;\mathrm{and}\\
      \|\varphi(T)\|_r&=\sup\{\|\lambda_{\tau}(\varphi(T))\|\mid \tau\in X^\ks\}.
    \end{aligned}
\end{equation}

First observe that, for each $e\in E(S)$ and $x\in X_e$, the map 
\[
  \rho_{e,\,x}\colon C_0(X_e)\ni f\mapsto f(x)\in\mathbb{C}=\mathbb{B}(\mathbb{C})
\]
is a non-degenerate representation of $C_0(X_e)$. 
Moreover, the direct sum
\[
  \rho_e:=\bigoplus_{x\in X_e}\rho_{e,\,x}\colon C_0(X_e)\to\mathbb{B}(E_e)
\]
is a non-degenerate, faithful representation of $C_0(X_e)$ on the Hilbert $\bigoplus_{x\in X_e}\mathbb{C}$-module $E_e=\bigoplus_{x\in X_e}\mathbb{C}$. 
From the equality \eqref{eq:rep} in Observation \ref{obs:ind-rep}, we have $\|\ind\rho_e(T)\|=\|\ell_e(T)\|$ for all $T\in S\alg C_0(X)$ and $e\in E(S)$. 
Moreover, since the representation $\ind\rho_e\colon S\alg C_0(X)\to\mathbb{B}(l^2(S_e,E_e))$ is unitarily equivalent to the direct sum 
\[
  \bigoplus_{x\in X_e}\ind\rho_{e,x}\colon S\alg C_0(X)\to \prod_{x\in X_e}\mathbb{B}(l^2(S_e))\subset\mathbb{B}\Bigl(\bigoplus_{x\in X_e}l^2(S_e)\Bigr),
\]
we obtain
\begin{align*}
  \|T\|_{\ks,r}
  &=\sup\{\|\ind\rho_e(T)\|\mid e\in E(S)\}\\
  &=\sup\{\|\ind\rho_{e,\,x}(T)\|\mid e\in E(S),\; x\in X_e\}.
\end{align*} 

Here, we recall that $\ind\rho_{e,\,x}\colon S\alg C_0(X)\to\mathbb{B}(l^2(S_e))$ is the $*$-representation on the Hilbert space $l^2(S_e)$ defined by 
\begin{equation}
  \label{eq:ind}
  \ind\rho_{e,\,x}(\delta_s\,f)\,(\delta_t)=
  \begin{cases}
    f(\theta_t(x))\,\delta_{st} & \mathrm{if}\;s^*st=t\\
    0 & \mathrm{otherwise}
  \end{cases}
\end{equation}
for all $s\in S$, $f\in C_0(X_{s^*s})$ and $t\in S_e$ (see Observation \ref{obs:ind-rep}).

We next show that the second supremum in \eqref{eq:norm} is equal to the supremum over $\{\tau_{e,\,x}\in X^\ks\mid e\in E(S),\;x\in X_e\}$. This will be done by extending the argument in \cite{ks1} to the setting of KS-groupoids.

Let $\tau\in X^\ks$. Observe that the set $\Lambda_\tau=\{e\in E(S)\mid \tau(e)\neq\infty_e\}$ is directed with respect to the partial order $\preceq$ given by $e\preceq f$ if $f\le e$. 

\begin{lem}
  \label{lem:E(S)-is-dense}
  For $\tau\in X^\ks$ and $x=p_X(\tau)$, the net $(\tau_{e,\,x})_{e\in\Lambda_\tau}$ converges to $\tau$ in $X^\ks$. In particular, $\{\tau_{e,\,x}\mid e\in E(S),\;x\in X_e\}$ is dense in $X^\ks$.
\end{lem}
\begin{proof}
  It suffices to show for any $f\in E(S)$ the net $(\tau_{e,\,x}(f))_{e\in\Lambda_\tau}$ converges to $\tau(f)$ in $\widetilde{X_f}$.

  Suppose first that $\tau(f)\neq\infty_f$, that is, $\tau(f)=x$. Then $f\in\Lambda_\tau$ and hence $\tau_{e,\,x}(f)=x$ for all $e\in\Lambda_\tau$ with $f\preceq e$. 

  Suppose next that $\tau(f)=\infty_f$. Then $\tau_{e,\,x}(f)=\infty_f$ for each $e\in\Lambda_\tau$. Indeed, if $\tau_{e,\,x}(f)\neq\infty_f$, then $e\le f$ by the definition of $\tau_{e,\,x}$, but this implies $\tau(f)\neq\infty_f$.
\end{proof}

\begin{prop}
  \label{prop:norm}
  $\|g\|_r=\sup\{\|\lambda_{\tau_{e,\,x}}(g)\|\mid e\in E(S),\;x\in X_e\}$ for any $g\in\mathcal{C}(S\ltimes X^\ks)$.
\end{prop}
\begin{proof}
  From Theorem \ref{thm:KS}, it suffices to show that for any $\tau\in X^\ks=(S\ltimes X^\ks)^{(0)}$ and $g\in \mathcal{C}(S\ltimes X^\ks)$ the net $(g(\tau_{e,\,x}))_{e\in\Lambda_\tau}$ converges to $g(\tau)$. From Lemma \ref{lem:partition}, we may assume that $g\in C_c([s,X^\ks_{s^*s}])$ for some $s\in S$. 

  We first consider the case $\tau\in [s,X^\ks_{s^*s}]$. Since the net $(\tau_{e,\,x})_{e\in\Lambda_\tau}$ converges to $\tau$ in $X^\ks$, there exists $e_0\in\Lambda_\tau$ such that $\tau_{e,\,x}\in [s,X^\ks_{s^*s}]$ for each $e\in\Lambda_\tau$ with $e_0\preceq e$. Using the continuity of $g|_{[s,X^\ks_{s^*s}]}$, now it is easy to see $(g(\tau_{e,\,x}))_{e\in\Lambda_\tau}$ converges to $g(\tau)$.

  Suppose next that $\tau\in X^\ks_{s^*s}$ but $\tau\notin [s,X^\ks_{s^*s}]$. Then $\tau_{e,\,x}\notin [s,X^\ks_{s^*s}]$ for each $e\in\Lambda_\tau$. Indeed, if $\tau_{e,\,x}\in [s,X^\ks_{s^*s}]$ for some $e\in\Lambda_\tau$, then there exists $f\in E(S)$ such that $ef=sf$ and $\tau_{e,\,x}\in X_f^\ks$. However, this implies $\tau\in X^\ks_f$ and hence $\tau=[e,\tau]=[s,\tau]\in[s,X^\ks_{s^*s}]$, a contradiction. Thus $g(\tau)=0=g(\tau_{e,\,x})$ for each $e\in \Lambda_\tau$.

  It remains to consider the case where $\tau\notin X^\ks_{s^*s}$. Then $e\not\le s^*s$ for each $e\in\Lambda_\tau$ since $\tau(s^*s)=\infty_{s^*s}$. It follows $\tau_{e,\,x}(s^*s)=\infty_{s^*s}$, that is, $\tau_{e,\,x}\notin X^\ks_{s^*s}$ for each $e\in \Lambda_\tau$. Thus $g(\tau)=0=g(\tau_{e,\,x})$ for each $e\in\Lambda_\tau$.
\end{proof}

From the arguments above, the equalities \eqref{eq:norm} are refined to 
\begin{equation}
  \label{eq:norm2}
    \begin{aligned}
      \|T\|_{\ks,r}&=\sup\{\|\ind\rho_{e,\,x}(T)\|\mid e\in E(S),\; x\in X_e\},\;\mathrm{and}\\
      \|\varphi(T)\|_r&=\sup\{\|\lambda_{\tau_{e,\,x}}(\varphi(T))\|\mid e\in E(S),\;x\in X_e\}.
    \end{aligned}
\end{equation}
To show $\|T\|_{\ks,r}=\|\varphi(T)\|_r$ for all $T\in S\alg C_c(X)$, we proceed to prove that $\|\ind\rho_{e,\,x}(T)\|=\|\lambda_{\tau_{e,\,x}}(\varphi(T))\|$ for all $e\in E(S)$ and $x\in X_e$.

Recall the notation $S*X=\{(s,x)\in S\times X\mid x\in X_{s^*s}\}$ from Subsection \ref{subsec:1.2}.

For each $(s,x)\in S*X$, define $\tau_{s,\,x}=[s,\tau_{s^*s,\,x}]\in S\ltimes X^\ks$. 
We note that, if $s\in E(S)$, the two definitions of $\tau_{s,\,x}$ coincide by the identification $X^\ks=(S\ltimes X^\ks)^{(0)}$.

\begin{lem}
  \label{lem:domain}
  For any $e\in E(S)$ and $x\in X_e$, 
  \[
    (S\ltimes X^\ks)_{\tau_{e,\,x}}=\{\tau_{s,\,x}\mid s\in S_e\}.
  \]
\end{lem}
\begin{proof}
  Clearly, the domain of $\tau_{s,\,x}$ is $\tau_{e,\,x}$ whenever $s\in S_e$. Conversely, if $[s,\tau_{e,\,x}]\in (S\ltimes X^\ks)_{\tau_{e,\,x}}$, then $\tau_{e,\,x}\in X^\ks_{s^*s}$, that is, $e\le s^*s$. Then $se\in S_e$ and $[s,\tau_{e,\,x}]=[se,\tau_{e,\,x}]=\tau_{se,\,x}$. 
\end{proof}

\begin{lem}
  The map
  \[
    S*X\ni (s,x)\mapsto \tau_{s,\,x}\in S\ltimes X^\ks
  \]
  is injective and its range is dense in $S\ltimes X^\ks$.
\end{lem}
\begin{proof}
  To show the injectivity, suppose that $\tau_{s,\,x}=\tau_{t,\,y}$. Then $\tau_{s^*s,\,x}=\tau_{t^*t,\,y}$, $se=te$ and $\tau_{s^*s,\,x}=\tau_{t^*t,\,y}\in X^\ks_e$ for some $e\in E(S)$. It follows that $x=y$ and $s^*s=t^*t\le e$. Thus we have $s=se=te=t$, which proves the injectivity.
  
  We next show that the range $\{\tau_{s,\,x}\mid (s,x)\in S*X\}$ is dense in $S\ltimes X^\ks$. Fix any nonempty open bisection $V$ in $S\ltimes X^\ks$. Then $d|_{V}\colon V\to d(V)$ is a homeomorphism and $d(V)$ is open in $(S\ltimes X^\ks)^{(0)}$. From Lemma \ref{lem:E(S)-is-dense}, there exist $e\in E(S)$ and $x\in X_e$ with $\tau_{e,\,x}\in d(V)$. Find $\gamma\in V$ with $d(\gamma)=\tau_{e,\,x}$. By the previous lemma, we can write $\gamma=\tau_{s,\,x}$ for some $s\in S_e$. Since $V\in\bis(S\ltimes X^\ks)\setminus\{\emptyset\}$ is arbitrary and $\bis(S\ltimes X^\ks)$ forms an open basis for $S\ltimes X^\ks$, we are done.
\end{proof}

Using the previous two lemmas, we observe that the map
\[
  u_{e,\,x}\colon l^2((S\ltimes X^\ks)_{\tau_{e,\,x}})\ni\delta_{\tau_{t,\,x}}\mapsto \delta_t\in l^2(S_e)
\]
is a unitary map for each $e\in E(S)$ and $x\in X_e$. We proceed to show that $u_{e,\,x}$ gives a unitary equivalence of two representations $\lambda_{\tau_{e,\,x}}\circ\varphi\colon S\alg C_c(X)\to \mathbb{B}(l^2((S\ltimes X^\ks)_{\tau_{e,\,x}}))$ and $\ind\rho_{e,\,x}\colon S\alg C_c(X)\to \mathbb{B}(l^2(S_e))$, which proves $\|\varphi(T)\|_r=\|T\|_{\ks,r}$ by the equalities \eqref{eq:norm2}.

\begin{lem}
  For any $t\in S$ and $x\in X_{t^*t}$,
  \[
    \theta_t^\ks(\tau_{t^*t,\,x})=\tau_{tt^*,\,\theta_t(x)}.
  \]
  In particular, $r(\tau_{t,\,x})=\tau_{tt^*,\,\theta_t(x)}$.
\end{lem}
\begin{proof}
  It is routine to verify this.
\end{proof}

\begin{lem}
  For $s,t\in S$ and $x\in X_{t^*t}$, the following are equivalent.
  \begin{enumerate}
    \item $(S\ltimes X^\ks)_{r(\tau_{t,\,x})}\cap [s,X^\ks_{s^*s}]\neq\emptyset$.
    \item $s^*s\ge tt^*$.
  \end{enumerate}
\end{lem}
\begin{proof}
  Notice that both conditions are equivalent to $\tau_{tt^*,\,\theta_t(x)}\in X^\ks_{s^*s}$.
\end{proof}

\begin{lem}
  \label{lem:unitary-equiv}
  For all $e\in E(S)$ and $x\in X_e$, the diagram 
  \begin{center}
    \begin{tikzpicture}[auto]
      \node(1) at (-0.15,1.5) {$S\alg C_c(X)$};
      \node (2) at (4,1.5) {$\mathcal{C}(S\ltimes X^\ks)$};
      \node (3) at (0,0) {$\mathbb{B}(l^2(S_e))$};
      \node (4) at (4.5,0) {$\mathbb{B}(l^2((S\ltimes X^\ks)_{\tau_{e,\,x}}))$};
      \draw[->] (1) to node {$\scriptstyle\varphi$} (2);
      \draw[->] (0,1.1) to node[swap] {$\scriptstyle\ind\rho_{e,\,x}$} (3);
      \draw[->] (3) to node[swap] {$\scriptstyle\mathrm{Ad}\,u^*_{e,\,x} $} (4);
      \draw[->] (2) to node {$\scriptstyle\lambda_{\tau_{e,\,x}}$} (4,0.35);
    \end{tikzpicture}
  \end{center}
  commutes. In particular, $\|\varphi(T)\|_r=\|T\|_{\ks,r}$ for all $T\in S\alg C_c(X)$.
\end{lem}
\begin{proof}
  It suffices to show
  \[
    u_{e,\,x}\bigl(\lambda_{\tau_{e,\,x}}(\varphi(\delta_s\,f))\,(\delta_{\tau_{t,\,x}})\bigr)=\ind\rho_{e,\,x}(\delta_s\,f)\, \bigl(u_{e,\,x}\delta_{\tau_{t,\,x}}\bigr)
  \]
  for all $s\in S$, $f\in C_c(X_{s^*s})$ and $t\in S_e$. Recall from the previous lemma that the two conditions
  \begin{enumerate}
    \item $(S\ltimes X^\ks)_{r(\tau_{t,\,x})}\cap [s,X^\ks_{s^*s}]\neq\emptyset$.
    \item $s^*s\ge tt^*$.
  \end{enumerate}
  are equivalent. 
  
  First suppose that these equivalent conditions hold. Then the unique element $\alpha$ in $(S\ltimes X^\ks)_{r(\tau_{t,\,x})}\cap [s,X^\ks_{s^*s}]$ can be written as $\alpha=[s,\,\theta^\ks_t(\tau_{t^*t,\,x})]=[s,\tau_{tt^*,\,\theta_t(x)}]$. Using this, we compute that
  \begin{align*}
    u_{e,\,x}\bigl(\lambda_{\tau_{e,\,x}}(\varphi(\delta_s\,f))\,(\delta_{\tau_{t,\,x}})\bigr)
    &=u_{e,\,x}\bigl(\varphi(\delta_s\,f)(\alpha)\,\delta_{\alpha\,\cdot\,\tau_{t,\,x}}\bigr)\\
    &=f(\theta_t(x))\,u_{e,\,x}(\delta_{[st,\,\tau_{t^*t,\,x}]})\\
    &=f(\theta_t(x))\,u_{e,\,x}(\delta_{\tau_{st,\,x}})\\
    &=f(\theta_t(x))\,\delta_{st}.
  \end{align*}
  On the other hand, the equality \eqref{eq:ind} gives 
  \[
    \ind\rho_{e,\,x}(\delta_s\,f)\, \bigl(u_{e,\,x}\delta_{\tau_{t,\,x}}\bigr)=f(\theta_t(x))\,\delta_{st}.
  \]

  If the equivalent conditions do not hold, then we have
  \[
    u_{e,\,x}\bigl(\lambda_{\tau_{e,\,x}}(\varphi(\delta_s\,f))\,(\delta_{\tau_{t,\,x}})\bigr)=0=\ind\rho_{e,\,x}(\delta_s\,f)\, \bigl(u_{e,\,x}\delta_{\tau_{t,\,x}}\bigr).
  \]
  The last statement follows from the equalities \eqref{eq:norm2}.
\end{proof}

\begin{thm}[The Reduced Case of Dynamical Paterson's Theorem]
  \[
    S\ksr C_0(X)\cong C^*_r(S\ltimes X^\ks).
  \]
\end{thm}
\begin{proof}
  Since $S\alg C_c(X)$ is dense in $S\ksr C_0(X)$ and $\varphi(S\alg C_c(X))$ is dense in $C^*_r(S\ltimes X^\ks)$, the isomorphism follows from the previous lemma.
\end{proof}

\subsection{Basic Properties of KS-Groupoids and KS-Crossed Products}
\label{subsec:4.3}

Dynamical Paterson's theorem proved in the previous subsection shows that the full and reduced KS-crossed products are naturally realized as the full and reduced groupoid C*-algebras of the associated KS-groupoid, respectively. 
Therefore, one can study C*-algebraic properties of KS-crossed products through corresponding topological properties of KS-groupoids. 
In this subsection, we apply this point of view to characterize some basic properties of KS-crossed products, such as separability, unitality, and $\sigma$-unitality, in terms of the associated KS-groupoids and the underlying inverse semigroup actions.

\noindent\textbf{Separability of KS-Crossed Products.}
It is well known that, for \'etale groupoids, second countability is equivalent to separability of the associated full and reduced groupoid C*-algebras. For readers' convenience, we first review this fact.

\begin{lem}
  Let $G$ be an \'etale groupoid. Then the following are equivalent.
  \begin{enumerate}
    \item $G$ is second countable.
    \item $C^*(G)$ is separable.
    \item $C_r^*(G)$ is separable.
    \item $G^{(0)}$ is second countable and $G$ is $\sigma$-compact.
  \end{enumerate}
\end{lem}
\begin{proof}
  (i)$\Rightarrow$(ii): Since $\bis(G)$ forms an open basis for $G$, we can find a countable open basis $\mathcal{O}$ with $\mathcal{O}\subset\bis(G)$. 
  From Lemma \ref{lem:partition}, we have $\mathcal{C}(G)=\Span \bigcup_{U\in\mathcal{O}}C_c(U)$. Since the full norm $\|\cdot\|$ on $C_c(U)\subset C^*(G)$ coincides with the sup norm $\|\cdot\|_\infty$ and $U$ is second countable for each $U\in\mathcal{O}$, the subset $C_c(U)$ is separable in $C^*(G)$ for each $U\in\mathcal{O}$. 
  Together with $C^*(G)=\overline{\Span} \bigcup_{U\in\mathcal{O}}C_c(U)$, we deduce that $C^*(G)$ is separable.

  \noindent
  (ii)$\Rightarrow$(iii): This follows by considering the canonical surjection from $C^*(G)$ onto $C^*_r(G)$.

  \noindent
  (iii)$\Rightarrow$(iv): First note that $C_0(G^{(0)})\subset C^*_r(G)$ is separable and hence the unit space $G^{(0)}$ is second countable.
  
  Since $\mathcal{C}(G)$ is dense in $C^*_r(G)$, we can find a countable subset $\{f_n\}_n\subset\mathcal{C}(G)$ such that $\{f_n\}_n$ is dense in $C^*_r(G)$. For each $n$, find a compact subset $K_n$ in $G$ satisfying $f_n(\gamma)=0$ for all $\gamma\in G\setminus K_n$. Then $G=\bigcup_n K_n$. Indeed, if $\gamma\in G\setminus \bigcup_nK_n$, then for any $g\in\mathcal{C}(G)$ with $g(\gamma)=1$, we have
  \[
    \|g-f_n\|_r\ge |\langle\delta_\gamma\mid \lambda_{d(\gamma)}(g-f_n)\,\delta_{d(\gamma)}\rangle|=|g(\gamma)-f_n(\gamma)|=1,
  \]
  where $\lambda_{d(\gamma)}\colon\mathcal{C}(G)\to\mathbb{B}(l^2(G_{d(\gamma)}))$ is the left regular representation of $\mathcal{C}(G)$ at $d(\gamma)\in G^{(0)}$.
  This contradicts the assumption that $\{f_n\}_n$ is dense in $C^*_r(G)$.

  (iv)$\Rightarrow$(i): Since $G$ is $\sigma$-compact, $G$ can be covered by countably many open bisections. On the other hand, every open bisection is second countable as it is homeomorphic to an open subset in the unit space $G^{(0)}$, which is assumed to be second countable. This proves the desired implication.
\end{proof}

From dynamical Paterson's theorem and the lemma above, separability of KS-crossed products is characterized by second countability of the associated KS-groupoid. In Proposition \ref{prop:2nd-ctbl}, the latter condition is expressed in terms of the underlying inverse semigroup action. We summarize this result.

\begin{cor}
  Let $\theta$ be an action of an inverse semigroup $S$ on a locally compact Hausdorff space $X$. Then the following are equivalent.
  \begin{enumerate}
    \item $X$ is second countable and $S^X=\{s\in S\mid X_{s^*s}\neq\emptyset\}$ is countable.
    \item $S\ltimes X^\ks$ is second countable.
    \item $S\kstimes C_0(X)$ is separable.
    \item $S\ksr C_0(X)$ is separable.
  \end{enumerate}
\end{cor}

\noindent\textbf{Unitality and $\sigma$-Unitality of KS-Crossed Products.}
It is also well known that, for \'etale groupoids, ($\sigma$-)compactness of the unit space is equivalent to ($\sigma$-)unitality of the associated full and reduced groupoid C*-algebras. For readers' convenience, we present a proof.

Recall first that the inclusion of $C_0(G^{(0)})$ into $C^*(G)$ (resp.\ $C^*_r(G)$) is non-degenerate, that is, any approximate unit for $C_0(G^{(0)})$ is also an approximate unit for $C^*(G)$ (resp.\ $C^*_r(G)$). 
See \cite[Proposition 3.18]{exel2008}.

\begin{lem}
  Let $G$ be an \'etale groupoid. Then the following are equivalent.
  \begin{enumerate}
    \item $G^{(0)}$ is compact (resp.\ $\sigma$-compact).
    \item $C^*(G)$ is unital (resp.\ $\sigma$-unital).
    \item $C^*_r(G)$ is unital (resp.\ $\sigma$-unital).
  \end{enumerate}
\end{lem}
\begin{proof}
  (i)$\Rightarrow$(ii): Use the fact that the inclusion $C_0(G^{(0)})\subset C^*(G)$ is non-degenerate.

  \noindent
  (ii)$\Rightarrow$(iii): This follows by considering the canonical surjection from $C^*(G)$ onto $C^*_r(G)$.

  \noindent
  (iii)$\Rightarrow$(i): If $C^*_r(G)$ is unital, then the unit belongs to $C_0(G^{(0)})$ since the inclusion $C_0(G^{(0)})\subset C^*_r(G)$ is non-degenerate. Hence $G^{(0)}$ is compact.

  Suppose that $C^*_r(G)$ is $\sigma$-unital. By approximating a countable approximate unit for $C_r^*(G)$ by elements of $\mathcal C(G)$, we may find a sequence $(f_n)_n\subset\mathcal{C}(G)$ such that $\|g*f_n-g\|_r\to 0$ for all $g\in\mathcal{C}(G)$. For each $n$, find a compact subset $K_n\subset G$ such that $f_n(\gamma)=0$ for all $\gamma\in G\setminus K_n$. We claim that $G^{(0)}=\bigcup_nd(K_n)$. 

  Let $x\in G^{(0)}$ and fix $g\in C_c(G^{(0)})\subset\mathcal{C}(G)$ with $g(x)=1$. Then there exists $n$ such that $\|g*f_n-g\|_r<1/2$. Now we observe that $g*f_n(x)=f_n(x)$ and 
  \[
    |f_n(x)-1|=|g*f_n(x)-g(x)|=|\langle\delta_x\mid \lambda_x(g*f_n-g)\,\delta_x\rangle|\le\|g*f_n-g\|_r<1/2.
  \]
  Thus $f_n(x)\neq 0$, in particular, $x\in d(K_n)$. This proves $G^{(0)}=\bigcup_nd(K_n)$.
\end{proof}

From dynamical Paterson's theorem and the lemma above, ($\sigma$-)unitality of the full and reduced KS-crossed products $S\kstimes C_0(X)$, $S\ksr C_0(X)$ is characterized by ($\sigma$-)compactness of the space $X^\ks$. In Theorem \ref{thm:sigma-cpt}, the latter condition is expressed in terms of the underlying inverse semigroup action. We summarize this result.

\begin{cor}
  Let $\theta$ be an action of an inverse semigroup $S$ on a locally compact Hausdorff space $X$. Then the following are equivalent.
  \begin{enumerate}
    \item There exists a finite (resp.\ countable) family of pairs $\{(f_i,K_i)\}_{i\in I}$ such that $f_i\in E(S)$, $K_i\subset X_{f_i}$ is compact for each $i\in I$ and $X_e\subset\bigcup_{i\in I,\, f_i\ge e} K_i$ for each $e\in E(S)$.
    \item $X^\ks$ is compact (resp.\ $\sigma$-compact).
    \item $S\kstimes C_0(X)$ is unital (resp.\ $\sigma$-unital).
    \item $S\ksr C_0(X)$ is unital (resp.\ $\sigma$-unital).
  \end{enumerate}
\end{cor}

\appendix
\section{Exel's Reduced Crossed Products \texorpdfstring{$S\ltimesr A$}{}}
\label{sec:A}

In this section, we review the construction of the reduced crossed product $S\ltimesr A$ due to Exel in \cite{exel2011}. This section also contains a detailed argument for the existence of appropriate cyclic vectors of GNS-representations on the $*$-algebra $S\alg A$.

\begin{lem}
  \label{lem:support}
  Let $A$ be a C*-algebra, $J\triangleleft A$ and $\varphi\in A^*$. The following are equivalent.
  \begin{enumerate}
    \item There exist $\psi\in J^*$ and $x\in J$ such that $\varphi(a)=\psi(xa)$ for all $a\in A$.
    \item There exist $\psi\in J^*$ and $x\in J$ such that $\varphi(a)=\psi(ax)$ for all $a\in A$.
    \item For any approximate unit $(u_\lambda)$ for $J$ and for all $a\in A$, we have
    \[
      \varphi(a)=\lim_\lambda\varphi(u_\lambda a).
    \]
    \item For any approximate unit $(u_\lambda)$ for $J$ and for all $a\in A$, we have
    \[
      \varphi(a)=\lim_\lambda\varphi(au_\lambda).
    \]
  \end{enumerate}
  If $\varphi$ is a state on $A$, these are equivalent to 
  \begin{enumerate}
    \item[(v)] $\pi_\varphi|_J\colon J\to\mathbb{B}(\mathcal{H}_\varphi)$ is non-degenerate, where $(\pi_\varphi,\mathcal{H}_\varphi)$ is the GNS-representation for $\varphi$.
  \end{enumerate}
  If, moreover, $\varphi$ is a pure state on $A$, these are equivalent to $\varphi|_J\neq 0$.
\end{lem}
\begin{proof}
  See Proposition 5.1, 5.4 and 5.5 of \cite{exel2011}.
\end{proof}

\begin{lem}[{\cite[Proposition 5.6]{exel2011}}]
  Let $A$ be an $S$-C*-algebra and $\varphi$ be a pure state on $A$. Set
  \[
    \supp\varphi=\{e\in E(S)\mid \varphi|_{A_e}\neq 0\}.
  \]
  Then $\supp\varphi$ is directed with respect to the preorder $\preceq$ defined by $e\preceq f$ if and only if $f\le e$. 
\end{lem}
\begin{proof}
  It suffices to show that if $\varphi|_I,\varphi|_J\neq 0$, then $\varphi|_{I\cap J}\neq 0$. But this is clear from the lemma above.
\end{proof}

\begin{lem}
  For $s,t\in S$ and $e\in E(S)$, the following are equivalent.
  \begin{enumerate}
    \item $e\le s^*t$.
    \item $e\le s^*s$ and $se=te$.
  \end{enumerate}
\end{lem}
\begin{proof}
  (i)$\Rightarrow$(ii): We have $s^*te=e=t^*se$. Using this, we compute that
  \begin{align*}
    se&=ss^*te=tt^*ss^*te=tt^*se=te\\
    s^*se&=s^*ss^*te=s^*te=e.
  \end{align*}
  (ii)$\Rightarrow$(i): We have $s^*te=s^*se=e$.
\end{proof}

\begin{prop}[cf.\ Proposition 6.9 of \cite{exel2011}]
  Let $\varphi$ be a pure state on $A$ and $e\in\supp\varphi$. Define a linear functional $\varphi_e\colon S\alg A\to\mathbb{C}$ by
  \begin{align*}
    \varphi_e(\delta_s\,x)=
    \begin{cases}
      \varphi(x) & \mathrm{if}\;s\ge e\\
      0 & \mathrm{otherwise}.
    \end{cases}
  \end{align*}
  Then the following hold.
  \begin{enumerate}
    \item $\varphi_e(T^*T)\ge 0$ for any $T\in S\alg A$.
    \item $\varphi_e$ is contractive with respect to the $L^1$-norm on $S\alg A$.
    \item $\varphi_e(\delta_t\,a)=\varphi_e(\delta_s\,a)$ for any $e\le s\le t$ and $a\in A_{s^*s}$.
  \end{enumerate}
\end{prop}
\begin{proof}
  (ii) and (iii) are trivial.

  For (i), write $T=\sum_{s\in F}\delta_s\,a_s$ for some finite subset $F\subset S$. Let $\{F_i\mid i\in I\}$ be the set of all equivalence classes of $F':=\{s\in F\mid s^*s\ge e\}$ with respect to the equivalence relation given by $s\sim t$ if and only if $se=te$.

  From $e\in\supp\varphi$, we can find a bounded linear functional $\psi$ on $A_e$ and $x\in A_e$ such that $\varphi(a)=\psi(xa)$ for all $a\in A$. 
  
  Using the lemma above, we compute that
  \begin{align*}
    \varphi_e(T^*T)
    &=\varphi_e\Bigl(\sum_{s,t\in F}(\delta_s\,a_s)^*(\delta_t\,a_t)\Bigr)\\
    &=\sum_{\substack{s,t\in F\\ s^*t\ge e}}\varphi(q_{t^*s}(a_s^*)a_t)\\
    &=\sum_{\substack{s,t\in F\\ s^*t\ge e}}\psi(xq_{t^*s}(a_s^*)a_t)\\
    &=\sum_{\substack{s,t\in F\\ s^*t\ge e}}\psi(xa_s^*a_t)\\
    &=\sum_{i\in I}\sum_{s,t\in F_i}\varphi(a_s^*a_t)\\
    &=\sum_{i\in I}\varphi\Bigl(\Bigl(\sum_{s\in F_i}a_s\Bigr)^*\Bigl(\sum_{t\in F_i}a_t\Bigr)\Bigr)\ge 0,
  \end{align*}
  where we used the observation in the fourth equality that $q_{t^*s}(y)=y$ for all $y\in A_e$.
\end{proof}

\begin{prop}[cf.\ Proposition 7.4 of \cite{exel2011}]
  \label{prop:pure-state}
  Let $\varphi$ be a pure state on $A$ and define a linear functional $\overline{\varphi}\colon S\alg A\to\mathbb{C}$ by
  \begin{align*}
    \overline{\varphi}(\delta_s\,a)=
    \begin{cases}
      \varphi(a) & \mathrm{if}\;e\le s\;\mathrm{for\;some\;}e\in\supp\varphi\\
      0 & \mathrm{otherwise}.
    \end{cases}
  \end{align*}
  Then the following hold.
  \begin{enumerate}
    \item The net $(\varphi_e)_{e\in\supp\varphi}$ converges pointwise to $\overline{\varphi}$.
    \item $\overline{\varphi}(T^*T)\ge 0$ for all $T\in S\alg A$.
    \item $\overline{\varphi}$ is contractive with respect to the $L^1$-norm on $S\alg A$.
    \item $\mathcal{N}_A\subset\ker\overline{\varphi}$.
  \end{enumerate}
\end{prop}
\begin{proof}
  For (i), it suffices to show that
  \[
    \overline{\varphi}(\delta_s\,a)=\lim_e\varphi_e(\delta_s\,a)
  \]
  for all $s\in S$ and $a\in A_{s^*s}$. 

  Consider first the case $e\not\le s$ for any $e\in\supp\varphi$. Then we have
  \[
    \overline{\varphi}(\delta_s\,a)=0=\varphi_e(\delta_s\,a)
  \]
  for each $e\in\supp\varphi$.

  Suppose next that $e\le s$ for some $e\in\supp\varphi$. Then $\overline{\varphi}(\delta_s\,a)=\varphi(a)$ and 
  \[
    \varphi_f(\delta_s\,a)=\varphi(a)=\varphi_e(\delta_s\,a)
  \]
  for each $f\in\supp\varphi$ with $e\preceq f$. This proves (i).

  (ii) and (iii) follow immediately from (i) and the proposition above.

  We now show (iv). Fix $t\ge s$ and $a\in A_{s^*s}$. We check that $\overline{\varphi}(\delta_t\,a-\delta_s\,a)=0$.

  \noindent\textbf{Case 1}\quad $e\le s\le t$ for some $e\in\supp\varphi$.

  In this case, $\overline{\varphi}(\delta_t\,a-\delta_s\,a)=\varphi(a)-\varphi(a)=0$.

  \noindent\textbf{Case 2}\quad $f\not\le s$ for any $f\in\supp\varphi$, but $e\le t$ for some $e\in\supp\varphi$.

  Then, by Lemma \ref{lem:support}, $\overline{\varphi}(\delta_t\,a-\delta_s\,a)=\varphi(a)=\lim_\lambda\varphi(au_\lambda)$, where $(u_\lambda)_\lambda$ is an approximate unit for $A_e$. By using $e,s\le t$, we have
  \[
    es=ets^*s=es^*s\in E(S)
  \]
  and $es\le s$. From the assumption, it follows $es\notin\supp\varphi$, that is, $\varphi$ vanishes on $A_{es}=A_{es^*s}$. Since $au_\lambda\in A_{es^*s}$, we obtain
  \[
    \overline{\varphi}(\delta_t\,a-\delta_s\,a)=\varphi(a)=\lim_\lambda\varphi(au_\lambda)=0,
  \]
  as desired.

  \noindent\textbf{Case 3}\quad $f\not\le t$ for any $f\in \supp\varphi$.

  Then $\overline{\varphi}(\delta_t\,a-\delta_s\,a)=0-0=0$.
\end{proof}

Since $\overline{\varphi}$ is contractive with respect to the $L^1$-norm on $S\alg A$, it follows from VI.18.14 of \cite{fell-doran} that
\[
  \overline{\varphi}(T_2^*T_1^*T_1T_2)\le\|T^*_1T_1\|_{L^1}\overline{\varphi}(T_2^*T_2)
\]
for all $T_1,T_2\in S\alg A$. Hence we can consider the GNS-representation $(\pi_{\overline{\varphi}},\mathcal{H}_{\overline{\varphi}})$ for $\overline{\varphi}$ on $S\alg A$ (see VI.19.3 of \cite{fell-doran}). From $\mathcal{N}_A\subset \ker\overline{\varphi}$, it follows that $\pi_{\overline{\varphi}}$ is an admissible representation for $S\alg A$.

For the proof of the reduced case of Theorem \ref{thm:isom}, we show that the GNS-representation $(\pi_{\overline{\varphi}},\mathcal{H}_{\overline{\varphi}})$ for $\overline{\varphi}$ has a cyclic vector $\xi_{\overline{\varphi}}$ satisfying
\[
  \overline{\varphi}(T)=\langle\xi_{\overline{\varphi}}|\pi_{\overline{\varphi}}(T)\xi_{\overline{\varphi}}\rangle
\]
for all $T\in S\alg A$. For this, in view of VI.18.7 and VI.19.6 of \cite{fell-doran}, it suffices to check that
\begin{enumerate}
  \item the functional $\overline{\varphi}$ is self-adjoint, that is, $\overline{\overline{\varphi}(T)}=\overline{\varphi}(T^*)$ for all $T\in S\alg A$, and
  \item we have $|\overline{\varphi}(T)|^2\le\overline{\varphi}(T^*T)$ for all $T\in S\alg A$.
\end{enumerate}

\begin{lem}
  For any pure state $\varphi$ on $A$ and $e\in\supp\varphi$, the functional $\varphi_e$ is self-adjoint. In particular, $\overline{\varphi}$ is self-adjoint.
\end{lem}
\begin{proof}
  It suffices to show $\overline{\varphi_e(\delta_s\,a)}=\varphi_e((\delta_s\,a)^*)$ for all $s\in S$ and $a\in A_{s^*s}$. 

  Compute that
  \begin{align*}
    \overline{\varphi_e((\delta_s\,a))}&=
    \begin{cases}
      \varphi(a^*)& \mathrm{if}\;s\ge e\\
      0 & \mathrm{otherwise}
    \end{cases}\\
    \varphi_e((\delta_s\,a)^*)
    &=\varphi_e(\delta_{s^*}\,\alpha_s(a^*))\\
    &=\begin{cases}
      \varphi(\alpha_s(a^*)) & \mathrm{if}\;s^*\ge e\\
      0 & \mathrm{otherwise}.
    \end{cases}
  \end{align*}
  Note that $s\ge e$ if and only if $s^*\ge e$. Since $e\in\supp\varphi$, we can find a bounded linear functional $\psi$ on $A_e$ and $x\in A_e$ such that $\varphi(b)=\psi(bx)$ for all $b\in A$. Then we have
  \[
    \varphi(\alpha_s(a^*))
    =\psi(\alpha_s(a^*)x)
    =\psi(a^*x)
    =\varphi(a^*).\qedhere
  \]
\end{proof}

\begin{lem}
  For any pure state $\varphi$ on $A$ and $e\in\supp\varphi$,
  \[
    |\varphi_e(T)|^2\le \varphi_e(T^*T)
  \]
  for all $T\in S\alg A$. In particular, the same inequality holds for $\overline{\varphi}$.
\end{lem}
\begin{proof}
  We first claim that for any approximate unit $(u_\lambda)_\lambda$ for $A_e$, we have 
  \[
    \varphi_e(T)=\lim_\lambda\varphi_e(T\cdot\delta_e\,u_\lambda).
  \]
  To see this, it suffices to check this for $T=\delta_s\,a$, where $s\in S$ and $a\in A_{s^*s}$. Observe that
  \begin{align*}
    \varphi_e(\delta_s\,a\cdot\delta_e\,u_\lambda)
    &=\varphi_e(\delta_{se}\,au_\lambda)\\
    &=\begin{cases}
      \varphi(au_\lambda) & \mathrm{if}\;se\ge e\\
      0 & \mathrm{otherwise}.
    \end{cases}
  \end{align*}
  Note that $s\ge e$ if and only if $se\ge e$. Hence this proves the claim.

  Using the equality which we have shown above, compute that
  \begin{align*}
    |\varphi_e(T)|^2
    &=|\varphi_e(T^*)|^2\\
    &=\lim_\lambda|\varphi_e(T^*\cdot\delta_e\,u_\lambda)|^2\\
    &\le \sup_\lambda\varphi_e(T^*T)\varphi_e((\delta_e\,u_\lambda)^*(\delta_e\,u_\lambda))\\
    &=\sup_\lambda\varphi_e(T^*T)\varphi(u_\lambda^2)\\
    &\le \varphi_e(T^*T).\qedhere
  \end{align*}
\end{proof}

Now we obtain:

\begin{prop}
  \label{prop:cyclic}
  Let $S$ be an inverse semigroup, $(A,\alpha)$ be an $S$-C*-algebra and let $\varphi$ be a pure state on $A$. Then the GNS-representation $(\pi_{\overline{\varphi}},\mathcal{H}_{\overline{\varphi}})$ for $\overline{\varphi}\colon S\alg A\to\mathbb{C}$ has a cyclic vector $\xi_{\overline{\varphi}}$ satisfying
  \[
    \overline{\varphi}(T)=\langle\xi_{\overline{\varphi}}|\pi_{\overline{\varphi}}(T)\xi_{\overline{\varphi}}\rangle
  \]
  for all $T\in S\alg A$.
\end{prop}

As pointed out in \cite[Corollary 4.6]{ns}, for any $E(S)$-C*-algebra $A$, the full crossed product $E(S)\ltimes A$ coincides with $A$. Although this is not needed in this paper, we present an alternative proof together with the reduced case of this isomorphism as it provides an interesting comparison between crossed products and KS-crossed products (cf.\ Proposition \ref{prop:spectrum}).

Observe that
\[
  A^0:=\sum_{e\in E(S)}A_e
\]
is a dense $*$-ideal of $A$.

\begin{prop}
  The map $\pi^\ks\colon E(S)\alg A\ni\delta_e\,x\mapsto x\in A^0$ is a $*$-homomorphism, whose kernel is 
  \[
    \mathcal{N}_A=\Span\{\delta_e\,x-\delta_f\,x\mid e\ge f,\;x\in A_f\}.
  \]
\end{prop}
\begin{proof}
  It is routine to verify that $\pi^\ks$ is a $*$-homomorphism whose kernel contains $\mathcal{N}_A$. It suffices to show that 
  \[
    \sum_{i=1}^n\delta_{e_i}\,x_i\in\mathcal{N}_A
  \]
  for any $n\ge 1$, $e_1,\ldots,e_n\in E(S)$ and $x_i\in A_{e_i}\;(i=1,\ldots,n)$ with $\sum_{i=1}^n\delta_{e_i}\,x_i\in\ker\pi^\ks$.

  We prove this by induction on $n\ge 1$.

  If $n=1$, the claim is trivially true.

  Suppose the claim is true for $n=k\ge 1$. Consider the case where $n=k+1$. First notice that 
  \[
    x_{k+1}=-\sum_{i=1}^kx_i\in A_{e_{k+1}}\cap\Bigl(\sum_{i=1}^kA_{e_i}\Bigr)=\sum_{i=1}^kA_{e_{k+1}}\cap A_{e_i}.
  \]
  Hence we can write
  \[
    x_{k+1}=\sum_{i=1}^ka_i,\quad a_i\in A_{e_{k+1}}\cap A_{e_i}\;(i=1,\ldots,k).
  \]
  Then $\sum_{i=1}^k(x_i+a_i)=-x_{k+1}+x_{k+1}=0$ and $x_i+a_i\in A_{e_i}$ for each $i=1,\ldots,k$. Thus, by the assumption of the induction, we obtain $\sum_{i=1}^k\delta_{e_i}(x_i+a_i)\in\mathcal{N}_A$. Since $\delta_{e_{k+1}}\,a_i-\delta_{e_i}\,a_i\in\mathcal{N}_A$ for each $i=1,\ldots,k$, it follows that
  \begin{align*}
    \sum_{i=1}^{k+1}\delta_{e_i}\,x_i
    &=\sum_{i=1}^k\delta_{e_i}\,x_i+\sum_{i=1}^k\delta_{e_{k+1}}a_i\\
    &=\sum_{i=1}^k\delta_{e_i}\,(x_i+a_i)+\sum_{i=1}^k(\delta_{e_{k+1}}a_i-\delta_{e_i}\,a_i)\in\mathcal{N}_A.\qedhere
  \end{align*}
\end{proof}

\begin{lem}
  Let $B$ be a C*-algebra and let $\pi\colon A^0\to B$ be a $*$-homomorphism.
  If $\pi|_{A_e}\colon A_e\to B$ is injective for each $e\in E(S)$, then $\pi\colon A^0\to B$ is injective.
\end{lem}
\begin{proof}
  Suppose $\pi(x)=0$ for some $x\in A^0$.
  Write $x=\sum_{i=1}^na_i$ for some $a_i\in A_{e_i}\,(i=1,\ldots,n)$.
  Then $\pi(a_j^*x)=0$ for each $j=1,\ldots,n$.
  Since $a_j^*x\in A_{e_j}$ and $\pi|_{A_{e_j}}$ is injective, we obtain $a_j^*x=0$ for each $j$.
  Thus $x^*x=0$ in $A^0\subset A$, from which it follows $x=0$.
\end{proof}

\begin{cor}
  \[
    E(S)\ltimes A\cong A\cong E(S)\ltimesr A
  \]
\end{cor}
\begin{proof}
  Now we have the commutative diagram
  \begin{center}
    \begin{tikzpicture}[auto]
      \node (0) at (0,1.55) {$A^0$};
      \node (1) at (2.5,1.5) {$E(S)\alg A\,/\,\mathcal{N}_A$};
      \node (2) at (5.9,1.5) {$E(S)\ltimes A$};
      \node (3) at (5.9,0) {$E(S)\ltimesr A$};
      \node (0-1) at (0.7,1.5) {$\cong$};
      \draw[->] (1) to node {} (2);
      \draw[->] (1) to node[swap] {} (3);
      \draw[->>] (2) to node {} (3);
    \end{tikzpicture}
  \end{center}
  From Proposition \ref{prop:pure-state}, the map $A_e\ni x\mapsto [\delta_e\,x]\in E(S)\ltimesr A$ is injective for each $e\in E(S)$. Since $A_e$ is a $*$-ideal in $A^0$, it follows that the map from $A^0$ to $E(S)\ltimesr A$ is injective.
\end{proof}

\section{Proof of Theorem \ref{thm:isom-of-KS}}
\label{sec:B}

In this appendix, we give a direct proof of Theorem \ref{thm:isom-of-KS} without using the theory of covariant representations for KS-crossed products. The proof proceeds by constructing mutually inverse $*$-homomorphisms on the algebraic crossed products and then passing to the corresponding C*-algebras. For this purpose, the following elementary lemma is useful for treating KS-crossed products in a purely algebraic way.

\begin{lem}
  \label{lem:alg-cp}
  Let $(B,\beta)$ be an $S$-C*-algebra and let $B^0$ be a $*$-subalgebra of $B$ satisfying
  \begin{enumerate}
    \item for each $e\in E(S)$, $B^0$ has a $*$-ideal $B^0_e$ which is dense in $B_e$,
    \item $B^0_e\subset B^0_f$ for each $e,f\in E(S)$ with $e\le f$,
    \item $B^0_eB^0_f\subset B^0_{ef}$ for each $e,f\in E(S)$,
    \item $p(b)\le \|b\|$ for any C*-seminorm $p$ on $B^0_e$ and $b\in B^0_e$, and
    \item $\beta_s(B^0_{s^*s})\subset B^0_{ss^*}$ for each $s\in S$.
  \end{enumerate}
  Then $S\alg B^0:=\Span\{\delta_s\,x\mid x\in B^0_{s^*s}\}$ is a $*$-subalgebra of $S\alg B$. 
  
  For any C*-algebra $D$ and $*$-homomorphism $\pi^0\colon S\alg B^0\to D$, there exists a unique $*$-homomorphism $\pi\colon S\alg B\to D$ such that the diagram 
  \begin{center}
    \begin{tikzpicture}[auto]
      \node (1) at (0,1.2) {$S\alg B^0$};
      \node (2) at (2.2,1.2) {$S\alg B$};
      \node (3) at (2.2,0) {$D$};
      \draw[->] (1) to node {} (2);
      \draw[->] (1) to node[swap] {$\scriptstyle\pi^0$} (3);
      \draw[->] (2) to node {$\scriptstyle\pi$} (3);
    \end{tikzpicture}
  \end{center}
  commutes. If, moreover, $\pi^0\colon S\alg B^0\to D$ satisfies $\pi^0(\delta_t\,x-\delta_s\,x)=0$ for all $t\ge s$ and $x\in B^0_{s^*s}$, then the above $*$-homomorphism $\pi\colon S\alg B\to D$ is admissible.
\end{lem}
\begin{proof}
  It is straightforward to check that $S\alg B^0$ is a $*$-algebra.

  Since $B^0_e\ni y\mapsto \pi^0(\delta_e\,y)\in D$ is a $*$-homomorphism, the condition (iv) implies that $\|\pi^0(\delta_e\,y)\|\le \|y\|$ for each $e\in E(S)$ and $y\in B_e^0$. Using this, we have
  \[
    \|\pi^0(\delta_s\,x)\|^2=\|\pi^0(\delta_{s^*s}\,x^*x)\|\le \|x^*x\|=\|x\|^2
  \]
  for all $s\in S$ and $x\in B^0_{s^*s}$. Hence $B^0_{s^*s}\ni x\mapsto \pi^0(\delta_s\,x)\in D$ uniquely extends to the map $\pi_s\colon B_{s^*s}\to D$ by (i). Let
  \[
    \pi\colon S\alg B\ni\delta_s\,x\mapsto \pi_s(x)\in D.
  \]
  By (v), it is easily verified that this is a $*$-homomorphism. 

  The last statement is clear from the construction of the $*$-homomorphism $\pi$.
\end{proof}

\begin{lem}
  Define a linear map $\Phi\colon S\alg A\to S\alg A^\ks$ by
  \[
    \Phi(\delta_s\,x)=\delta_s\,(\delta_{s^*s}\,x),
  \]
  for all $s\in S$ and $x\in A_{s^*s}$. Then $\Phi$ is a $*$-homomorphism.
\end{lem}
\begin{proof}
  Compute that
  \begin{align*}
    \Phi(\delta_s\,x\cdot \delta_t\,y)
    &=\Phi(\delta_{st}\,\alpha_{t^*}(x\alpha_t(y)))\\
    &=\delta_{st}\,(\delta_{t^*s^*st}\,\alpha_{t^*}(x\alpha_t(y))),\\
    \Phi(\delta_s\,x)\Phi(\delta_t\,y)
    &=\delta_s\,(\delta_{s^*s}\,x)\cdot\delta_t\,(\delta_{t^*t}\,y)\\
    &=\delta_{st}\,(\alpha_{t^*}^\ks(\delta_{s^*s}\,x\cdot\alpha_t^\ks(\delta_{t^*t}\,y)))\\
    &=\delta_{st}\,(\alpha_{t^*}^\ks(\delta_{s^*s}\,x\cdot\delta_{tt^*}\,\alpha_t(y)))\\
    &=\delta_{st}\,(\alpha_{t^*}^\ks(\delta_{s^*stt^*}\,x\alpha_t(y)))\\
    &=\delta_{st}\,(\delta_{t^*s^*stt^*t}\,\alpha_{t^*}(x\alpha_t(y)))\\
    &=\delta_{st}\,(\delta_{t^*s^*st}\,\alpha_{t^*}(x\alpha_t(y)))
  \end{align*}
  and
  \begin{align*}
    \Phi((\delta_s\,x)^*)
    &=\Phi(\delta_{s^*}\,\alpha_s(x^*))
    =\delta_{s^*}\,(\delta_{ss^*}\,\alpha_s(x^*)),\\
    \Phi(\delta_s\,x)^*
    &=(\delta_s\,(\delta_{s^*s}\,x))^*
    =\delta_{s^*}\,(\alpha_{s}^\ks(\delta_{s^*s}\,x^*))
    =\delta_{s^*}\,(\delta_{ss^*}\,\alpha_s(x^*)).
  \end{align*}
  Hence $\Phi$ is a $*$-homomorphism.
\end{proof}

\begin{rem}
  From the universality of $S\kstimes A$, the map $\Phi$ above induces a $*$-homomorphism $\overline{\Phi}\colon S\kstimes A\to S\ltimes A^\ks$ such that
  \[
    \overline{\Phi}(\delta_s\,x)=[\delta_s\,(\delta_{s^*s}\,x)]
  \]
  for all $s\in S$ and $x\in A_{s^*s}$.
\end{rem}

We next construct an admissible $*$-homomorphism $\Psi$ from $S\alg A^\ks$ to $S\kstimes A$ which induces an inverse for $\overline{\Phi}$.

\begin{lem}
  There exists a unique $*$-homomorphism $\Psi:S\alg A^\ks\to S\kstimes A$ such that
  \[
    \Psi(\delta_s\,(\delta_e\,x))=\delta_{se}\,x,
  \]
  for all $e\le s^*s$ and $x\in A_e$. Moreover, this $*$-homomorphism is admissible.
\end{lem}
\begin{proof}
  We first show that the linear map 
  \[
    \Psi_0\colon S\alg (E(S)\alg A)\ni\delta_s\,(\delta_e\,x)\mapsto \delta_{se}\,x\in S\kstimes A
  \]
  is a $*$-homomorphism.

  Compute that
  \begin{align*}
    \Psi_0(\delta_s\,(\delta_e\,x)\cdot\delta_t\,(\delta_f\,y))
    &=\Psi_0(\delta_{st}\,\alpha_{t^*}^\ks(\delta_e\,x\cdot\alpha_t^\ks(\delta_f\,y)))\\
    &=\Psi_0(\delta_{st}\,\alpha_{t^*}^\ks(\delta_e\,x\cdot\delta_{tft^*}\,\alpha_t(y)))\\
    &=\Psi_0(\delta_{st}\,\alpha_{t^*}^\ks(\delta_{etft^*}\,x\alpha_t(y)))\\
    &=\Psi_0(\delta_{st}\,(\delta_{t^*etft^*t}\,\alpha_{t^*}(x\alpha_t(y))))\\
    &=\Psi_0(\delta_{st}\,(\delta_{t^*etf}\,\alpha_{t^*}(x\alpha_t(y))))\\
    &=\delta_{stt^*etf}\,\alpha_{t^*}(x\alpha_t(y))\\
    &=\delta_{setf}\,\alpha_{t^*}(x\alpha_t(y))\\
    \Psi_0(\delta_s\,(\delta_e\,x))\Psi_0(\delta_t\,(\delta_f\,y))
    &=\delta_{se}\,x\cdot\delta_{tf}\,y\\
    &=\delta_{setf}\,\alpha_{ft^*}(x\alpha_{tf}(y))\\
    &=\delta_{setf}\,\alpha_{t^*}(x\alpha_{t}(y)).
  \end{align*}
  and
  \begin{align*}
    \Psi_0((\delta_s\,(\delta_e\,x))^*)
    &=\Psi_0(\delta_{s^*}\,(\delta_{ses^*}\,\alpha_{s}(x^*)))\\
    &=\delta_{s^*ses^*}\,\alpha_s(x^*)\\
    &=\delta_{es^*}\,\alpha_s(x^*)\\
    \Psi_0((\delta_s\,(\delta_e\,x)))^*
    &=(\delta_{se}\,x)^*\\
    &=\delta_{es^*}\,\alpha_{se}(x^*)\\
    &=\delta_{es^*}\,\alpha_s(x^*).
  \end{align*}
  Hence $\Psi_0$ is a $*$-homomorphism. 
  From Lemmas \ref{lem:alg-cp} and \ref{lem:closed}, it uniquely extends to a $*$-homomorphism $\Psi\colon S\alg A^\ks\to S\kstimes A$.

  To show that $\mathcal{N}_{A^\ks}=\Span\{\delta_t\,z-\delta_s\,z\mid t\ge s,\;z\in A^\ks_{s^*s}\}\subset\ker\Psi$, it suffices to show that
  \[
    \Psi(\delta_t\,(\delta_e\,x)-\delta_s\,(\delta_e\,x))=0
  \]
  for all $s\le t$, $e\le s^*s$ and $x\in A_e$. But this is clear since $s\le t$ implies $se=(ts^*s)e=te$.
\end{proof}

\begin{rem}
  From the universality of $S\ltimes A^\ks$, there exists a unique $*$-homomorphism $\overline{\Psi}\colon S\ltimes A^\ks\to S\kstimes A$ such that $\overline{\Psi}([\delta_s\,(\delta_e\,x)])=\delta_{se}\,x$ for all $e\le s^*s$ and $x\in A_e$.
\end{rem}

\noindent\textbf{Proof of Theorem \ref{thm:isom-of-KS}:}

  For $s\in S$ and $a\in A_{s^*s}$, compute that
  \begin{align*}
    \overline{\Psi}\circ\overline{\Phi}(\delta_s\,a)
    &=\overline{\Psi}([\delta_s\,(\delta_{s^*s}\,a)])\\
    &=\delta_{ss^*s}\,a\\
    &=\delta_s\,a.
  \end{align*}
  Hence $\overline{\Psi}\circ\overline{\Phi}=\mathrm{id}_{S\kstimes A}$.

  For $s\in S$, $e\le s^*s$ and $x\in A_e$, compute that 
  \begin{align*}
    \overline{\Phi}\circ\overline{\Psi}([\delta_s\,(\delta_e\,x)])
    &=\overline{\Phi}(\delta_{se}\,x)\\
    &=[\delta_{se}\,(\delta_{es^*se}\,x)]\\
    &=[\delta_s\,(\delta_e\,x)].
  \end{align*}
  Since 
  \[
    \Span\{[\delta_s\,(\delta_e\,x)]\mid s\in S,\;e\le s^*s,\;x\in A_e\}
  \]
  is dense in $S\ltimes A^\ks$, it follows $\overline{\Phi}\circ\overline{\Psi}=\mathrm{id}_{S\ltimes A^\ks}$. 
\qed

\section{Proof of Theorem \ref{thm:isom}}
\label{sec:C}

Let $\theta$ be an action of an inverse semigroup $S$ on a locally compact Hausdorff space $X$.
Recall that the subset 
\[
  \Theta_s:=[s,X_{s^*s}]=\{[s,x]\mid x\in X_{s^*s}\}
\]
is an open bisection in $S\ltimes X$. Hence
\[
  d_s:=d|_{\Theta_s}\colon \Theta_s\to X_{s^*s}
\]
is a homeomorphism. 

Notice that the assignments $B=C_0(X)$, $B^0=C_c(X)$ and $B^0_e=C_c(X_e)$ for each $e\in E(S)$ satisfy the assumptions of Lemma \ref{lem:alg-cp}. Hence we can consider the $*$-subalgebra $S\alg C_c(X)$ of $S\alg C_0(X)$. The following lemma is a restatement of the argument in Section 7 of \cite{exel2008}.

\begin{lem}
  Define a linear map $\Phi_0\colon S\alg C_c(X)\to \mathcal{C}(S\ltimes X)$ by
  \[
    \Phi_0(\delta_s\,f) = f\circ d_s
  \]
  for all $s\in S$ and $f\in C_c(X_{s^*s})$. Then $\Phi_0$ is a $*$-homomorphism.
\end{lem}
\begin{proof}
  Since $d_s=d|_{\Theta_s}\colon \Theta_s\to X_{s^*s}$ is a homeomorphism for each $s\in S$, $\Phi_0$ is well-defined. 

  We next show that $\Phi_0$ is multiplicative. To this end, it suffices to show that 
  \[
    \Phi_0(\delta_s\,f)*\Phi_0(\delta_t\,g)=\Phi_0(\delta_s\,f\cdot\delta_t\,g)
  \]
  for each $s,t\in S$, $f\in C_c(X_{s^*s})$ and $g\in C_c(X_{t^*t})$. Note first that since $\Phi_0(\delta_s\,f)\in C_c(\Theta_s)$ and $\Phi_0(\delta_t\,g)\in C_c(\Theta_t)$, we have 
  \[
    \Phi_0(\delta_s\,f)*\Phi_0(\delta_t\,g),\,\Phi_0(\delta_s\,f\cdot\delta_t\,g)\,\in\, C_c(\Theta_{st}).
  \]
  For $[st,x]\in \Theta_{st}$, compute that
  \begin{align*}
    (\Phi_0(\delta_s\,f)*\Phi_0(\delta_t\,g))([st,x])
    &=(f\circ d_s\,*\, g\circ d_t)([st,x])\\
    &=f\circ d_s([s,\theta_t(x)])\cdot g\circ d_t([t,x])\\
    &=f(\theta_t(x))g(x)\\
    \Phi_0(\delta_s\,f\cdot \delta_t\,g)([st,x])
    &=(\alpha_{t^*}(f\alpha_t(g))\circ d_{st})([st,x])\\
    &=\alpha_{t^*}(f\alpha_t(g))(x)\\
    &=f(\theta_t(x))g(x).
  \end{align*}
  It only remains to prove that $\Phi_0$ preserves the involution, that is, $\Phi_0(\delta_s\,f)^*=\Phi_0((\delta_s\,f)^*)$ for each $s\in S$ and $f\in C_c(X_{s^*s})$. First note that $\Phi_0(\delta_s\,f)^*,\Phi_0((\delta_s\,f)^*)\in C_c(\Theta_{s^*})$. For $[s^*,x]\in \Theta_{s^*}$, compute that
  \begin{align*}
    \Phi_0(\delta_s\,f)^*([s^*,x])
    &=(f\circ d_s)^*([s^*,x])\\
    &=\overline{(f\circ d_s)([s^*,x]^{-1})}\\
    &=\overline{(f\circ d_s)([s,\theta_{s^*}(x)])}\\
    &=\overline{f(\theta_{s^*}(x))}
  \end{align*}
  and
  \begin{align*}
    \Phi_0((\delta_s\,f)^*)([s^*,x])
    &=\Phi_0(\delta_{s^*}\,\overline{\alpha_s(f)})([s^*,x])\\
    &=(\overline{\alpha_s(f)\circ d_{s^*}})([s^*,x])\\
    &=\overline{\alpha_s(f)(x)}\\
    &=\overline{f(\theta_{s^*}(x))}.\qedhere
  \end{align*}
\end{proof}

\begin{rem}
  Let  $s\ge t$ and $f\in C_c(X_{t^*t})$. Then since $d_s\colon \Theta_{s}\to X_{s^*s}$ is an extension of $d_t\colon\Theta_{t}\to X_{t^*t}$, it follows that
  \[
    \Phi_0(\delta_s\,f)-\Phi_0(\delta_t\,f)=f\circ d_s-f\circ d_t=0 \quad\mathrm{in}\quad \mathcal{C}(S\ltimes X).
  \]
  Using this equality together with Lemma \ref{lem:alg-cp}, the $*$-homomorphism $\Phi_0\colon S\alg C_c(X)\to \mathcal{C}(S\ltimes X)\subset C^*(S\ltimes X)$ induces an admissible $*$-homomorphism from $S\alg C_0(X)$ to $C^*(S\ltimes X)$. Moreover, from the universality of $S\ltimes C_0(X)$, we obtain a $*$-homomorphism $\Phi\colon S\ltimes C_0(X)\to C^*(S\ltimes X)$ such that
  \[
    \Phi([\delta_s\,f])=f\circ d_s
  \]
  for all $s\in S$ and $f\in C_c(X_{s^*s})$.
\end{rem}

It is natural to expect that the family of linear maps $\{\Psi_s\}_{s\in S}$, given by 
\[
  \Psi_s\colon C_c([s,X_{s^*s}])\ni g\mapsto [\delta_s\, g\circ (d_s)^{-1}]\in S\ltimes C_0(X)
\]
for each $s\in S$, gives rise to a $*$-homomorphism $\Psi$ from $C^*(S\ltimes X)$ to $S\ltimes C_0(X)$, which is an inverse of $\Phi$. However, it is somewhat non-trivial that the family of maps $\{\Psi_s\}$ defines a well-defined map from $\mathcal{C}(S\ltimes X)$ to $S\ltimes C_0(X)$. For this difficulty, we need the following technical lemma.

\begin{prop}
  \label{prop:technical-lemma}
  Suppose that $G$ is an \'{e}tale groupoid, $S$ is an inverse semigroup and there is a semigroup homomorphism $S\ni s\mapsto U_s\in\bis(G)$ satisfying the following conditions.
\begin{enumerate}
  \item $\bigcup_{s\in S}U_s=G$.
  \item For any $s,t\in S$ and $\gamma\in U_s\cap U_t$, there exists $w\in S$ such that $\gamma\in U_w$ and $w\le s,t$. In other words, 
  \[
    U_s\cap U_t=\bigcup_{w\le s,t}U_w
  \]
  for any $s,t\in S$.
\end{enumerate}
  Let $A$ be a C*-algebra and let $\varphi_s\colon C_c(U_s)\to A$ be a linear map for each $s\in S$.

  Suppose that the following conditions hold.
  \begin{enumerate}
    \item[(1)] $\varphi_{e_2}|_{C_c(U_{e_1})}=\varphi_{e_1}$ for any $e_1\le e_2$ in $E(S)$.
    \item[(2)] $\varphi_{st}(f*g)=\varphi_s(f)\varphi_t(g)$ for any $s,t\in S$, $f\in C_c(U_s)$ and $g\in C_c(U_t)$.
    \item[(3)] $\varphi_{s^*}(f^*)=\varphi_s(f)^*$ for any $s\in S$ and $f\in C_c(U_s)$.
  \end{enumerate}
  Then there exists a unique $*$-homomorphism $\varphi\colon C^*(G)\to A$ extending each $\varphi_s$.
\end{prop}
\begin{proof}
  Recall that $\mathcal{C}(G)=\Span\bigcup_{s\in S}C_c(U_s)$ by (i) and Lemma \ref{lem:partition}. Together with the universality of $C^*(G)$, it suffices to show that $\varphi\colon\mathcal{C}(G)\to A$ defined by
  \[
    \varphi(f)=\varphi_s(f),\quad f\in C_c(U_s)
  \]
  is a well-defined $*$-homomorphism. From (2) and (3), it is easy to verify that $\varphi$ is a $*$-homomorphism once well-definedness has been checked.  
  To show the well-definedness, we need some sublemmas.

  \vspace{8pt}\noindent\textbf{Sublemma 1}\quad $\varphi_t|_{C_c(U_s)}=\varphi_s$ for $s,t\in S$ with $s\le t$ in $S$.
  \begin{proof}
    Fix $f\in C_c(U_s)\subset C_c(U_t)$. Since $s\le t$ in $S$, we have $s^*s=s^*t=t^*s\le t^*t$. 
    
    Compute that
    \begin{align*}
      (\varphi_t(f)-\varphi_s(f))^*(\varphi_t(f)-\varphi_s(f))
      &=\varphi_{t^*t}(f^**f)-\varphi_{t^*s}(f^**f)\\&\qquad-\varphi_{s^*t}(f^**f)+\varphi_{s^*s}(f^**f)\\
      &=\varphi_{t^*t}(f^**f)-\varphi_{s^*s}(f^**f)\\
      &=0,
    \end{align*}
    where we used (2), (3) in the first equality and (1) in the last equality. This proves $\varphi_t(f)=\varphi_s(f)$.
  \end{proof}

  \vspace{8pt}\noindent\textbf{Sublemma 2}\quad $\|\varphi_s(f)\|\le \|f\|_\infty$ for any $s\in S$ and $f\in C_c(U_s)$.
  \begin{proof}
    Observe first that if $e\in E(S)$, then $\varphi_e\colon C_c(U_e)\to A$ is a $*$-homomorphism by the assumptions (2) and (3). Hence, in this case, we obtain $\|\varphi_e(g)\|\le\|g\|_\infty$ for all $g\in C_c(U_e)$ by \cite[Proposition 3.14]{exel2008}.
    
    Using this observation, we calculate for $f\in C_c(U_s)$ that
    \[
      \|\varphi_s(f)\|^2
      =\|\varphi_s(f)^*\varphi_s(f)\|
      =\|\varphi_{s^*s}(f^**f)\|
      \le \|f^**f\|_\infty
      =\|f\|^2_\infty.\qedhere
    \]
  \end{proof}

  It follows that $\varphi_s$ uniquely extends to a bounded linear map on $C_0(U_s)$, which we also denote by $\varphi_s$. In this proof, we may view $C_0(U_s)\subset\mathbb{C}^G$. From Sublemma 1, we have $\varphi_t|_{C_0(U_s)}=\varphi_s$ for $s,t\in S$ with $s\le t$ in $S$.

  \vspace{8pt}\noindent\textbf{Sublemma 3}\quad  $\varphi_s(f)=\varphi_t(f)$ for any $s,t\in S$ and $f\in C_0(U_s\cap U_t)$.
  \begin{proof}
    First note that\footnote{Notice that $C_c(U_s\cap U_t)\supset C_c(U_s)\cap C_c(U_t)$ \textbf{does not} hold in general. This is the reason that we extend the domain of $\varphi_s$ to $C_0(U_s)$.} $C_0(U_s\cap U_t)=C_0(U_s)\cap C_0(U_t)$ in $\mathbb{C}^G$.

    Fix $\varepsilon>0$ and find $g\in C_c(U_s\cap U_t)$ with $\|f-g\|_\infty<\varepsilon/2$. Recall from (ii) that
    \[
      U_s\cap U_t=\bigcup_{w\le s,t}U_w.
    \]
    Hence, by Lemma \ref{lem:partition}, there exist $m\ge 1$, $w_i\in S$ and $g_i\in C_c(U_{w_i})$ for each $i=1,\ldots,m$, such that
    \[
      g=g_1+\cdots+g_m,\quad w_i\le s,t\quad(i=1,\ldots,m).
    \]
    From $\varphi_s|_{C_0(U_{w_i})}=\varphi_{w_i}=\varphi_t|_{C_0(U_{w_i})}$, it follows
    \[
      \varphi_s(g)=\sum_{i=1}^m\varphi_s(g_i)=\sum_{i=1}^m\varphi_{w_i}(g_i)=\sum_{i=1}^m\varphi_t(g_i)=\varphi_t(g).
    \]
    Using Sublemma 2, we evaluate that
    \[
      \|\varphi_s(f)-\varphi_t(f)\|\le \|\varphi_s(f)-\varphi_s(g)\|+\|\varphi_t(g)-\varphi_t(f)\|<\varepsilon.
    \]
    Since $\varepsilon>0$ is arbitrary, we are done.
  \end{proof}

  We now prove that $\sum_{i=1}^n\varphi_{s_i}(f_i)=0$ in $A$ for any $n\ge 1$, $s_1,\ldots,s_n\in S$ and $f_i\in C_0(U_{s_i})\;(i=1,\ldots,n)$ with $\sum_{i=1}^nf_i=0$ in $\mathbb{C}^G$. We show this by induction on $n\ge 1$.

  When $n=1$, the claim is trivially true.
  
  Suppose the claim is true for $n=k\ge 1$. We now consider the case where $n=k+1$. Then $f_{k+1}=-\sum_{i=1}^kf_i$ in $\mathbb{C}^G$. Thus $f_{k+1}$ vanishes on
  \[
    U_{s_{k+1}}^c\cup\Bigl(\bigcap_{i=1}^kU_{s_i}^c\Bigr)=\Bigl(U_{s_{k+1}}\cap \Bigl(\bigcup_{i=1}^kU_{s_i}\Bigr) \Bigr)^c,
  \]
  that is, $f_{k+1}$ belongs to 
  \[
    C_0\Bigl(U_{s_{k+1}}\cap \Bigl(\bigcup_{i=1}^k U_{s_i}\Bigr)\Bigr)=\sum_{i=1}^kC_0(U_{s_{k+1}}\cap U_{s_i})\;\subset\; \mathbb{C}^G.
  \]
  Write $f_{k+1}=\sum_{i=1}^kg_i$ in $C_0(U_{s_{k+1}})$, where $g_i\in C_0(U_{s_{k+1}}\cap U_{s_i})$ for each $i=1,\ldots,k$. Then
  \[
    \sum_{i=1}^k(g_i+f_i)=f_{k+1}-f_{k+1}=0\quad\mathrm{in}\quad\mathbb{C}^G,\quad\mathrm{and}\quad g_i+f_i\in C_0(U_{s_i})\quad(i=1,\ldots,k).
  \]
  Hence, by assumption, we obtain $\sum_{i=1}^k\varphi_{s_i}(f_i+g_i)=0$ in $A$, or equivalently, $\sum_{i=1}^k\varphi_{s_i}(f_i)=-\sum_{i=1}^k\varphi_{s_i}(g_i)$ in $A$. This implies that
  \begin{align*}
    \sum_{i=1}^{k+1}\varphi_{s_i}(f_i)
    &=\varphi_{s_{k+1}}(f_{k+1})-\sum_{i=1}^k\varphi_{s_i}(g_i)\\
    &=\varphi_{s_{k+1}}(f_{k+1})-\sum_{i=1}^k\varphi_{s_{k+1}}(g_i)\\
    &=0,
  \end{align*}
  where we used $g_i\in C_0(U_{s_{k+1}}\cap U_{s_i})$ and Sublemma 3 in the second equality.
\end{proof}

From the construction of the transformation groupoid, one can show that the correspondence
\[
  S\ni s\mapsto \Theta_s\in \bis(S\ltimes X)
\]
is a semigroup homomorphism which satisfies the conditions (i) and (ii) of Proposition \ref{prop:technical-lemma}.

\begin{lem}
  There exists a $*$-homomorphism $\Psi\colon C^*(S\ltimes X)\to S\ltimes C_0(X)$ such that
  \[
    \Psi(f)=[\delta_s\,f\circ d_s^{-1}]
  \]
  for all $f\in C_c(\Theta_s)$.
\end{lem}
\begin{proof}
  For each $s\in S$, define $\Psi_s\colon C_c(\Theta_s)\to S\ltimes C_0(X)$ by
  \[
    \Psi_s(f)=[\delta_s\,f\circ d_s^{-1}].
  \]
  It suffices to check the conditions (1), (2) and (3) of Proposition \ref{prop:technical-lemma}. 

  For (1), suppose $e_1\le e_2$ and $f\in C_c(\Theta_{e_1})$. Then the homeomorphism $d_{e_2}\colon \Theta_{e_2}\to X_{e_2}$ is an extension of $d_{e_1}\colon \Theta_{e_1}\to X_{e_1}$. Hence $f\circ d_{e_2}^{-1}=f\circ d_{e_1}^{-1}$ in $C_c(X_{e_1})$. It follows
  \[
    \Psi_{e_2}(f)=[\delta_{e_2}\,f\circ d_{e_2}^{-1}]=[\delta_{e_1}\,f\circ d_{e_1}^{-1}]=\Psi_{e_1}(f).
  \]
  This proves $\Psi_{e_2}|_{C_c(\Theta_{e_1})}=\Psi_{e_1}$.

  For (2), fix $f\in C_c(\Theta_s)$ and $g\in C_c(\Theta_t)$. First compute that
  \begin{align*}
    \Psi_s(f)\Psi_t(g)
    &=[\delta_s\,f\circ d_s^{-1}][\delta_t\,g\circ d_t^{-1}]\\
    &=[\delta_{st}\,\alpha_{t^*}((f\circ d_s^{-1})\cdot\alpha_t(g\circ d_t^{-1}))]\\
    \Psi_{st}(f*g)&=[\delta_{st}\,(f*g)\circ d_{st} ^{-1}].
  \end{align*}
  Now we show that
  \[
    \alpha_{t^*}((f\circ d_s^{-1})\cdot\alpha_t(g\circ d_t^{-1}))=(f*g)\circ d_{st}^{-1}
  \]
  in $C_c(X_{(st)^*(st)})$. Fix $x\in X_{(st)^*(st)}=\theta_t^{-1}(X_{s^*s}\cap X_{tt^*})$. Compute that
  \begin{align*}
    \alpha_{t^*}\bigl((f\circ d_s^{-1})\cdot\alpha_t(g\circ d_t^{-1})\bigr)(x)
    &=(f\circ d_s^{-1})(\theta_t(x))\cdot\alpha_t(g\circ d_t^{-1})(\theta_t(x))\\
    &=f([s,\theta_t(x)])g([t,x])
  \end{align*}
  and 
  \begin{align*}
    ((f*g)\circ d_{st}^{-1})(x)
    &=(f*g)([st,x])\\
    &=f([s,\theta_t(x)])g([t,x]),
  \end{align*}
  where we used $[st,x]=[s,\theta_t(x)][t,x]$ in the second equality. This proves (2).

  For (3), fix $s\in S$ and $f\in C_c(\Theta_s)$. Compute that
  \begin{align*}
    \Psi_{s}(f)^*
    &=[\delta_s\,f\circ d_s^{-1}]^*\\
    &=[\delta_{s^*}\,\alpha_s(\overline{f\circ d_s^{-1}})]\\
    \Psi_{s^*}(f^*)
    &=[\delta_{s^*}\,f^*\circ d_{s^*}^{-1}].
  \end{align*}
  Hence it suffices to show that $\alpha_s(\overline{f\circ d_s^{-1}})=f^*\circ d_{s^*}^{-1}$ in $C_c(X_{ss^*})$. To see this, fix $x\in X_{ss^*}$ and verify that
  \begin{align*}
    \alpha_s(\overline{f\circ d_s^{-1}})(x)
    &=\overline{f\circ d_s^{-1}(\theta_{s^*}(x))}\\
    &=\overline{f([s,\theta_{s^*}(x)])}\\
    (f^*\circ(d_{s^*})^{-1})(x)
    &=f^*([s^*,x])\\
    &=\overline{f([s^*,x]^{-1})}\\
    &=\overline{f([s,\theta_{s^*}(x)])}.
  \end{align*}
  This completes the proof.
\end{proof}

Now we obtain the full version of the isomorphisms in Theorem \ref{thm:isom}.

\begin{thm}
  Let $S$ be an inverse semigroup acting on a locally compact Hausdorff space $X$. Then the $*$-homomorphisms $\Phi\colon S\ltimes C_0(X)\to C^*(S\ltimes X)$ and $\Psi\colon C^*(S\ltimes X)\to S\ltimes C_0(X)$ are inverses of each other. In particular,
  \[
    C^*(S\ltimes X)\cong S\ltimes C_0(X).
  \]
\end{thm}
\begin{proof}
  For $s\in S$ and $f\in C_c(X_{s^*s})$, compute that
  \[
    \Psi\circ\Phi([\delta_s\,f])=\Psi(f\circ d_s)=[\delta_s\,f].
  \]
  Hence $\Psi\circ\Phi=\mathrm{id}_{S\ltimes C_0(X)}$.
  
  On the other hand, for $g\in C_c(\Theta_s)$, we calculate that
  \[
    \Phi\circ\Psi(g)=\Phi([\delta_s\,g\circ d_s^{-1}])=g.
  \]
  Since $\mathcal{C}(S\ltimes X)=\Span\bigcup_{s\in S} C_c(\Theta_s)$ by Lemma \ref{lem:partition}, it follows $\Phi\circ\Psi=\mathrm{id}_{C^*(S\ltimes X)}$.
\end{proof}

Recall that every point $x\in X$ corresponds to the pure state $\ev_x$ on $C_0(X)$ and every pure state on $C_0(X)$ is of this form. Observe that $e\in E(S)$ belongs to $\supp(\ev_x)$ if and only if $x\in X_e$. In addition, we have
\begin{align*}
  \overline{\ev_x}(\delta_s\,f)=
  \begin{cases}
    f(x) & \mathrm{if}\;x\in X_e\;\mathrm{for\;some\;}e\in E(S)\;\mathrm{with}\;e\le s\\
    0 & \mathrm{otherwise}.
  \end{cases}
\end{align*}
Here, we note that the condition $x\in X_e$ for some $e\in E(S)$ with $e\le s$ is equivalent to $x\in [s,X_{s^*s}]=\Theta_s$ in $S\ltimes X$.

\begin{lem}
  For any $x\in X$ and $T\in S\alg C_c(X)$, we have $\|\pi_{\overline{\ev_x}}(T)\|=\|\lambda_x\circ\Phi_0(T)\|$.
\end{lem}
\begin{proof}
  First note that, since the image of $S\alg C_c(X)$ is dense in that of $S\alg C_0(X)$ under any $*$-representation $\pi$ of $S\alg C_0(X)$, 
  it follows from Proposition \ref{prop:cyclic} that 
  $\xi_{\overline{\eval_x}}\in\mathcal{H}_{\overline{\eval_x}}$
   is a cyclic vector for $\pi_{\overline{\eval_x}}|_{S\alg C_c(X)}\colon S\alg C_c(X)\to\mathbb{B}(\mathcal{H}_{\overline{\eval_x}})$ satisfying 
  \[
    \overline{\eval_x}(T)=\langle \xi_{\overline{\eval_x}}| \pi_{\overline{\eval_x}}(T)\xi_{\overline{\eval_x}}\rangle
  \]
  for all $T\in S\alg C_c(X)$.
  We next claim that $\lambda_x\circ\Phi_0\colon S\alg C_c(X)\to\mathbb{B}(l^2((S\ltimes X)_x))$ has a cyclic vector $\xi_x$ satisfying
  \[
    \overline{\ev_x}(T)=\langle\xi_x|\lambda_x\circ\Phi_0(T)\xi_x\rangle
  \]
  for all $T\in S\alg C_c(X)$. From the uniqueness of cyclic representation (see VI.19.8 of \cite{fell-doran}), this implies that the two representations $\pi_{\overline{\ev_x}}|_{S\alg C_c(X)}\colon S\alg C_c(X)\to\mathbb{B}(\mathcal{H}_{\overline{\ev_x}})$ and $\lambda_x\circ\Phi_0\colon S\alg C_c(X)\to \mathbb{B}(l^2((S\ltimes X)_x))$ are unitarily equivalent.

  To check the claim, put $\xi_x=\delta_x\in l^2((S\ltimes X)_x)$. Then, for any $g\in\mathcal{C}(S\ltimes X)$, we have
  \[
    \lambda_x(g)\xi_x=\sum_{[s,x]\in (S\ltimes X)_x}g([s,x])\delta_{[s,x]}.
  \]
  Fix $[s,x]\in (S\ltimes X)_x$ and find $f\in C_c(X_{s^*s})$ with $f(x)=1$. A direct computation gives
  \[
    \lambda_x(\Phi_0(\delta_s\,f))\xi_x=\lambda_x(f\circ d_s)\xi_x=\delta_{[s,x]}.
  \] 
  Since $[s,x]\in (S\ltimes X)_x$ is arbitrary, it follows that $\xi_x$ is a cyclic vector for the $*$-representation $(\lambda_x\circ\Phi_0,l^2((S\ltimes X)_x))$. For $t\in S$ and $f\in C_c(X_{t^*t})$, compute further that
  \begin{align*}
    \langle\xi_x|(\lambda_x\circ\Phi_0(\delta_t\,f))\xi_x\rangle
    &=\sum_{[s,x]\in (S\ltimes X)_x}f\circ d_t([s,x])\,\langle\delta_x|\delta_{[s,x]}\rangle\\
    &=\begin{cases}
      f(x) & \mathrm{if}\;x\in\Theta_t\\
      0 & \mathrm{otherwise}
    \end{cases}\\
    &=\overline{\ev_x}(\delta_t\,f).
  \end{align*}
  This completes the proof.
\end{proof}

From the construction of reduced crossed products, we obtain the following:

\begin{prop}
  Suppose that an inverse semigroup $S$ acts on a locally compact Hausdorff space $X$. Then the map $\Phi_0\colon S\alg C_c(X)\to \mathcal{C}(S\ltimes X)$ induces an isomorphism
  \[
    S\ltimesr C_0(X)\cong C^*_r(S\ltimes X).
  \]
\end{prop}

\bibliographystyle{amsalpha}
\bibliography{ks_groupoid.bbl}

\providecommand{\bysame}{\leavevmode\hbox to3em{\hrulefill}\thinspace}
\providecommand{\MR}{\relax\ifhmode\unskip\space\fi MR }
\providecommand{\MRhref}[2]{%
  \href{http://www.ams.org/mathscinet-getitem?mr=#1}{#2}
}
\providecommand{\href}[2]{#2}
\begin{thebibliography}{BKM25}

\bibitem[BE12]{be}
Alcides Buss and Ruy Exel, \emph{Fell bundles over inverse semigroups and twisted \'etale groupoids}, J. Operator Theory \textbf{67} (2012), no.~1, 153--205.

\bibitem[BHM18]{BHM2018}
Alcides Buss, Rohit~D. Holkar, and Ralf Meyer, \emph{A universal property for groupoid {$\rm C^*$}-algebras. {I}}, Proc. Lond. Math. Soc. (3) \textbf{117} (2018), no.~2, 345--375.

\bibitem[BKM25]{BKM2025}
Krzysztof Bardadyn, Bartosz Kwa\'sniewski, and Andrew McKee, \emph{Banach algebras associated to twisted \'etale groupoids: inverse semigroup disintegration and representations on {$L^p$}-spaces}, J. Funct. Anal. \textbf{289} (2025), no.~12, Paper No. 111163, 66. \MR{4948043}

\bibitem[BM17]{bm}
Alcides Buss and Ralf Meyer, \emph{Inverse semigroup actions on groupoids}, Rocky Mountain J. Math. \textbf{47} (2017), no.~1, 53--159.

\bibitem[BN17]{bn}
Erik B\'edos and Magnus~D. Norling, \emph{On {F}ell bundles over inverse semigroups and their left regular representations}, New York J. Math. \textbf{23} (2017), 1013--1044. \MR{3690239}

\bibitem[DP85]{dp}
J.~Duncan and A.~L.~T. Paterson, \emph{{$C\sp \ast$}-algebras of inverse semigroups}, Proc. Edinburgh Math. Soc. (2) \textbf{28} (1985), no.~1, 41--58.

\bibitem[EP16]{ep}
Ruy Exel and Enrique Pardo, \emph{The tight groupoid of an inverse semigroup}, Semigroup Forum \textbf{92} (2016), no.~1, 274--303.

\bibitem[ES17]{es}
Ruy Exel and Charles Starling, \emph{Amenable actions of inverse semigroups}, Ergodic Theory Dynam. Systems \textbf{37} (2017), no.~2, 481--489.

\bibitem[Exe08]{exel2008}
Ruy Exel, \emph{Inverse semigroups and combinatorial {$C^\ast$}-algebras}, Bull. Braz. Math. Soc. (N.S.) \textbf{39} (2008), no.~2, 191--313.

\bibitem[Exe11]{exel2011}
\bysame, \emph{Noncommutative {C}artan subalgebras of {$C^*$}-algebras}, New York J. Math. \textbf{17} (2011), 331--382.

\bibitem[FD88]{fell-doran}
J.~M.~G. Fell and R.~S. Doran, \emph{Representations of {$^*$}-algebras, locally compact groups, and {B}anach {$^*$}-algebraic bundles. {V}ol. 1}, Pure and Applied Mathematics, vol. 125, Academic Press, Inc., Boston, MA, 1988, Basic representation theory of groups and algebras.

\bibitem[FKU]{FKU}
Takuto Fujieda, Takeshi Katsura, and Tomoki Uchimura, \emph{A categorical approach to inverse semigroups, \'etale groupoids, and their {C*}-algebras}, Tokyo Journal of Mathematics, to appear, arXiv:2410.20661.

\bibitem[KS02]{ks1}
Mahmood Khoshkam and Georges Skandalis, \emph{Regular representation of groupoid {$C^*$}-algebras and applications to inverse semigroups}, J. Reine Angew. Math. \textbf{546} (2002), 47--72.

\bibitem[KS04]{ks2}
\bysame, \emph{Crossed products of {$C^*$}-algebras by groupoids and inverse semigroups}, J. Operator Theory \textbf{51} (2004), no.~2, 255--279.

\bibitem[Lan95]{lance}
E.~C. Lance, \emph{Hilbert {$C^*$}-modules}, London Mathematical Society Lecture Note Series, vol. 210, Cambridge University Press, Cambridge, 1995, A toolkit for operator algebraists.

\bibitem[Law26]{lawson}
Mark~V. Lawson, \emph{Inverse semigroups---the theory of partial symmetries}, second ed., World Scientific Publishing Co. Pte. Ltd., Hackensack, NJ, [2026] \copyright 2026.

\bibitem[Nes26]{Neshveyev2026}
Sergey Neshveyev, \emph{A disintegration theorem for non-second-countable \'etale groupoids}, 2026, preprint, arXiv:2607.06010.

\bibitem[Pat99]{paterson}
Alan L.~T. Paterson, \emph{Groupoids, inverse semigroups, and their operator algebras}, Progress in Mathematics, vol. 170, Birkh\"auser Boston, Inc., Boston, MA, 1999.

\bibitem[Ren80]{renault}
Jean Renault, \emph{A groupoid approach to {$C\sp{\ast} $}-algebras}, Lecture Notes in Mathematics, vol. 793, Springer, Berlin, 1980.

\bibitem[Sie97]{ns}
N\'andor Sieben, \emph{{$C^\ast$}-crossed products by partial actions and actions of inverse semigroups}, J. Austral. Math. Soc. Ser. A \textbf{63} (1997), no.~1, 32--46.

\bibitem[SSW20]{as}
Aidan Sims, G\'abor Szab\'o, and Dana Williams, \emph{Operator algebras and dynamics: groupoids, crossed products, and {R}okhlin dimension}, Advanced Courses in Mathematics. CRM Barcelona, Birkh\"auser/Springer, Cham, 2020, Lecture notes from the Advanced Course held at Centre de Recerca Matem\`atica (CRM) Barcelona, March 13--17, 2017.

\bibitem[Ste10]{steinberg}
Benjamin Steinberg, \emph{A groupoid approach to discrete inverse semigroup algebras}, Adv. Math. \textbf{223} (2010), no.~2, 689--727.

\end{thebibliography}
\end{document}